\documentclass[11pt, english, reqno]{amsart}

\usepackage[T1]{fontenc}
\usepackage{babel}
\usepackage{mathrsfs}
\usepackage{amsthm}
\makeindex

\usepackage[dvips,top=3.8cm,left=4cm,right=3.8cm,
foot=3.8cm,bottom=4.0cm]{geometry}
\usepackage{amsfonts,amssymb,amsmath,amsthm, booktabs, 
latexsym}
\usepackage{centernot}
\usepackage{tikz-cd}
\usepackage{bbm}

\usepackage{enumerate}
\usepackage{color}

\newtheorem{theorem}{Theorem}[section]
\newtheorem{lemma}[theorem]{Lemma}
\newtheorem{proposition}[theorem]{Proposition}
\newtheorem{corollary}[theorem]{Corollary}
\theoremstyle{definition}
\newtheorem{definition}[theorem]{Definition}
\newtheorem{example}[theorem]{Example}
\theoremstyle{remark}
\newtheorem{remark}[theorem]{Remark}

\numberwithin{equation}{section}

\numberwithin{equation}{section}
\newcommand{\be}{\begin{equation}}
\newcommand{\ee}{\end{equation}}
\newcommand{\ba}{\begin{aligned}}
\newcommand{\ea}{\end{aligned}}

\newcommand{\N}{{\mathbb N}}

\newcommand{\R}{{\mathbb R}}

\newcommand{\h}{{\mathcal H}}
\def\va{\varphi}
\def\la{\lambda}
\def\csi1{\circ\sigma^{-1}}

\def\ol{\overline}
\def\wt{\widetilde}

\def\mc{\mathcal}

\newcommand{\B}{{\mathcal B}}
\newcommand{\Bfin}{\B_{\mathrm{fin}}}
\newcommand{\Dfin}{\mc D_{\mathrm{fin}}}

\begin{document}

\title[Graph Laplace and Markov operators]{Graph Laplace and Markov
 operators on a measure space}

\author{Sergey Bezuglyi}
\address{Department of Mathematics, University of Iowa, Iowa City,
52242 IA, USA}
\email{sergii-bezuglyi@uiowa.edu}
\email{palle-jorgensen@uiowa.edu}

\author{Palle E.T. Jorgensen}

\subjclass[2010]{37B10, 37L30, 47L50, 60J45}


\keywords{Laplace operator, standard measure space, symmetric measure,
 Markov operator, Markov process, 
harmonic function, dissipation space, finite energy space}

\begin{abstract}
The main goal of this paper is to build a measurable analogue to the
 theory of weighted networks on infinite graphs. Our basic setting is an
 infinite $\sigma$-finite measure space $(V, \B, \mu)$ and a symmetric
 measure $\rho$ on $(V\times V, \B\times \B)$ supported by a 
measurable  symmetric subset $E\subset V\times V$. This applies to such
 diverse areas as optimization, graphons (limits of finite graphs), symbolic
  dynamics, measurable equivalence relations,
 to determinantal processes, to jump-processes;
and it extends earlier studies of infinite graphs $G = (V, E)$ which are
 endowed with a symmetric  weight function 
 $c_{xy}$ defined on the set of edges $E$. As in the theory of weighted 
 networks, we consider the Hilbert spaces $L^2(\mu), L^2(c\mu)$ and
 define two other Hilbert spaces,  the dissipation space $Diss$ and  finite
  energy space $\h_E$. 
 Our main results include a number of explicit spectral theoretic and potential
  theoretic theorems that apply to two realizations of Laplace operators,
 and the associated jump-diffusion semigroups,     
 one  in $L^2(\mu)$, and, the second, its counterpart in $\h_E$. We show in  particular that it is the second setting (the energy-Hilbert space and the 
  dissipation Hilbert space) which is needed in a detailed study of transient 
  Markov  processes.  
 
\end{abstract}

\maketitle

\tableofcontents
\section{ Introduction}\label{sect Intro}

\textbf{Motivation.} Recent works on graph Laplacians and Markov 
processes (details and definitions are given  below) for
 networks suggest a duality between the two settings, (a) a discrete 
 Laplacian $\Delta$, and (b) an associated  Markov transition operator $P$.
  This duality 
 approach is used in turn for answering questions in dynamics, such as 
 deciding when a walk is transient or recurrent; identifying classes of harmonic 
 functions, and an harmonic analysis; building path-space models, and 
 formulate an associated theory of martingales and of boundary spaces.

 As is known, this setting is as follows: (a) the graph Laplacians will have 
positive spectrum; and (b) the transition operators (generalized 
Perron-Frobenius operators) will be positive, in that they map positive 
functions to 
positive functions. However the setting of these studies is discrete; as is clear 
for example for graphs and networks. In other words,  we have countable
 discrete sets of 
vertices and edges; and so the relevant Hilbert spaces will be defined from 
counting measures, weighted or not.

Nonetheless, there are many important applications where the framework of 
countable discrete sets of vertices $V$ and edges $E$ is much too
 restrictive. 
 The list of applications is long, both connections to probability, analysis,
 signal processing and more: graphons (limits of finite graphs), determinantal
 processes, machine learning, jump processes, integral operators, harmonic 
 analysis etc.
Certainly there is a rich variety of Markov processes where the natural setting 
for state space is a general measure space. It is our purpose here, in the 
measure theoretic setting, to make precise the duality between the two, 
transition operator and ``graph'' Laplacian.

 Of course for general measure spaces, the word ``graph'' should perhaps be 
 given a different meaning; see below. Starting with a Markov transition 
 operator, in the measure-dynamic setting, what is the dual Laplacian; and
  vice  versa?

In the countable discrete cases from network models, spectral theory and the 
tools of dynamics rely on a certain Hilbert space that measures ``energy''
 and dissipation, but there, one refers to weighted counting measures on the 
 respective sets $V$ and $E$. Our present paper deals with measure
  theoretic 
 dynamics. We answer the following three questions: (i) What are the relevant 
 measures for the general setting; (ii) What are the correct notions of 
 positivity for both operators in the measure theoretic setting;  and (iii) What 
 is then the extended duality between transition operator and Laplacian?
\\

\textbf{Discrete and measurable settings.}
We begin here with precise definitions, and clarifications of the three 
problems. We first point out explicit parallels between the main objects in the
 theory of discrete networks and their  counterparts defined in the 
 measurable framework. More details are given in Section \ref{sect basics}.

In this paper, we focus on the study of a measurable  analogue of countable
weighted networks, which are known also by names electrical or resistance
networks (we will use them as synonyms). 
We recall that $(G, c)$ is called a \textit{weighted
 network} if $G = (V, E)$ is a countable connected locally finite graph with no
 loops,  and $c = c_{xy}$
 is a symmetric function defined on pairs of of connected vertices  (a more
   detailed  definition is given in Section \ref{sect basics}). One can think of a 
  countable
 network as a discrete measure space $(V, m)$ with the counting measure
 $m$. In general, the theory of weighted
 networks is built around two important operators acting on the space of
 functions $f : V \to \R$. They are the Laplace operator $\Delta$ and 
 the Markov operator $P$. 
\be\label{eqIntro 1}
(\Delta f)(x) := \sum_{y \sim x} c_{xy}(f(x) - f(y)), \qquad
P(f)(x) = \sum_{y \sim x} p(x, y)f(y), \ \ x \in V,
\ee
where the transition probabilities are defined by 
$$
p(x,y) = \dfrac{c_{xy}}{c(x)}, \qquad c(x) = \sum_{y\sim x}c_{xy}.
$$
It is customary to study these operators in
the Hilbert spaces naturally related to a network $(V, E, c)$ such as 
$l^2(V), l^2(V, c)$, and the finite energy Hilbert space $\h$. The Laplacian
$\Delta$ generates the operator in the Hilbert spaces $l^2(V), l^2(V,c)$,
and the finite energy space $\h$ which is formed by functions $f : V \to
\R$ such that
$$ 
\| f \|^2_{\h} =  \frac{1}{2}\sum_{x,y: x\sim y} c_{xy} (f(x) - f(y))^2.
$$

Our approach to the construction of a measurable analogue  is based on 
the following setting. An infinite $\sigma$-finite measure space 
$(V, \B, \mu)$ is the underlying space (``vertices''), a symmetric
subset $E \subset  V\times V$ plays the role of ``edges'', and a symmetric 
measure $\rho$ supported by $E$ is an analogue of the function $c_{xy}$. 
Since $\rho$ is a measure in the product space $V\times V$, it can be 
disintegrated with respect to $\mu$, i.e., 
$$
\rho(f) = \int_V \rho_x(f)\; d\mu(x).
$$
It is assumed that $\rho_x(V) =: c(x)$ is finite and locally integrable
on $(V, \B,\mu)$.  This property is analogous to local finiteness of discrete
networks. 

We define  measurable analogues of the Laplacian and Markov operator 
from (\ref{eqIntro 1}) as follows:
\be\label{eqMeas Delta P}
\Delta(f)(x) = \int_V (f(x) - f(y))\; d\rho_x(y), \qquad 
P(f)(x) = \int_V f(y)\; d\rho_x(y),  
\ee
where $c(x) P(x, dy) = d\rho_x(y)$. A function $f$ satisfying the condition
$\Delta f =0$ (or equivalently, $Pf =f$) is called harmonic.

The corresponding Hilbert spaces are $L^2(V, \B, \mu)$, 
$L^2(V, \B, \nu)$ where $d\nu(x) = c(x) d\mu(x)$, and the finite energy
space $\h_E$ with norm defined by
$$
\| f\|_{\h_E}^2 = \iint_{V\times V} (f(x) - f(y))^2\; d\rho(x,y).
$$

These definitions clarify the similarity between spaces and operators 
considered in discrete in measurable cases. More parallels can be found in
Section \ref{sect basics}, see Tables 1 and 2. 
\\

\textbf{Main results and outline of the paper.} 
The first part of the paper, Sections \ref{sect basics} and 
\ref{sect Graph L and M operator}, contains principal definitions and notions
that are constantly used  in the paper. It is important to emphasize that
 we consider only infinite  $\sigma$-finite measures on a standard Borel 
 space $(V, \B)$ because probability measures would correspond to finite
 networks. In Section \ref{sect basics} we recall several crucial results about 
 the Laplacian and Markov operator in the context of weighted networks. 
 The second part of Section \ref{sect basics} 
is mostly devoted to symmetric measures $\rho$ defined on 
 symmetric Borel subsets $E$ of $V\times V$. These measures 
play the central role in our study. We give a few equivalent approaches 
to the definition of symmetric measures including polymorphisms and
symmetric operators. Remark that this concept can be met in various areas 
of mathematics. One of them is the theory of graphons. At the end of this
section we give two extreme models for symmetric measures: the first model
is based on the case when $(\mu\times \mu)(E) > 0$, and the second one
deals with countable Borel equivalence relations $E$, i.e.,  
 $(\mu\times \mu)(E) = 0$. 
 
 Section \ref{sect Graph L and M operator} contains the definitions and 
 properties of the graph Laplace  and Markov operators $\Delta$ and $P$,
 as well as of two  auxiliary operators $R$ and $\wt R$. These operators,
 which are formally defined as integral operators on the space of Borel
  functions,   can be realized as operators acting on Hilbert spaces, 
  $L^2$-spaces and the finite energy space $\h_E$. 

Section \ref{sect markov} deals with a Markov process
generated by a Markov operator $P$. The difference from the classical
approach  to Markov processes is that we have to work with an infinite 
measure space. We focus here on the following topics: spectral properties
of the operator $P$, harmonic functions for $P$, the Markov process
generated by $P$, the path spaces $\Omega$ and $\Omega_x$,
and the corresponding path measures, reversibility of the Markov process. 

In Section \ref{sect diss space}, we define the dissipation space $Diss$
 whose analogue in discrete networks is used for the study of the finite 
 energy space.  The dissipation space $Diss$ is, in fact, represented as an 
 $L^2$-space with infinite measure. It turns out that the 
 spaces we are interested in can be embedded into the dissipation space.
 This fact is extremely useful since it gives the possibility to apply  the 
structure of the dissipation space  to the study of our main objects 
considered now as operators in $Diss$. 

The finite energy space is thoroughly studied in Section \ref{sect energy}.
We first prove a curious result that can be interpreted as connectedness
of a ``graph'' whose ``vertices'' are sets of finite measure in 
$(V, \B, \mu)$. To study the properties of $\h_E$, we utilize an isometric
 embedding of $\h_E$ into $L^2(\rho)$ and contractive embeddings of 
 $\h_E$ into both $L^2(\nu)$ and  $Diss$. A number of 
 results about the norm of various elements of $\h_E$ is proved. We also
 characterize harmonic functions in the Hilbert space $\h_E$, and 
 find conditions for orthogonality of elements of $\h_E$. 
 
Sections \ref{sect Delta in L2} and \ref{sect Laplace in H} are devoted to 
the study of spectral properties of 
the Laplacians $\Delta_2$ and $\Delta_{\h}$, the Laplace operators acting 
in $L^2(\mu)$ and $\h_E$, respectively.  It turns out that $\Delta_2$
is positive definite and self-adjoint unbounded operator in $L^2(\mu)$. On 
the other hand $\Delta_{\h}$ is a symmetric operator that admits many 
self-adjoint extensions. 
 
Our main results can be found in  Theorems \ref{thm R properties}, 
\ref{thm P is s-a},
\ref{thm harmonic}, \ref{thmHarm}, \ref{thm Poisson},  
\ref{thm H is H space}, 
\ref{thm main norm f}, \ref{thm inner product is energy norm}, 
\ref{thm Delta is s. a.}, \ref{thm Delta_h}, 
Propositions  \ref{prop TFAE reverse MP},  
\ref{prop orth decomp in D}, and
Corollaries \ref{cor harmonic}, \ref{cor Riesz}, \ref{cor Royden}.

\section{Basic setting}\label{sect basics}

Our goal is to introduce and study the concepts
which can be viewed as measurable analogues of basic objects from the
 weighted networks theory. For more details regarding  discrete networks 
 and their Laplacians, the reader may consult the following items 
 \cite{LyonsPeres2016, JorgensenSong2013, Chen_et_al2016}
  and  the papers cited there.

  This section contains the main definitions of
notions considered below. We first discuss  the
underlying measure space and symmetric measures. To illustrate  parallels
between discrete and measurable setting, we consider two models for a 
measurable setting.  

\subsection{Discrete case: electrical networks}\label{subsect Networks} 
For the reader's convenience,  we briefly recall several principal facts and 
definitions from the theory of weighted networks. 
 
Let $G = (V,E)$ denote a connected undirected locally finite graph
 with single edges between vertices such that the vertex set $V$ is 
(countably) infinite, and the edge set $E$ has no loops.  The set 
$E(x) := \{ y \in V : y \sim x\}$ of all neighbors of $x$ is finite for any 
vertex $x$. The edge $e\in E$ connecting vertices $x$ and $y$ is denoted 
by $(xy)$. The connectedness of $G$ means that, for any two vertices 
$x, y \in V$, there exists a finite path $\gamma = (x_0,  x_1, ... , x_n)$ 
such that $x_0 = x, x_n = y$ and $(x_ix_{i+1}) \in E$ for all  $i$. 

\begin{definition}\label{def electrical network}  An {\em weighted  network} 
$(G,c)$ is a weighted graph $G$ with a symmetric {\em conductance 
function} $c : V\times V\to [0, \infty)$, i.e., $c_{xy} = c_{yx}$ for any $(x 
y) \in E$. Moreover,  $c_{xy} >0$ if and only if $(xy) \in E$.  
For any $x\in V$, the total conductance at $x$ is defined as
\be\label{eq total cond}
c(x) := \sum_{y \sim x} c_{xy}.
\ee
\end{definition}

Given a weighted network $(G, c) = (V,E,c)$ with  fixed conductance 
function $c$, we associate the following three Hilbert spaces of functions
on $V$:
\be\label{eq l2(V)}
l^2(V) := \{u : V \in \mathbb R :  ||u||^2_{l^2} = \sum_{x\in V} u(x)^2
 < \infty\},
\ee
\be\label{eq l2(V,c)}
l^2(V, c) := \{u : V \in \mathbb R :  ||u||^2_{l^2(V,c)} = \sum_{x\in V} 
c(x)u(x)^2 < \infty\},
\ee
and
$$
\h_E:= \mbox{equivalence\ classes\ of\ functions\ on\ $V$\ such\ that}
$$
\be\label{eq def of norm energy}
||u||^2_{\h_E} = \frac1{2}\sum_{(xy)\in E}c_{xy}(u(x) -u(y))^2 < \infty,
\ee
where $u_1$ and $u_2$ are equivalent if $u_1 - u_2 =\mathrm{constant}$. 
The Hilbert space $\h_E$ is called the {\em finite energy space} 

We note that in this paper we focus on real-valued functions. The case of 
complex-valued functions is considered with obvious changes.

\begin{definition}\label{def of Laplace operator and harmonic}
The {\em Laplacian} on $(G,c)$ is the linear operator $\Delta$ which is 
defined on the linear space of functions $f : V \to \R$ by the formula
\begin{equation}\label{Laplacian formula}
(\Delta f)(x) := \sum_{y \sim x} c_{xy}(f(x) - f(y)).
\end{equation}
A function $f : V \to \R$ is called {\em harmonic} on $(G,c)$ if $\Delta f(x) = 
0$ for every $x\in V$.
\end{definition}

The Laplace operator $\Delta$ can be realized as an operator either in
 $l^2(V)$, or in $l^2(V, c)$, or in the energy space $\h_E$. 
 The corresponding operators, $\Delta_2$, $\Delta_c$, and
 $\Delta_{\h}$ are, in general, unbounded and densely defined. 
 The description  of their domains requires a careful study of \textit{dipoles},
  see details  in \cite{JorgensenPearse2016, JorgensenPearse_2017}.
   We refer to the following literature devoted to the Laplace operator studied 
  in the   context of weighted graphs (electrical networks) theory: 
  \cite{AlpayJorgensen2012, AlpayJorgensenSeagerVolok2013,
  JorgensenPearse2014, Jorgensen_Tian2015, JorgensenTian2015b, 
  JorgensenPedersenTian2016, SasakiSuzuki2017}.

To any conductance function $c$ on a network $G$, we  associate a 
{\em reversible  Markov kernel} $P = (p(x,y))_{x,y \in V}$ with transition
 probabilities defined by  $p(x,y) = \dfrac{c_{xy}}{c(x)}$. Then,  for any 
 $x,y \in V$, 
 \be\label{eq discr reversible}
  p(x, y)c(x) =  p(y, x)c(y),
\ee
 that is  the Markov process defined by $P$ is  {\em reversible}.  
 Define the probability transition operator for  $ f : V \to  \R$ by setting
\be\label{eq P}
P(f)(x) = \sum_{y \sim x} p(x, y)f(y), \ \ x \in V.
\ee
Then $P$ is called the \textit{Markov operator}. 
It is clear that the Laplace operator $\Delta$ can be represented in terms of 
the Markov operator as follows:
$$
\Delta(f) (x) = c(x)( f(x) - P(f)(x))
$$
or $\Delta = c(I - P)$ where $c$ stands for the operator of multiplication by 
$c$.

The operator $P$ defines also a Markov process $(P_n)$ on the probability 
path space $(\Omega_x, \mathbb P_x)$. Here $\Omega_x$ is the set of
infinite paths beginning at $x \in V$, and $\mathbb P_x$ is the probability
measure on $\Omega_x$ determined by transition probabilities $p(x, y)$. 
Let $X_n$ be the sequence of random variables on $\Omega_x$ such 
that $X_n(\omega) = \omega_n$.

A Markov process $(P_n)$ is called \textit{transient} if the function
$$
G(x, y) := \sum_n p_n(x, y)  
$$
is finite for any pair $x, y \in V$ ($p_n(x, y)$ denotes the probability of the
event $\mathbb P_x( X_n = y)$).

We summarize the following results which can be  found, in  particular, in
 \cite{Jorgensen_Pearse2011, Dutkay_Jorgensen2011, 
Jorgensen_Tian2015, JorgensenPearse2016}.

\begin{theorem}\label{thm summary discr}
(1) $\Delta_2$ is essentially self-adjoint, generally unbounded operator
with dense domain  in $l^2(V)$;

(2) $\Delta_{\h}$ is an unbounded, non-negative, closed, and symmetric
 operator with dense 
domain in  $\h_E$;  in general, $\Delta_{\h}$ is not a self-adjoint operator; 

(3) $P$ is bounded and self-adjoint in $l^2(V,c)$ and its spectrum is in
$[-1, 1]$.
\end{theorem}

In order to illustrate the parallels between discrete networks and measurable 
spaces, we give two tables below. They contain  definitions of the main
 objects for countable 
weighted network and its continuous counterpart, the measurable space
$(V \times V, \B\times \B)$ equipped with a symmetric measure $\rho$. 
Table 1 is focused on the comparison of geometrical objects in the two 
cases such. On the other hand, Table 2 is about operators acting in the
corresponding Hilbert spaces. 
More detailed  definitions can be found  in the text below. 

 \textbf{Space and Network.}  In discrete models the set $V$ will typically 
 be a specified set of vertices in a big network; generally countably infinite. In 
the non-discrete, or measurable, case, $V$ will instead be part of a measure
space. In both cases, we will consider edges, and specified conductance 
functions. While the discrete case is better understood because its history, 
and a rich literature, in both pure and applied models, the continuous case 
(i.e., non-discrete) is perhaps less familiar. A common feature for the two is 
their use in the study of reversible Markov processes. While there is already a 
rich literature in the case of discrete networks (see cited references), the 
continuous, or rather, measurable, setting is of more recent vintage. It is the 
focus of our paper. However, a comparison between the two is useful, see 
Table 1. We will study infinite networks, both discrete and measurable, often 
as limits of finite ones. But many measure-space models arise in applications 
which do not make reference to limits of discrete counterparts.
 
\textbf{Symmetry and Conductance}.  In the discrete models, symmetry
 refers to a function defined on the set $E$ of edges. In the special case of 
 electrical networks of resisters, such a function could be a conductance; i.e., 
 the reciprocal of resistance. There, functions on the set $V$ of vertices could 
 be voltage, and functions on the edges current. Computations will then make 
 use of Ohm's law, and Kirchhoof's law. Continuous or measurable models are 
 more subtle; they may arise as limits of discrete ones, for example as 
 graphons, but their study is interesting in its own right. Another instance of 
 discrete vs continuous is classical potential theory: for example, a classical 
 Laplacian is studied in numerical analysis as a limit of discretized Laplacians.
 
  \textbf{ Laplacian, Markov operator and Transition probability.}  In the
   discrete setting, a typical case of interest is that of transition matrices for 
a Markov chain, for example in the study of dynamical systems described by 
 Bratteli diagrams; and in the continuous case, it takes the form of a 
 measurable family of transition probability measures, indexed by points $x$
  in $V$, so that $P(x, \cdot)$ represents transition from $x$. Since our
 dynamical theories are based on a specified graph Laplacians, the 
 corresponding Markov processes will be assumed reversible (defined in the 
 paper).
 
\textbf{ Hilbert spaces}.  Our proofs will rely on the theory of operators in 
  Hilbert space, and their corresponding spectral theory, but each of the 
  operators under consideration entails its own Hilbert space. A given operator 
  may be selfadjoint in one but not in another. As a result, we must introduce 
  several weighed $l^2$ spaces (and $L^2$ spaces in the measure space 
  case). Our study of boundary theory and of stochastic completeness entails 
  the notion of energy Hilbert spaces, and dissipation Hilbert spaces, and each 
  playing a crucial role in both the discrete and the continuous/measurable 
  models.

 In our outline above we briefly sketched and discussed some main themes, 
 as they arise in both discrete settings, as well as in their measurable 
 counterparts; the focus of our paper. We should stress that, especially for 
 the discrete models, the existing literature is quite extensive. Below we cite a 
 sample, but the reader will be able to supplement with papers cited there: 
 \cite{BezuglyiJorgensen2015, BezuglyiJorgensen2017, Cho2014, 
 DutkayJorgensen2006, Dutkay_Jorgensen2011, Jorgensen2012, 
 Jorgensen_Pearse2011,  JorgensenPearse2013, 
 JorgensenPearse2014, JorgensenPearse2016, JorgensenPearse2016, 
 JorgensenPearse_2017, JorgensenTian2015b, Jorgensen_Tian2015,  
 LyonsPeres2016, SasakiSuzuki2017}. Papers which cover aspects and
  applications in the measurable framework include
   \cite{BezuglyiJorgensen2017,
   GaubertQu2015, JacksonKechrisLouveau2002, JorgensenPaolucci2012, 
 JorgensenPedersenTian2016, Kanovei2008, Kechris2010, Lukashiv2016}.

\begin{table}[h!]
  \centering
  \caption{Comparison of discrete and continuous cases}
  \label{tab:table1}
  {\renewcommand{\arraystretch}{1.2}
  \begin{tabular}{c  c c}
    \toprule
Objects &   Discrete space & Measurable space  \\
    \midrule
    \hline
    \\
   Space & ($V, |\cdot |) $, where $V$ is  vertices  of a connected 
   & $(V, \B, \mu)$ standard $\sigma$-finite\\
   & graph  $G$ and $| \cdot |$ is the counting measure  &  measure space\\
   \\
   \hline
   \\
   Network & $ G = (V, E, c)$ weighted  network & $(V\times V, \B\times \B,
   \rho)$ measure space \\
   \\
   \hline
   \\
  Symmetry & $c: E \to \R$ conductance function  & 
   $\rho = \int_V \rho_x d\mu(x)$ symmetric measure \\    
    & $c_{xy} = c_{yx}$ & on a symmetric set $E \subset V\times V$,\\
    & & $\rho(A\times B) = \rho(B\times A)$\\
    \\
  \hline
  \\
  Conductance & $c(x) = \sum_{y \sim x} c_{xy}$ & 
$  c(x) = \int_V \; d\rho_x = \rho_x(V)$\\
\\
\hline
 \bottomrule
  \end{tabular}}
\end{table}  

\begin{table}[h!]
  \centering
  \caption{Comparison of operators in discrete and continuous cases}
  \label{tab:table2}
  {\renewcommand{\arraystretch}{1.2}
  \begin{tabular}{c  c c}
    \toprule
Objects &   Discrete case & Measurable case  \\
    \midrule
    \hline
    \\
  
  Laplacian & $\Delta(f)(x) = \sum_{y\sim x} c_{xy}(f(x) -f(y))$ &
  $\Delta(f)(x) = \int_V(f(x) - f(y)) \; d\rho_x(y)$\\
  \\
  \hline
  \\
  Markov & $P(f)(x) = \frac{1}{c(x)}\sum_{y \sim x} c_{xy}f(y)$ & 
  $P(f)(x) = \frac{1}{c(x)} \int_V f(y)\; d\rho_x(y)$\\
  operator & &\\
    \hline
    \\
Transition  & $p(x, y) = \frac{1}{c(x)}c_{xy}$ & $P(x, A) = \int_V \chi_A(y)
\frac{1}{c(x)} \; d\rho_x(y)$\\
probabilities  & & \\
\hline
\\
Hilbert spaces & $l^2(V), l^2(V, c)$ & $L^2(\mu), L^2(c\mu), Diss$ \\
\\
\hline
  \\
  Energy & $|| f ||^2_{\mathcal H_E} = $&  $ || f ||^2_{\mathcal H_E} =$ \\
  space  $\mathcal H_E$ &  $ \frac{1}{2} \sum_{x,y}c_{xy} (f(x) - f(y))^2$
    & $\frac{1}{2} \int_{V \times V}  (f(x) - f(y))^2 \; d\rho(x,y)$ \\
  \\
  \hline
\\
Finitely supported  & $\langle\delta_x, f \rangle_{\h_E}= \Delta f(x)$ & $  
\langle \chi_A,  f \rangle_{\h_E} = \int_A \Delta f \; d\mu$\\
functions & & $A\in \Bfin$ \\
\\
\hline  
  
    \bottomrule
  \end{tabular}}
\end{table}

\subsection{From discrete to measurable setting} 
\label{subsect measures} 

We recall that, for every network 
$(V, E, c)$, an atomic measure space $(V, m)$ is given where $m$ is the
counting measure. The conductance function $c$ defines another atomic
measure $\rho$ on $E\subset V\times V$ by setting $\rho(x,y) = c_{xy}$.
In what follows, we define, in terms of measure spaces, similar objects  
which can be regarded as analogues to the basic notions for weighted 
networks.
\\

\textbf{Measure space}. 
Let $V$ be a separable completely metrizable topological space (a 
\textit{Polish space}, for short), and let $\mathcal B$  be  the 
\textit{$\sigma$-algebra of Borel subsets} of $V$. Then  
$(V, \mathcal B)$ is called a \emph{standard Borel space}. We recall that all 
uncountable  standard Borel spaces are Borel isomorphic, so that one can use 
any convenient  realization of the space $V$. If $\mu$ is a continuous (i.e.,  
 non-atomic) Borel  measure  on  $(V, \mathcal B)$, then 
 $(V, \mathcal B, \mu)$ is called a \emph{standard measure space}. We use
 this name for both finite and $\sigma$-finite measure spaces.
Also the same notation, $\B$ is applied for the $\sigma$-algebras of 
Borel sets and
 measurable  sets of a standard measure space.  In the context of
 measure spaces,  we always assume that  $\B$ is  \textit{complete}  
 with respect  to the measure $\mu$. By $\mathcal F(V, \B)$ we denote 
 the space of real-valued Borel functions on $(V, \B)$. For $f \in 
 \mathcal F(V, \B)$ and a Borel measure $\mu$ on $(V, \B)$, we write 
 $$
 \mu(f) = \int_V f\; d\mu.
 $$
As a rule, we will deal only with continuous  $\sigma$-finite 
 measures on $(V, \B)$ (unless the opposite is clearly indicated). This 
 choice of measures is motivated by the discrete case where the counting
 measure plays the role of a $\sigma$-finite Borel measure on a measure
 space. 
 
All objects, considered in the context of measure spaces (such as sets,
functions, transformations, etc), are determined  by modulo sets 
of zero measure (they are also called null sets). In most cases, we will  
implicitly use this mod 0 convention not mentioning the sets of 
zero measure explicitly. 
 
Suppose now that a $\sigma$-finite continuous measure $\mu$ is chosen 
and fixed on $(V, \B)$, so that $(V, \B, \mu)$ is a standard measure space. 
We denote by 
\be\label{eq Bfin}
\Bfin = \B_{\mathrm{fin}}(\mu)  = \{ A \in \B : \mu(A) < \infty\}
\ee
the algebra of Borel sets of finite measure $\mu$. Clearly, the set $V$ can be 
partitioned into a disjoint countable union of sets $A_i$ from $\Bfin$.

We notice that the set 
$\Bfin$ can be used to define a subset of Borel functions
 which is dense in every $L^p(\mu)$-space. For this, we take
\be\label{eq Dfin}
\mathcal D_{\mathrm{fin}} := \left\{ \sum_{i\in I} a_i \chi_{A_i}  : A_i \in 
\Bfin,\ a_i \in \mathbb R,\ | I | <\infty\right\} = \mbox{Span}\{\chi_A : A \in \Bfin\}. 
\ee
\\

\textbf{Symmetric measures.} We first define the notion of a symmetric set.

\begin{definition}\label{def symmetric set}
Let $E$ be an uncountable Borel subset of the direct product 
$(V \times V, \B\times \B)$  such that:

(i)  $(x, y) \in E\  \Longleftrightarrow \ (y, x) \in E$;


(ii) $ E_x := \{y \in V : (x, y) \in E\} \neq \emptyset, \ \   \forall x \in X$;

(iii) for every $x \in V$, 
 $(E_x, \B_x)$ is a standard Borel space where $\B_x$ is 
 the the $\sigma$-algebra of Borel sets induced on $E_x$ from $(V, \B)$.
 
We call  $E$ a \textit{symmetric set}. 
\end{definition}

It follows from (iii) that the projection of $E$ on each of two margins is 
$V$.

We observe that  conditions
(ii) and (iii) are not related to the symmetry property; they are  
included in Definition \ref{def symmetric set} for convenience,
so that we will not have to make additional assumptions. 

The next definition of a symmetric measure is crucial for this paper.

\begin{definition}\label{def symm measure rho}
 Let $(V, \B)$ be a standard Borel space.
We say that a measure $\rho$ on $(V \times V, \B\times B)$ is
 \textit{symmetric} if 
 $$
 \rho(A \times B) = \rho(B\times A), \ \ \ \forall A, B \in \B.
 $$
If $(E_x, \B_x)$ is an uncountable standard Borel space for every $x \in V$, 
then the symmetric measure $\rho$ is  called \textit{irreducible}.
\end{definition}

The meaning of the definition of irreducible symmetric measures is clarified
in Section \ref{sect markov}. Here we see that the projection of the support
of the irreducible measure $\rho$ is the set $V$. 

\begin{lemma}\label{lem rho implies symm E}
If $\rho$ is a symmetric measure on $(V \times V, \B\times \B)$, then
the support of $\rho$, the set $E$, is mod 0 symmetric. 
\end{lemma}

\begin{proof}
The proof is  direct and easy, so that we leave it for the reader.
\end{proof}

The following remark contains two natural properties of symmetric measures
which are implicitly added to Definition \ref{def symm measure rho}.

\begin{remark}
(1) In the paper, we consider the symmetric measures whose supporting sets 
$E$ satisfy Definition \ref{def symmetric set}. In other words, we require
that, for every $x \in V$, the set $E_x\subset E$ is uncountable.  

(2) In general the notion of a symmetric measure is defined in the context 
of standard Borel spaces $(V, \B)$ and $(V\times V, \B\times \B)$. 
But if a $\sigma$-finite measure $\mu$ is given on $(V, \B)$, then we need 
to introduce a relation between $\rho$ and $\mu$. 
Let $\pi_1 : V \times V \to V$ be the projection on the first coordinate. 
We require that the symmetric measure must satisfy
the property $\rho\circ \pi_1^{-1} \ll \mu$. The 
meaning of this assumption is clarified in Theorem \ref{thm Simmons} below.
\end{remark}
\medskip

\textit{\textbf{Assumption A}.} 
Let $(V\times V, \B\times \B)$ be  a $\sigma$-finite
measure space. In this paper we will assume that the 
symmetric  measure $\rho$ is irreducible, i.e., it 
satisfies also the following properties: (i) $E_x \neq \emptyset$, and (ii) 
$(E_x, \rho_x)$ is a  standard measure space for $\mu$-a.e. $x \in V$.  
That is the projection of $E$ onto $V$ is a measurable set of full measure
$\mu$.

\textbf{Measure disintegration.} 
Every Borel set $E$ in the product space $V \times V$ can be partitioned 
into ``vertical'' (or ``horizontal'') fibers. These partitions give examples 
of the so called \textit{measurable partitions}. 
The theory of measurable partitions, developed by Rohlin in 
\cite{Rohlin1949}, is a useful tool for the study of measures on standard
Borel spaces. The case of probability
measures was studied, in general, in  \cite{Rohlin1949}.  It was proved 
that any probability measure admits a unique disintegration with respect to a 
measurable partition.  
For $\sigma$-finite measures, there are similar results establishing 
the existence of such a disintegration. We refer here to  a theorem proved in 
\cite{Simmons2012}. This theorem is  formulated below in the form which 
is adapted to our purposes.

Denote by $\pi_1$ and $\pi_2$ the projections from $V \times V$ onto the
first and second factor, respectively.  Then  $\{\pi_1^{-1}(x) : x \in V\}$
and $\{\pi_2^{-1}(y) : y \in V\}$ are  the measurable partitions of 
$V \times V$ into vertical and horizontal fibers. 

\begin{theorem}[\cite{Simmons2012}] \label{thm Simmons} 
For a $\sigma$-finite measure space $(V, \B, \mu)$, 
let $\rho$ be a $\sigma$-finite measure on $(V\times V, \B\times \B)$ 
such that $\rho\circ \pi_1^{-1} \ll \mu$. Then there exists a unique system
of conditional $\sigma$-finite measures $(\wt\rho_x)$ such that
$$
\rho(f)  = \int_V \wt\rho_x(f)\; d\mu(x), \ \ \ f\in \mathcal F(V\times V, 
\B\times \B).
$$
\end{theorem}

We apply Theorem  \ref{thm Simmons} to a symmetric 
 $\sigma$-finite measure  $\rho$ with support $E$.  Here $E= 
 \mbox{supp}(\rho)$ denotes a 
 subset of $V \times V$ such that $\rho((V \times V) \setminus E) =0$.
 Clearly, this set is defined up to a set of zero measure. 
 
Let $E$ be  partitioned into the fibers $\{x\} \times E_x$. Then, by Theorem
\ref{thm Simmons}, there exists 
a unique  system of  conditional measures $\wt\rho_x$ such that, for any 
  $\rho$-integrable function $f(x, y)$, we have
\be\label{eq disint for rho}
\iint_{V\times V} f(x, y) \; d\rho(x, y)  = \int_V \wt\rho_x(f) \; d\mu(x). 
\ee
It is obvious that, for $\mu$-a.e. $x\in V$, $\mbox{supp}(\wt
\rho_x) = \{x\} \times E_x$ (up to a set of zero measure).   

In the following remark we collect several facts that clarify the essence of 
the defined objects.

\begin{remark}\label{rem symm meas}
(1) We first remark note that formulas involving integrals (see, e.g.,  
 (\ref{eq disint for rho}) and (\ref{eq disint formula}))
are understood in the sense of the extended real line, i.e., the infinite value 
of measurable functions are  allowed. 

 (2) We would like to clarify our notation. It follows from 
Theorem \ref{thm Simmons} that we have a measurable field of
sets $x \mapsto E_x \subset V$ and a measurable field of $\sigma$-finite
Borel  measures $x \mapsto \rho_x $ on $(V, \B)$
 where  the measures $\rho_x$ are defined by the relation
\be\label{eq rho_x def}
\wt\rho_x = \delta_x \times \rho_x. 
\ee
Hence, relation (\ref{eq disint for rho}) can be also written as 
\be\label{eq disint formula}
\iint_{V\times V} f(x, y) \; d\rho(x, y)  = \int_V \left(\int_V f(x, y) \; 
d\rho_x(y)\right)\; d\mu(x).
\ee
In other words, we have a measurable family of measures $(x \mapsto 
\rho_x)$ (which is called a \textit{random measure}),  
and it defines a new measure $\nu$ on $(V, \B)$ by setting
\be\label{eq def of nu}
\nu(A) := \int_V \rho_x(A)\; d\mu(x), \quad A \in \B.
\ee
In contrast to the definition of $\rho$, we consider the measure $\rho_x$ 
to be defined on the subset $E_x$ of$(V, \B)$, $x \in V$.

(3) The symmetry of the set $E$ allows us to define 
a ``mirror'' image of the measure $\rho$. Let $E^y := \{x \in V : (x, y) \in 
E\}$, and let $(\wt\rho^y)$ be the system of conditional measures with 
respect to the partition of $E$ into the sets $E^y \times \{y\}$. It can be 
easily proved, using the symmetry of $\rho$, that, for the measure, 
$$
\rho' = \int_V \wt \rho^y d\mu(y)
$$
the relation $\rho = \rho'$ holds. 

(4) It is worth noting that, in general, the set $E$, the support of a
 symmetric
measure $\rho$, do not need to be a set of positive measure with respect to 
$\mu\times \mu$. In other words, we admit both the  cases: (a) $\rho$ is 
equivalent to the product measure $\mu\times\mu$,  (b) $\rho$ and
$\mu\times\mu$ are mutually singular. 

(5) To simplify our notation, we will always write 
$\int_V f\; d\rho_x$ and $\iint_{V \times V} f\; d\rho$ though the 
measures $\rho_x$ and $\rho$ have the supports $E_x$ and $E$,  
respectively. 
\end{remark}

\textit{\textbf{Assumption B}}.  In general, when a $\sigma$-finite measure 
 $\rho$  is disintegrated, the 
measures $\wt\rho_x$ supported by fibers $\{x\} \times E_x, x \in V$ 
are also $\sigma$-finite. In this paper, we will consider the class of symmetric 
measures for which 
\be\label{eq c finite}
0 < c(x) := \rho_x(V) <\infty
\ee
for $\mu$-a.e. $x\in V$. This assumption is made in accordance with local
finiteness of weighted networks, see (\ref{eq total cond}).
\medskip

As an immediate consequence of  Remark \ref{rem symm meas} (3), 
we have the  following important formula.

\begin{lemma} For a symmetric measure $\rho$, 
\be\label{eq formula fo symm meas}
\iint_{V \times V} f(x, y) \; d\rho(x, y) = \iint_{V \times V} f(y, x) \; 
d\rho(x, y) 
\ee
or
\be\label{eq from symm of rho}
\int_V \int_V f(x, y) \; d\rho_x(y) d\mu(x) = \int_V \int_V f(x, y) \; 
d\rho_y(x) d\mu(y). 
\ee
\end{lemma}

In particular, relation  (\ref{eq formula fo symm meas}) is used to prove the
 equivalence of the measures $\mu$ and $\nu$.

\begin{lemma}\label{lem nu via c(x) and mu} Suppose that $ c(x) =
\rho_x(V)$ is as in (\ref{eq c finite})  for $\mu$-a.e. $x \in V$.
The measure $\nu$ defined in (\ref{eq def of nu}) is equivalent to $\mu$, 
and $d\nu(x) = c(x) d\mu(x)$. 
\end{lemma}

\begin{proof} For any set $A\in \B$, we obtain from (\ref{eq def of nu}) and
(\ref{eq formula fo symm meas}) that
$$
\ba
\nu(A) = & \int_V \rho_x(A) \; d\mu(x)\\
= & \int_V \int_V \chi_A(y) \; d\rho(x, y)\\
= & \int_V \int_V \chi_A(x) \; d\rho(x, y)\\
=& \int_V \chi_A(x) c(x)\; d\mu(x)\\
= & \int_A c(x) \; d\mu(x).\\
\ea
$$
Hence, $\nu$ is equivalent to $\mu$ and $c(x)$ is the Radon-Nikodym 
derivative.
\end{proof}

\textbf{Symmetric measures vs symmetric operators}. Symmetric
measures $\rho$ on $(V\times V, \B\times \B)$ can be described in terms
of  positive linear operators $R$ acting in appropriate functional spaces. 

For a given Borel measure measure $\rho = (\rho_x)$ on the space
$(V\times V, \B\times \B)$, we can define a linear 
operator by setting
\be\label{eq def of R} 
R(f)(x) := \int_V f(y) \; d\rho_x(y) = \rho_x(f), \quad f \in \mathcal 
F(V, \B).
\ee 
Clearly, $R$ is a \textit{positive operator}, i.e., $f\geq 0 
\Longrightarrow R(f) \geq 0$.

\begin{remark}
(1) In this paper we will consider a number of linear operators acting in some
functional spaces. We define them formally as operators on the space
of Borel functions  $\mathcal F(V, \B)$. But our main interest is focused
on their   realizations as  operators acting in $L^p(V, \B, \mu)$-spaces
$p =1,2$ and some other Hilbert spaces (see next sections). 
In particular, we discuss the properties of $R$ (and a more general operator 
$\wt R$) in  Section \ref{sect Graph L and M operator}.

(2) We recall that a Borel measure can be determined as a positive functional 
on a space of functions. 
In particular,  a measure is completely defined by its values on a dense 
subset of functions. In the case of a measure $\rho$  on $V \times V$, 
it suffices to determine $\rho$ on the so-called ``cylinder functions'' 
$(f \otimes g)(x, y) := f(x)g(y)$ (this approach corresponds to the definition 
of a measure on rectangles first).
 \end{remark}

Denote by $\mathbbm 1$ the constant function on $(V, \B, \mu)$ which 
equals 1 at every point $x$.

\begin{proposition}\label{prop equiv of def of rho} Let $(V, \B, \mu)$ be a
$\sigma$-finite standard measure space. Let $\rho = 
\int_V \delta_x \times \rho_x \; d\mu(x)$ be  a continuous Borel measure 
on $(V \times V, \B\times B)$.  The following are equivalent:

(1) $\rho$ is a symmetric measure such that 
$$
0< c(x) = \rho_x (V)  < \infty
$$
for $\mu$-a.e. $x\in V$.

(2) There exists a positive operator $R$ in $\mathcal F(V, \B)$
 such that $R(\mathbbm 1)(x) = c(x)$ and
\be\label{eq symm in terms R}
\int_V f R(g)\; d\mu = \int_V R(f) g\; d\mu,
\ee
for any $f, g \in F(X, \B)$.
\end{proposition}

The operator $R$ satisfying (\ref{eq symm in terms R}) is called 
\textit{symmetric}.

\begin{proof} (1) $\Rightarrow$ (2). This implication is straightforward:
if $\rho$  is a symmetric measure,  define a positive operator $R$ as in
(\ref{eq def of R}). Then, using the disintegration of $\rho$, we have 
$$
\ba
\rho(A \times B) = & \iint_{V\times V} \chi_A(x)\chi_B(y)\; d\rho(x, y)\\
= & \int_V \chi_A(x) R(\chi_B)(x) \;d\mu(x).\\ 
\ea
$$
Similarly, 
$$
\rho(B \times A) = \int_V \chi_B(x) R(\chi_A)(x) \;d\mu(x).
$$
Since $\rho(A \times B) = \rho(B \times A)$, we obtain that  
(\ref{eq symm in terms R}) holds for any simple function, and therefore 
the result follows.

(2) $\Rightarrow$ (1). The above proof can be  used to show that the
converse implication also holds.
\end{proof}

\textbf{Polymorphisms and symmetric measures.}
The approach to symmetric measures, which uses a positive operator $R$
(see Proposition \ref{prop equiv of def of rho}), can be developed in more 
general setting. The key concept here is the notion of a
 \textit{polymorphism} which was defined and studied in a series of papers 
 by A. Vershik. It turns out that the main objects of ergodic theory can be 
 considered in the framework of polymorphisms. We refer to  
 \cite{Vershik2000, Vershik2005} for further details. 
 
 Let $\mu_1$ and $\mu_2$ be Borel measures on a standard 
Borel space $(X, \B)$.  By definition, a \textit{polymorphism}
 $\Pi$ of a standard  Borel space $(X, \B)$ to itself is a diagram consisting 
 of an ordered triple of standard measure spaces:
$$
(X, \B, \mu_1) \stackrel {\pi_1} \longleftarrow \ 
(X \times X, \B\times \B, m)
\stackrel {\pi_2} \longrightarrow \ (X,\B, \mu_2),
$$
where $\pi_1$ and $\pi_2$ are the projections onto the first and second 
component of the product space $(X \times X, \B\times \B, m)$, and 
 $m$ is a measure on  $(X \times X, \B\times \B)$ such
 that $m \circ\pi_i^{-1} = \mu_i, i =1,2$. Remark that this notion is also
 used in the theory of optimal transport, see e.g. \cite{Villani2009}.

This definition can be naturally extended to the case of two distinct
measure spaces $(X_i, \B_i, \mu_i), i =1,2$. Then, the above definition
gives  a polymorphism defined between these measure spaces.

Suppose that $R$ is a positive operator acting on $\mc F(V, \B)$.
Then $R$ defines an action on the space of measures:  if
$\mu$ is a Borel measure on $(V, \B)$, then one defines
\be\label{eq R acts on mu}
 \mu R(f) := \int_V R(f)\; d\mu.
\ee
More details can be found, fore example,  in \cite{BezuglyiJorgensen2017}. 

\begin{definition}\label{def symm poly}
Let $(V, \B, \mu)$ be a  measure space, and let $R$ be a positive operator
 defined on Borel functions $\mathcal F(V, \B)$ such that 
 $$
\int_V f(x) R(g)(x) \; d\mu(x) =  \int_V R(f)(x) g(x) \; d\mu(x)
$$
 for any functions $f, g \in \mathcal F(V, \B)$.
 
 Then the polymorphism 
 $$
\frak R := (V, \B, \mu R) \stackrel {\pi_1} \longleftarrow \ 
(V \times V, \B\times \B, \rho)
\stackrel {\pi_2} \longrightarrow \ (V,\B, \mu R),
$$
is called a \textit{symmetric polymorphism} defined by a positive
operator $R$ and  measure $\mu$ (here $\mu R$ is defined by 
(\ref{def symm poly})). 
\end{definition}

The following result relates the notions of symmetric measures and 
symmetric  polymorphisms.

\begin{proposition} \label{prop on R and rho} 
Suppose that a symmetric measure $\rho$ on 
$(V\times V, \B\times \B)$ satisfies the property
$$
c(x) = \rho_x(\mathbbm 1) \in (0, +\infty) \ \ \mbox{for\ $\mu$-a.e.}\ x 
\in V. 
$$
Then $\rho$ defines a positive symmetric operator $R$ and a polymorphism 
$$
\frak R := (V, \B, \nu) \stackrel {\pi_1} \longleftarrow \ 
(V \times V, \B\times \B, \rho)
\stackrel {\pi_2} \longrightarrow \ (V,\B, \nu)
$$
such that $\nu = \mu R$ and 
\be\label{eq rho via R}
\rho(f\otimes g) = \int_V f R(g)\; d\mu(x).
\ee

Conversely, suppose that a positive operator $R$ is defined on Borel
functions over $(V, \B, \mu)$ and 
$R(\mathbbm 1) = c(x)$. Then relation (\ref{eq rho via R}) defines a measure 
$\rho$ on $(V \times V, \B\times \B)$. The measure $\rho$ is symmetric
if and only if, for any functions $f$ and $g$,
$$
\int_V f R(g)\; d\mu(x) = \int_V R(f) g\; d\mu(x).
$$
\end{proposition}

\begin{proof}
These results have been partially proved in Proposition
 \ref{prop equiv of def of rho}. The statements involving the notion 
 of a polymorphism follow directly from the definitions. We leave 
 the details to the reader.
\end{proof}

\textbf{Question A}. 
Let $\rho$ be a symmetric measure on $(V\times V,
\B\times \B)$. Denote by $\mathcal P(\rho)$ the set of all pairs  $(R, \mu)$,
where $R$ is a positive symmetric  operator on $\mathcal F(V, \B)$ and 
$\mu$ is a measure on $(V, \B)$, such that relation (\ref{eq rho via R})
 holds. Proposition \ref{prop on R and rho} states that the 
 set $\mathcal P(\rho)$ is not empty (provided finiteness of $c(x)$).
  Suppose that $(R, \mu)$ and 
 $(R',\mu')$ are two pairs from the set $\mathcal P(\rho)$, i.e., they define
 the same $\rho$. What relations hold between $(R, \mu)$ and $(R', \mu')$?
 As was shown in Proposition \ref{prop on R and rho}, a pair $(R, \mu)$ must 
 satisfy the conditions: $\mu R = c \mu$ where $c(x) = R(\mathbbm 1)(x)$;  
 and $\rho\circ\pi_1^{-1} = \rho\circ\pi_2^{-1} = c\mu$.
\medskip

\begin{example}[\textbf{Applications}] \label{ex 1}
We give here a few obvious examples of symmetric measures $\rho$ 
(more details can be found in Subsection \ref{subsect two models}).

(1) Let the measure $\nu$ be defined on $(V, \B, \mu))$ by $d\nu(x) = c(x) 
d\mu(x)$. Set $\rho_0(A\times B) := \nu(A \cap B)$ where $A, B \in 
\B$. Then $\rho_0$ is a symmetric measure on $(V\times V, \B\times \B)$
such that $R_0(\mathbbm 1) = c(x)$, where $R_0$ is the corresponding
symmetric operator whose action on functions is given by $R_0(f) =
c(x)f(x)$. 

(2)  Let $c_{xy}$ be a symmetric function defined on a symmetric set 
$E \subset  V\times V$. Consider a measurable field of finite Borel measures 
$x \mapsto \rho^c_x$ where $\rho^c_x$ is supported by $c_{xy}$. Then, 
setting 
$$
\rho^c = \int_V c_{xy}\; d\rho_x^c(y),
$$
we define a symmetric measure $\rho^c$ on the set $E$.

(3) Our approach in the study of symmetric measures, and the corresponding
graph Laplace operators, is close to the basic setting of the theory of 
\textit{graphons} and graphon operators. We refer to several basic works
in this theory \cite{Borgs_2008, Borgs_2012,  Lovasz2012, 
Janson2013, Avella-Medina_2017}. More references can be found in
\cite{Lovasz2012, Janson2013}. Informally speaking, a graphon is the limit 
of a converging sequence of finite graphs with increasing number of vertices.
Formally, a graphon is a symmetric measurable function $W : (\mc X, m) 
\times (\mc X, m) \to [0, 1]$ where $(\mc X, m)$ is a probability measure
 space. The  linear operator $\mathbb W: L^2(\mc X, m) \to  
 L^2(\mc X, m) $ acting by the formula
 $$
 \mathbb W(f)(x) = \int_{\mc X} W(x, y)f(y)\; dm(y)
 $$
is called the\textit{ graphon operator}. The properties of $\mathbb W$ are 
studied in \cite{Avella-Medina_2017}. 

Below in Section 3, we consider a similar operator $\wt R$ defined by a 
symmetric measure $\rho$. The principal difference is that we consider
infinite measure spaces and symmetric functions which are 
not bounded, in general. 

(4) Another application of our results can be used in the theory of 
determinantal measures and determinantal point processes,  see e.g.  
\cite{Lyons2003, Hough_et_al_2009, BufetovQiu2015, 
BorodinOlshanski2017}. For example, the result
of \cite[Proposition 4.1]{Ghosh2015} gives the formula for the norm
in the energy space for a specifically chosen symmetric measure $\rho$. 
To make this statement more precise, we quote loosely the proposition 
proved in \cite{Ghosh2015}: 

Let $\Pi$ be a determinantal point process on a locally compact
space $(X, \mu)$ with positive definite determinantal kernel 
$K(\cdot , \cdot)$ such that 
$K$ is an integral operator on $L^2(\mu)$. Then, for
every compactly supported function $\psi$,
$$
Var\left[\int_x \psi \; d[\Pi] \right] = \iint_{X\times X}
|\psi(x) - \psi(y)|^2 |K(x, y)|^2 \; d\mu(x)d\mu(y). 
$$ 
This formula is exactly the formula for the norm in the energy space when 
the symmetric measure $\rho $ is defined by the symmetric function 
$K(x, y)$: $d\rho(x, y) = |K(x, y)|^2d\mu(x)d\mu(y)$, see Section
\ref{sect energy} below.

(5) Another interesting application of symmetric measures and 
finite energy space is related to Dirichlet forms, see e.g., 
\cite{MaRockner1992, MaRockner1995}. We mention here the 
Beurlng-Deny formula as given in \cite{MaRockner1992}. It states that a
symmetric Dirichlet form on $L^2(U)$, where $U$ is an open subset in 
$\R^d$,  can be uniquely expressed as follows:
$$
\ba
\mc E(u, v)  & = \sum_{i,i=1}^d \int \frac{\partial u}{\partial x_i}
\frac{\partial v}{\partial x_j}\; d\nu_{ij}\\
& \ \ \  \  +  \int_{(U \times U) \setminus \mathrm{diag}}
(u(x) - u(y)) (v(x ) - v(y)) \; J(dx, dy)\\
& \ \ \  \  +  \int uv\; dk.
\ea
$$
Here $u, v \in C^\infty_0(U)$,  $k$ is a positive Radon measure on $U
 \subset \R^d$, and $J$ is a symmetric measure on $(U \times U) 
 \setminus \mathrm{diag}$. The first term on the right hand side in this
  formula is called the diffusion term, the second, the jump term, and the last, 
  the killing term; a terminology deriving from their use in the study of general 
  Levy processes \cite{Applebaum2009}.  
  We see that the second term in this
 formula corresponds to the inner product in the finite energy space $\h_E$
 (details are in Section \ref{sect energy} below).

\end{example}
 
\subsection{Two basic models} \label{subsect two models}

We consider here two models which illustrate the setting described in 
Subsection \ref{subsect measures}. The first model is based on the case
when the support of a symmetric measure $\rho$ is of positive measure
$\mu\times\mu$. The other model deals with a countable Borel 
equivalence relation $\mathcal E$ which supports a singular measure $\rho$ 
with respect to $\mu\times\mu$.

$\mathbf{1^{\mbox{st}}}$ \textbf{case:} $(\mu\times\mu)(E) >0$.
Suppose that $(x, y) \mapsto c_{xy}$ is a positive real-valued Borel
function whose domain is a symmetric Borel set $E\subset V\times V$
of positive measure $\mu\times \mu$.
 Additionally, we require that $c_{xy}$ is \textit{symmetric}, i.e., 
\be\label{eq symmetric c}
c_{xy} = c_{yx},\ \ \ \ \ \forall (x,y) \in E.
\ee
By analogue with the theory of electrical networks, the function $c_{xy}$
is called a \textit{conductance function}.

 Without loss of generality, we can  assume 
 that the function $(x, y) \mapsto c_{xy}$ is 
defined everywhere on $V \times V$ by setting $c_{xy} = 0$ for $(x, y)
 \notin E$.
\medskip

\textbf{\textit{Assumption C:}}
(1) For $\mu$-a.e. $x\in V$, the function $c_x(\cdot) = c_{x, \cdot}$ is 
$\mu$-integrable, i.e.,  
 $$
 c(x) := \int_V c_{xy}\ d\mu(y)
 $$ 
 is positive and finite 
for $\mu$-a.e. x$\in V$. 

(2) We also assume that $c(x)\in L^1_{\mathrm{loc}}(\mu)$, i.e.,
for any $A\in \Bfin$, 
$$
\int_A c(x)\; d\mu(x) <\infty.
$$
\\

These assumptions hold automatically for the discrete case of electrical 
networks. Note that the origin of condition (2) lies in local finiteness of 
graphs used in networks. 

Next, we define (in terms of $\mu$ and $c_{xy}$) two measures, $\nu$ 
and $\rho$, on $V$ and  $E$, respectively. 

\begin{definition}\label{def rho}
Let $\rho$ be a $\sigma$-finite Borel  measure on $E$ such that the 
Radon-Nikodym derivative of $\rho$ with respect to $\mu \times \mu$ is
$c_{xy}$, i.e.,  
\be\label{eq def rho}
\frac{d\rho}{d(\mu\times \mu)}(x, y) = c_{xy}, \quad (x,y) \in E. 
\ee
For every $x \in V$, we define a measure $\rho_x$ on $(V, \B)$ 
by the formula
\be\label{eq def rho_x}
\rho_x (A) = \int_A c_{xy} \; d\mu(y), \ \ A\in \B.
\ee
In other words,  $d\rho_x(y) = c_{xy} d\mu(y)$.  
\end{definition}

Clearly, the measures $\rho$ and $\rho_x$ are uniquely determined by 
$\mu$ and $c_{xy}$.

In the following assertion, we collect the properties of measures $\rho$ and 
$\rho_x$ that follow from the definition.

\begin{lemma}\label{lem prop of rho and rho_x} Suppose that 
$(V, \B, \mu)$ is a standard measure space and $c_{xy}$ is a 
symmetric function on $E \subset V\times V$ as above. Then: 

(1) The supports of  $\rho$  and $\rho_x, x \in V,$ are  the sets 
$E\subset V\times V$, and $E_x\subset V$, respectively. 

(2) The measure $\rho$ can be  disintegrated with respect to the 
``vertical'' and ``horizontal'' partitions 
$E = \bigcup_{x\in V} \{x\} \times E_x $ and $E = \bigcup_{y\in V} 
 E^y \times  \{y\}  $ such that
\be\label{eq disintegration of rho}
\rho = \int_E \delta_x \times \rho_x \; d\mu(x) =  \int_E 
\rho^y \times \delta_y \; d\mu(y) .
\ee

(3) The measure $\rho$ is symmetric:
$$
\rho(A \times B) = \rho(B \times A) , \ \ A, B \in \B, 
$$
or equivalently, 
$$
\int_E f(x, y)\; d\rho(x, y) = \int_E f(y, x)\; d\rho(x, y)
$$
where $f$ is any Borel function on $(V \times V, \B\times \B)$.

(4) For $\mu$-a.e. $x \in V$, 
$$
\rho_x (V) = \rho_x(E_x) = c(x).
$$

(5) The family of measures $(\rho_x)$ determines a positive linear operator 
$R$ 
\be\label{eq def of R via c}
R(f)(x) : = \int_V f(y) \; d\rho_x(y) = \int_V c_{xy} f(y) \; d\mu(y).
\ee
acting on $\mathcal F(X, \B)$ such that, for any Borel functions 
$f$ and $g$, 
$$
\rho(f \otimes g) = \int_V f(x) R(g)(x) \; d\mu(x) = 
 \int_V R(f)(x) g(x) \; d\mu(x). 
$$

\end{lemma}

\begin{remark} As stated in Lemma \ref{lem prop of rho and rho_x}, the
measures $\rho$ and $\rho_x$ are supported by the sets $E$ and $E_x$ 
where the functions $c_{xy}$ and $c_x : y \to c_{xy}$ are 
positive. Hence, we could equally use the formulas
$$
\iint_{V\times V}c_{xy} f(x,y)\; d\mu(x)d\mu(y) = \int_E f(x,y) d\rho(x,y)
$$
and 
$$ 
 \int_{V}c_{xy} f(y)\; d\mu(y) = \int_{E_x}f(y)  d\rho_x(y).
$$
\end{remark}

\textit{Proof  of Lemma \ref{lem prop of rho and rho_x}.} 
The first assertion is obvious due to the definition of the 
conductance function $c_{xy}$.

 To see that (2) holds, we compute for arbitrary functions $f$ and $g$:
$$
\ba 
\rho (f(x) \otimes g(y)) = & \int_E c_{xy} f(x) g(y) \; d\mu(y) d\mu(x)\\
=& \int_V f(x) \rho_x(g) \; d\mu(x)\\
= & \int_V  (\delta_x \times \rho_x)(f \otimes g) \; d\mu(x).
\ea
$$
Similarly, one can show that 
$$
\rho (f(x) \otimes g(y))  = \int_V  (\rho^y \times \delta_y)(f \otimes g) 
\; d\mu(y).
$$
To finish the proof, we note that the space spanned by cylinder functions is
dense in $L^1(\rho)$.

For (3), we find that
$$
\ba 
\rho(A \times B) =& \int_E c_{xy} \chi_A(x) \chi_B(y) \; d\mu(y) d\mu(x)\\
=& \int_V \chi_A(x) \rho_x(\chi_B)\; d\mu(x).
\ea
$$
On the other hand, since $c_{xy} = c_{yx}$, we have
$$
\ba 
\rho(B \times A) =& \int_E c_{yx} \chi_A(y) \chi_B(x) \; d\mu(x) d\mu(y)\\
=& \int_V \chi_A(y) \rho_y(\chi_B)\; d\mu(y),
\ea
$$
and the result follows. 

Statement (4)  of the lemma follows from the disintegration formula used  in
(2), the assumption about $c(x)$, and  from Definition \ref{def rho}.

For (5),  we obtain
$$
\ba
\rho(f\otimes g) = &\int_E f(x) g(y) \; d\rho(x, y)\\
 = & \int_V f(x) \left(\int_V g(y) \; d\rho_x(y) \right)\; d\mu(x) \\
 = & \int_V f R(g) \; d\mu. 
\ea
$$
Clearly, if $f \geq 0$, then $R(f) \geq 0$, i.e., $R$ is a positive operator. 
\hfill$\square$
\\

More generally, we can treat $x \mapsto \rho_x$ as a measurable field 
of measures defined on $(V, \B)$. We see that relation (\ref{eq def rho_x})
 and (\ref{eq disintegration of rho}) define such a 
field. This means that, for 
any $A \in \B$, the function $x \mapsto \rho_x(A)$ is measurable on 
$(V, \B, \mu)$.  In other words,  this field of measures $x\mapsto \rho_x$ 
is called a \textit{random measure} \cite{Kallenberg1983, Aaronson1997}. 

It follows from this observation that we can define a new measure 
$\nu$ on $(V, \B)$ by setting
\be\label{eq def of nu 1}
\nu(A) = \int_V \rho_x(A)\; d\mu(x), \ \ \ \ A \in \B.
\ee
or 
\be\label{eq nu in terms of c}
\nu(A) = \int_V \left( \int_A c_{xy} \; d\mu(y)\right)\; d\mu(x).
\ee

\begin{lemma} The measure $\nu(V)$ is finite if and only if 
$c \in L^1(\mu)$.  The measures $\mu$ and $\nu$ are equivalent and 
$$
\frac{d\nu}{d\mu}(x) = c(x), \qquad x\in V.
$$
\end{lemma}

\begin{proof} 
These assertions follow from (\ref{eq nu in terms of c}):
$$
\ba
\nu(A) & = \int_V \left( \int_A c_{xy} \; d\mu(y)\right)\; d\mu(x)\\
&= \int_A\int_V c_{xy} \; d\mu(x)d\mu(y)\\
& = \int_A c(y) \; d\mu(y).
\ea
$$
\end{proof}

\begin{remark} (1) The operator $R$ is not normalized: $R(\mathbbm 1)(x)
 = c(x)$ where $\mathbbm 1 $ is the constant function equal to 1.

(2) The operator $R$ acts on the space of measures $M(V)$ by the 
following rule:
$$
(\mu R)(f) = \int_V R(f)\; d\mu. 
$$

(3) It follows from (\ref{eq def of nu 1}) and (\ref{eq def of R}) that
$$
\mu R = \nu.
$$
\end{remark}

\textit{Summary.} We summarize here the discussion   in this
 subsection.
We defined the following objects: a standard measure space 
$(V, \B, \mu)$, a symmetric set $E$, and  a symmetric positive function 
$c_{xy} : E \to \mathbb R_+$. Under some natural assumptions about 
$E$ and 
$c_{xy}$, we defined new objects: a symmetric measure $\rho$ with the
system of conditional measures $(\rho_x : x \in V)$, a locally integrable
function $c(x)$, and a positive operator $R$ such that $d(\mu R)(x)
 = c(x)d\mu(x)$. In short notation, we have
 $$
 (\mu, c_{xy}) \ \Longrightarrow \ (\rho, \rho_x, R). 
 $$
The approach  used in Subsection \ref{subsect measures} gives also 
the reverse
implication: given a symmetric measure $\rho $ on $(V \times V, \B \times
\B)$ one defines a positive operator $R$ and the measures $\rho_x$.

 $\mathbf{2^{\mbox{nd}}}$ \textbf{case}: $E$ \textbf{is a
  countable Borel  equivalence relation}. 
  We consider here a particular case when
a symmetric Borel subset $E$ is a \textit{countable Borel equivalence 
relation}. This means that $E$ is a Borel symmetric subset of $V \times V$
 which satisfies the following properties:

(i) $(x, y), (y, z) \in E \Longrightarrow\ (x,z) \in E$;

(ii) $E_x =\{y \in V : (x, y) \in E\}$ is countable for every $x$. 

The concept of a countable Borel equivalence relation 
has been studied extensively last decades in the context of the 
descriptive set theory, see e.g.
 \cite{JacksonKechrisLouveau2002, Kanovei2008, Kechris2010}

Let $| \cdot |$ be the counting measure on every $E_x$. Suppose that 
$c_{xy}$ is a symmetric function on $E$ such that, for every $x \in V$,
$$
c(x) = \sum_{y \in E_x} c_{xy} \in (0, \infty).
$$ 
Then we can define the atomic measure $\rho_x$ on $V$ by setting
$$
\rho_x(A) = \sum_{y \in E_x \cap A} c_{xy}. 
$$
Finally, define the  measure $\rho$ on $E$:
\be\label{eq rho for cber}
\rho = \int_V \delta_x \times \rho_x\; d\mu(x).
\ee

\begin{lemma}
The measure $\rho$ is a symmetric irreducible measure on $E$ which is
 singular with respect to $\mu \times\mu$.
\end{lemma}

\begin{proof} Since $(\mu \times \mu)(E) =0$, the singularity of $\rho$
is obvious. It follows from the symmetry of the function $c_{xy}$ and 
(\ref{eq rho for cber}) that, for $A, B \in \B$,
$$
\ba
\rho(A \times B) = & \int_A \sum_{y \in E_x \cap B} c_{xy}\; d\mu(x)\\
 = & \int_B \sum_{x \in E_y \cap A} c_{xy}\; d\mu(y)\\
 = & \rho(B \times A).
\ea
$$
\end{proof}

\begin{definition}\label{def meas field ER}
Let $E$ be a countable Borel equivalence relation on a standard Borel space
$(V, \B)$. A symmetric subset $G \subset $E is called a \textit{graph} 
if $(x, x) \notin G, \forall x \in V$. A \textit{graphing} of $E$ is a graph $G$
such that the connected components of $G$ are exactly the 
$E$-equivalence classes. In other words, a graph $G$ generates $E$.
\end{definition}

The notion of a graphing is useful for the construction of the path space
$\Omega$ related to a Markov process, see Section \ref{sect markov}. 

The following lemma can be easily proved.

\begin{lemma} Let $\rho$ be a countable equivalence relation on 
$(V, \B)$, and let $\rho$ be a  symmetric measure on $E$. Suppose $G$
is a graphing of $E$. Then $\rho(G) >0$.
\end{lemma} 


For more details regarding integral operators, and analysis of machine 
learning kernels, the reader may consult the following items
 \cite{Atkinson1975, CuckerZhou2007, ChenWheelerKochenderfer2017, 
Ho2017, JorgensenTian2015}  and the papers cited there. 

We refer to the following papers regarding the the theory of  positive definite
 kernels \cite{Aronszajn1950, Adams_et_al1994, PaulsenRaghupathi2016},
  and transfer operators \cite{BezuglyiJorgensen2017, 
 Jorgensen2001, JorgensenTian_2017}. The reader will find more 
 references in the  papers cited there. Various applications of  positive definite
 kernels are discussed in \cite{AlpayJorgensenLevanony2011,
  AlpayJorgensenVolok2014, AplayJorgensen2014, 
 AlpayJorgensenKimsey2015,  AlpayJorgensen2015,
  AlpayJorgensenLewkowicz2015,  AlpayJorgensenLevanony2017}.

\section{Linear operators in functional spaces related to symmetric 
measures}
\label{sect Graph L and M operator}

In this section, we consider several linear operators acting in various 
functional spaces. Our main focus will be on the basic properties of the 
graph Laplace and Markov operators. These definitions and results are  then 
used in the subsequent sections.

\subsection{Definitions of operators $R$, $\wt R$, $\Delta$, and $P$}
\textit{The following objects are fixed in this section:} $(V, \B, \mu)$ is a 
$\sigma$-finite measure space; $\rho$ is a symmetric measure on $(V 
\times V, \B\times \B)$, supported by a symmetric subset $E\subset 
V\times V$; $x \mapsto \rho_x$ is a measurable family of measures on
$(V, \B)$ that disintegrates  $\rho$. 
These objects are used in the text below constantly. 
\medskip

\textit{\textbf{Assumption D}}. We will assume that the symmetric 
measure $\rho$ satisfies the properties:\\
 (a) $c(x) = \rho_x(V) \in (0, \infty)$ for $\mu$-a.e. $x\in V$;\\
 (b) the function $c(x)$ belongs to  $L^1_{\mbox{loc}}(\mu)$, i.e., 
 $\int_A c(x) \; d\mu(x) <\infty,\ \forall A \in \Bfin$.
 \medskip

\begin{remark}
(1) One can easily see that the function $c(x)= \rho_x(V)$ belongs to
 $L^1(V, \mu)$ if and only if $\rho(E) < \infty$. 
 
 (2) We recall 
that the measure $\nu$ on $V$ is defined by $d\nu(x) = c(x)d\mu(x)$. 
It is obvious that  $L^2(\nu) = L^2(\mu)$ if and only if there exist  
$m , M \in (0, \infty)$ such that $m < c(x) < M$ a.e. More general, one can 
observe that 
$$
\nu(A) < \infty \ \Longleftrightarrow \ \int_A c \; d\mu < \infty. 
$$
Therefore,  if $c \in  L^1_{\mbox{loc}}(\mu)$, then $\Bfin(\mu) 
\subset \Bfin(\nu)$. 

(3) In general,   $\nu$ is  an infinite $\sigma$-finite measure, and  $\nu$ 
is finite if and only if $c \in L^1(\mu)$. 
\end{remark}

We now introduce several linear operators. They are defined on the space 
of Borel functions $\mc F(V, \B)$. In fact, a rigorous
definition would require an exact description of  the domains and ranges 
of these operators. We intend to do this when we study realizations of these
operators in some Hilbert spaces.

\begin{definition}\label{def Delta}  Let $(V, \B, \mu)$, $\rho$, and 
$x \mapsto \rho_x$ be as above. The \textit{graph Laplace operator} 
is defined on the space of Borel functions $\mathcal F(V, \B)$ by 
the formula
\be\label{eq def of Delta}
\Delta(f)(x) = \int_V (f(x) - f(y)) \; d\rho_x(y).
\ee
A function $h\in \mc F(V, \B)$ is called \textit{harmonic} for the
 graph Laplace operator $\Delta$ if $\Delta h = 0.$ The set of harmonic 
 functions is denoted $\h arm$. 
 
When the operator $\Delta$ is considered as an operator acting in the 
 space $L^2$, or in the  energy space $\h_E$, then we use the notation
 $\h arm_2$ and $\h arm_{\h}$, respectively,  for the corresponding sets of
 harmonic functions.
\end{definition}

It is obvious that every constant function is harmonic. The problem about the
existence of nontrivial harmonic functions is extremely important. We will 
discuss this problem  in Sections \ref{sect markov} and \ref{sect energy}.

The most important realizations of $\Delta$ are the corresponding linear
operators acting in the Hilbert space $L^2(\mu)$ and the energy space 
$\mathcal H_E$ (see the definition of $\h_E$ below). These 
realizations are discussed in Sections \ref{sect Delta in L2} and 
\ref{sect Laplace in H}.

In Section \ref{sect basics}, we already used the positive operator $R$ 
acting on the space of Borel functions $\mc F(V, \B)$ by the formula:
$$
R(f)(x) = \int_V f(y)\; d\rho_x(y).
$$

In the following sections we will work with a Markov operator $P$ and
a graph Laplacian $\Delta$ which can be defined by means of the operator 
$R$. The definition of $\Delta$ is written in the following form:
$$
\Delta (f)(x) = f(x) \int_V d\rho_x(y) - \int_V f(y)\; d\rho_x(y) 
$$
and
\be\label{eq Delta = c - R}
\Delta (f)(x) = c(x) f(f) - R(f)(x) = (cI   - R)(f)(x).  
\ee

Define now the operator $P$ by setting 
$$
P(f)(x) := \frac{1}{c(x)} R(f)(x).
$$
This definition can be  given in more precise terms as follows:
\be\label{eq def P}
P(f)(x) = \int_V f(y)\; d\ol\rho_x(y)
\ee
where $d\ol\rho_x = c(x)^{-1} d\rho_x(y)$ is the probability measure 
obtained from $\rho_x$ by normalization. In other words, we define the 
measure $\ol \rho$ on $(V\times V, \B\times \B)$ by letting
$d\ol\rho(x, y) = c(x) d\rho(x, y)$. 
Then (\ref{eq Delta = c - R}) is written as
\be\label{eq Delta via P}
\Delta (f) = c(I - P)(f).
\ee
This formula will be constantly used in the next sections. We observe that 
relation (\ref{eq Delta via P}) gives an equivalent definition of 
harmonic functions: a Borel function $h$ is harmonic if $Ph = h$. 
\medskip

Together with $R$, we  consider another linear operator $\wt R$ defined  
on $\mathcal F(V \times V, \B\times\B)$:
$$
\wt R : f \mapsto ( x \mapsto \rho_x(f)),
$$ 
or, equivalently,  
\be\label{eq R on L}
\wt R(f)(x) = \int_V f(x, y) \; d\rho_x(y).
\ee
If $d\rho_x(y) = c_{xy}d\mu (y)$ (that is $\rho$ is equivalent to $\mu
\times \mu$), then 
$$
\wt R(f)(x) = \int_V c_{xy} f(x,y)\; d\mu(y).
$$
Clearly, $\wt R$ is a \textit{positive operator}  in the following sense: if  
$f \geq 0$, then $\wt R(f) \geq 0$.

It is worth noting that  similar operators are considered in various areas, e.g. 
in the theory of graphons \cite{Janson2013, Lovasz2012}. 


\subsection{A few facts about the operators $R$, $\wt R$}

In the following theorem, we collected the  properties of operators 
$R$, $\wt R$  acting in $L^p$-spaces, $p=1, 2$.  
The other two operators, $\Delta$ and $P$, are studied in the next 
sections.

\begin{theorem}\label{thm R properties} 
(1) The operator $\wt R$ maps $L^1(E, \rho)$  onto $L^1(V, \mu)$.  
For any integrable function $f$ on  $(E, \rho)$, the relation
$$
\rho(f) = \mu(\wt R(f)),
$$
holds. $\wt R$ is not one-to-one.

(2) If $g \in L^\infty(\mu)$ and $\pi : (x, y) \mapsto x$ is the projection
from $V\times V$ to $V$,  then 
  $$
  \wt R((g\circ \pi)f) = g\wt R(f)
  $$ 
 for any $f \in L^1(E, \rho)$.
  
(3)  The operator $\wt R : L^2(c \rho) \to L^2(\mu)$ is bounded and
$$
||\wt R||_{L^2(c\rho) \to L^2(\mu)} \leq 1.
$$

(4) If $c\in L^\infty(\mu)$, then $\wt R$ is a bounded operator from 
$L^2(\rho)$ to $L^2(\mu)$ and 
$$
||\wt R||_{L^2(\rho) \to L^2(\mu)} \leq || c||_\infty.
$$  
 
 (5) For $\pi : E \to V: \pi(x,y) = x$, let $U_\pi : L^p(\nu) \to L^p(\rho)$ 
 be the operator acting by the formula $U_\pi : f(x) \mapsto \wt f(x, y)$
  where $\wt f(x, y) = (f\circ \pi)(x, y)$.  
Then $U_\pi $ is an isometric operator for $p =1, 2$.  

(6)  Suppose the function $c(x)$ is such that $L^2(\mu) = L^2(\nu)$. Then 
$L^2(c\rho)   = L^2(\rho)$ and, for any functions $f\in L^2(\mu)$  and 
$g \in L^2(\rho)$,  the relation 
\be\label{eq adjoint for R}
\langle U_\pi f, g \rangle_{L^2(\rho)} = \langle  f,\wt R( g) 
\rangle_{L^2(\mu)}
\ee
holds. In other words,  $\wt R= U_\pi^*$ and the operator $\wt R$ 
is a co-isometry.

(7) Let $\Pi_x : f(x, y) \mapsto f_x(y)$ be the restriction 
of $f(x, y)$ onto $\{x\} \times V$. Then, for any $f \in L^p(\rho),  p =1,2,$
 $$
 \wt R(f)(x) = (R \circ \Pi_x)(f)(x).
 $$
 
(8) Suppose that $c \in L^\infty(\mu)$. Then $R :  L^2(\mu) \to 
L^2(\mu)$ is a bounded operator, and 
$$
||R||_{L^2(\mu) \to L^2(\mu)} \leq ||c||_{\infty}.
$$

(9) Suppose that the function $x \mapsto \rho_x(A) \in L^2(\mu)$ for 
any $A \in \Bfin$.  Then $R$ is a symmetric unbounded operator  in $L^2(\mu)$, i.e.,   
$$
\langle g, R(f) \rangle_{L^2(\mu)} = \langle R(g), f \rangle_{L^2(\mu)}.
$$

(10)  The operator $R : L^1(\nu) \to L^1(\mu)$ is contractive, i.e.,  
$$
||R(f)||_{L^1(\mu)} \leq ||f||_{L^1(\nu)}, \qquad f \in L^1(\nu).
$$
Moreover,   for any function $f \in L^{1} (\nu)$, 
the formula 
\be\label{eq c(x) and rho_x}
\int_V R(f) \; d\mu(x) = \int_V f(x) c(x) \; d\mu(x)
\ee
holds. In other words, $\nu = \mu R$ and 
$$
\frac{d(\mu R)}{d\mu}(x) = c(x).
$$
\end{theorem}

\begin{proof}
We will not prove every statement of this theorem with complete details. 
A part of these  results can be easily deduced from the  definitions given 
above. 

For (1), we compute using (\ref{eq def rho}): 
$$
\begin{aligned}
\rho(f) & = \int_E f(x, y) \; d\rho(x, y)\\
& = \int_V \left( \int_V f(x, y)\; d\rho_x(y)\right) \; d\mu(x)\\
& =  \int_V \wt R(f)(x)\; d\mu(x) \\
&= \mu(\wt R(f)).
\end{aligned}
$$
It follows from the proved relation that the condition  $f \in L^1(E, \rho)$
 implies $\wt R(|f|) \in L^1(V, \mu)$, i.e, $\int_V \wt R(|f|)\; d\mu < \infty$.   
Therefore, $\wt R(f) \in L^1(V, \mu)$ because 
$$
\int_V |\wt R(f)|\; d\mu \leq \int_V \wt R(|f|) \; d\mu. 
$$
To see that $\wt R$ is onto $L^1(\mu)$, it suffices to consider this
operator on characteristic functions over $(V\times V, \B \times \B)$. 
The image of the set of these functions is dense in $L^1(\mu)$. 

On the other hand, $\wt R$ is not one-to-one because the kernel of 
$\wt R$ is not trivial. In particular,one can find 
distinct functions $f_1(x, y)= \va_1(x)\psi_1(y)$ and $f_2(x, y) = 
\va_2(x)\psi_2(y) $ such that  $\wt R(f_1)(x) = \wt R(f_2)(x)$. 

(2) The result follows from the relation 
$$
\begin{aligned}
\wt R((g \circ\pi) f)(x) & = \int_V g(\pi(x,y))f(x, y)\; d\rho_x(y)\\
& = g(x) \int_V f(x, y)\; d\rho_x(y)\\
& = g(x) \wt R(f)(x).
\end{aligned}
$$

(3) Suppose $f \in L^2(c\rho)$. Because of Theorem \ref{thm Simmons},
we obtain that $L^2(c\rho)$ is represented as the direct integral of 
Hilbert spaces $L^2(c\rho_x)$ over the measure space $(V, \B, \mu)$. 
Moreover, $f \in L^2(c\rho)$ if and only if the following conditions hold:

(i) $f(x, \cdot) \in L^2(c\rho_x)$ for $\mu$-a.e. $x \in V$,

(ii) the function $x \mapsto || f(x, \cdot) ||^2_{c(x)\rho_x}$ is in 
$L^1(\mu)$. 

Since $c(x)$ is positive and finite a.e., we conclude that  $f(x, \cdot) \in 
L^2(c\rho_x)$ if and only if $f(x, \cdot) \in L^2(\rho_x)$.

We claim that $\wt R(f) \in L^2(\mu)$. 
 Indeed,  using the Schwarz inequality, we obtain 
$$
\begin{aligned}
||\wt R(f)||^2_{L^2(\mu)} & = \int_V \wt R(f)^2 \; d\mu\\
 & = \int_V \left(\int_V f(x,y)\; d\rho_x(y)\right)^2
d\mu(x)\\
& \leq \int_V \left(\int_V f(x,y)^2 \; d\rho_x(y) \right) \left(\int_V\; 
d\rho_x(y)\right)d\mu(x)\\
& = \int_V \left(\int_V f^2(x, y)c(x)\; d \rho_x(y)\right) \; d\mu(x)\\
& = \int_E f^2(x, y)c(x) \; d \rho(x, y)\\
& = ||f||^2_{L^2(c\rho)}.
\end{aligned}
$$
It shows that $||\wt R||_{L^2(c\rho) \to L^2(\mu)} \leq 1$, and the
 assertion is proved.

(4) This statement can be proved similarly to (3). 

(5) The result follows from the equality
$$
\begin{aligned}
||U_\pi(f)||^2_{L^2(\rho)} & =  \int_E (f\circ\pi)^2(x, y) \; d\rho(x, y) \\
&= \int_V f^2(x) c(x)\; d\mu(x)\\
& = ||f||^2_{L^2(\nu)}
\end{aligned}
$$
because $\rho_x(V) = c(x)$. A similar equality gives the condition 
$$||U_\pi(f)||^2_{L^1(\rho)}  = ||f||^2_{L^1(\nu)}.$$

(6)  It is easy to see that the assumption $L^2(\mu) = L^2(\nu)$ is
 equivalent to the equality  $L^2(c\rho)=  L^2(\rho)$. Then $\wt R$ can
  be viewed as an operator from $L^2(\rho)$ onto $L^2(\mu)$, and 
 we calculate, for $f \in L^2(\mu), g \in L^2(\rho)$, 
$$
\begin{aligned}
\langle  f,\wt R( g) \rangle_{L^2(\mu)} & = \int_V f(x) \wt R(g)(x)\; d\mu(x)\\
& = \int_V f(x)\left(\int_V g(x, y) \;  d\rho_x(y) \right)\; d\mu(x)\\
& = \int_E (f\circ\pi)(x, y) g(x, y)\; d\rho(x, y)\\
&= \langle U_\pi f, g \rangle_{L^2(\rho)}. 
\end{aligned}
$$
This proves that $\wt R = U_\pi^*$, so that $\wt R$ is a co-isometry. 

We note that it follows from this relation that the operator $\wt R$ is onto
 $L^2(\mu)$. Indeed, we use that 
$\mbox{Range} (\wt R) ^{\bot} = \mbox{Ker}(\wt R^*)$. Since
 $\wt R^*$ is one-to-one,  it has the trivial kernel, and  $\wt R$ is onto.

(7) This is obvious. 

(8) In order to prove this assertion, we apply Jensen's inequality 
for the probability measure $d\ol\rho_x = c(x)^{-1}d\rho_x$. Then, 
for any $f \in L^2(\mu)$ we have, 
$$
\begin{aligned}
\int_V [R(f)(x)]^2\; d\mu(x) & = \int_V\left( \int_V  f(y)c(x)\; 
d\ol\rho_x(y)\right)^2\; d\mu(x) \\
& \leq \int_V\left( \int_V  f^2(y)c^2(x)\;  d\ol\rho_x(y)\right)\; d\mu(x)\\
& = \int_V\left( \int_V  f^2(y)c(x)\;  d\rho_x(y)\right)\; d\mu(x)\\
& \leq ||c||_{\infty} \iint_{V\times V}  f^2(y)\;  d\rho(x, y) \qquad \quad
 (\rho \ \mbox{is
symmetric}) \\ 
& = ||c||_{\infty}\iint_{V\times V}  f^2(x)\;  d\rho(x, y)  \\ 
& = ||c||_{\infty} \iint_{V} c(x) f^2(x)\;  d\mu(x)  \qquad \quad (\rho_x(V) 
= c(x))   \\
&\leq ||c||^2_{\infty} \int_V f^2(x)\; d\mu(x)
\end{aligned}
$$
Hence,  $R(f) \in L^2(\mu)$, and the norm of $R$ in $L^2(\mu)$
 is bounded by $||c||_{\infty}$. We note that if $c \notin L^\infty(\mu)$,
 then $R$ is, in general, an unbounded operator.
 
(9) We first observe that the assumption that $\rho_x(A) \in L^2(\mu),
A \in \Bfin, $ means that $R$ is a densely defined operator. Then,  
to show that $R$ is symmetric, we use the fact that the measure 
$\rho$ is symmetric:
$$
\begin{aligned}
\langle g, R(f) \rangle_{L^2(\mu)} & = \int_V g(x) \left( 
\int_V f(y) \; d\rho_x(y) \right)\; d\mu(x)\\
& = \iint_{V\times V} f(y)  g(x) \; d\rho(x, y)\\
& = \int_V f(y) \left( \int_V  g(x) \; d\rho_y(x) \right)\; d\mu(y)\\
& = \int_V f(y) R(g)(y) \; d\mu(y)\\
&= \langle R(g), f \rangle_{L^2(\mu)}.
\end{aligned}
$$ 

(10) We compute 
$$
\ba
||R(f)||_{L^1(\mu)} & =  \int_V \left|\int_V f(y) \; d\rho_x(y)\right|
\; d\mu(x)\\
& \leq \int_V\int_V |f(y)|  \; d\rho_x(y)d\mu(x)\\
& = \int_V |f(y)| c(y)\; d\mu(y)\\
&= || f ||_{L^1(\nu)}.
\ea
$$ 
The other statement in (10) is obvious.
\end{proof}

\section{Markov processes associated with symmetric measures}
\label{sect markov} 

In this section,  we introduce a Markov process related to a symmetric 
measure $\rho$ on $(V\times V, \B\times \B)$ and generated by a Markov 
 operator. 

\subsection{Markov operators} By a \textit{Markov operator}, 
we mean a positive 
self-adjoint operator $P$ in a $L^2$-space satisfying the normalization 
condition $P(\mathbbm 1) = \mathbbm 1$. The book \cite{Revuz1984}
is a remarkable introduction to homogeneous Markov chains with
measurable state space. More information about various aspects of Markov 
chains can be found in the following papers:
\cite{BezuglyiJorgensen2015, DutkayJorgensen2006, DutkayJorgensen2006, 
GaubertQu2015,JorgensenPaolucci2012,  KapicaMorawiec2015,
 Lukashiv2016}.

An example of a Markov operator built by a symmetric measure has been
given in Section \ref{sect Graph L and M operator}. 
We consider here the main properties of the operator $P$ defined above
in (\ref{eq def P}). 

\begin{definition} Let $(V,\B,\mu)$ be a a measure space, 
$E$ a symmetric subset of $V \times V$,  and $\rho$ a
 symmetric measure with support $E$. Let $R$ be  a positive operator 
 defined by $\rho$ as in  (\ref{eq def of R}). We set  
$$
P(f)(x) = \frac{1}{c(x)} R(f)(x)
$$
or
 \be\label{eq formula for P}
 P(f)(x) = \frac{1}{c(x)}  \int_V f(y) \; d\rho_x(y) = \int_V f(y) \; 
 P(x, dy)
 \ee
 where  $P(x, dy)$ is the probability measure obtained by normalization
 of $\rho_x$. 
 \end{definition}
 
\begin{remark} (1) Because every measure $P(x, \cdot)$ is probability, 
the positive operator $P$ is obviously normalized, i.e., $P(\mathbbm 1)
 = \mathbbm 1$. In fact, $P$ is defined on a set of Borel functions over
 $(V, \B)$ (allowing functions with infinite values). 
  Our main interest in the operator $P$ is focused on the properties of
  $P$ as an operator on  the spaces $L^2(\mu)$ and $L^2(\nu)$.

(2) It is worth noting that if a measurable field of probability measures
$x \to \mu(x), x \in V,$ is given on the space $(V, \B)$, then there exists a 
normalized positive operator $P$ (Markov operator) 
defined by $x\to \mu(x)$ similar to 
(\ref{eq formula for P}). But the converse is not true: not every Markov
operator determines a measurable field of probability measures.

  (3) Working with a positive normalized operator $P$, we will often use 
Jensen's    inequality which states that $(P(f))^2 \leq P(f^2)$ for any Borel 
 function  $f$. 

 (4) The notation $P(x, dy)$ for the measure $\ol\rho_x = c(x)^{-1} \rho_x$
  is used in (\ref{eq formula for P}) for consistency with notations common in 
the  literature on  Markov processes. 
 
(5) If $d\rho_x(y) = c_{xy} d\mu(y)$, then the operator $P$ is defined by
its density \be\label{eq P via p(x,y)}
  P(f)(x) = \int_V p(x, y) f(y) \; d\mu(y)
\ee
 where
 $$
 p(x, y) = \frac{c_{xy}}{c(x)}.
 $$
It follows from the definition that $p(x, y) > 0$ for any $(x, y) \in E$, and
$$
\int_V p(x,y) \; d\mu(y) = 1, \ \ \ \forall x \in V. 
$$ 
This simple fact makes clear parallels with Markov processes defined on 
discrete electrical networks. 
\end{remark}
 
In the following result, we  show how the operator $P$ acts on the measures
$\nu$ and $\mu$. 
 
 \begin{lemma}\label{lem nu P = nu}
(1)  Let $d\nu(x) = c(x)d\mu(x)$ where $c(x) = \rho_x(V)$.
 Then $\nu P = \nu$. 

(2) 
$$
\frac{d\mu P}{d\mu}(x) = \int_V \frac{1}{c(y)}\; d\rho_x(y).
$$
 \end{lemma}
\begin{proof}
(1) We use (\ref{eq formula for P}) and the symmetry of $\rho$ 
(see (\ref{eq from symm of rho}))
to compute the Radon-Nikodym derivative $\dfrac{d(\nu P)}{d\nu}$. 
Let $f$ be any Borel  function over $(V, \B)$, then
$$
\ba
\int_V P(f)\; d\nu =   & \int_V \left(\frac{1}{c(x)}  \int_V f(y) \; d
\rho_x(y)\right)c(x)\;  d\mu(x)\\
 = & \int_V f(y) \left( \int_V  \; d\rho_y(x)\right)\; d\mu(y) \qquad
 (\mathrm{by\ symmetry\ of} \ \rho)\\
 = & \int_V f(y) c(y)\; d\mu(y)\\
= & \int_V f\; d\nu.
\ea
$$

(2) To prove the second assertion, we need to find a measurable function
 $g(x)$ such that
$$
\int_V P(f)\; d\mu = \int_V f g \; d\mu.
$$
For this,
$$
\ba
\int_V P(f)\; d\mu =& \int_V \left(\int_V f(y) P(x, dy)\right) \; d\mu(x) \\
= & \iint_{V\times V} \frac{f(y)}{c(x)}\; d\rho_x(y)d\mu(x)\\
= &\int_{V} f(x) \left(\int_V \frac{1}{c(y)}\; d\rho_x(y)\right) \; d\mu(x).\\
\ea
$$
This proves (2).
\end{proof} 

The following theorem contains several basic properties of the operator 
$P$. The most important is its self-adjointess in $L^2(\nu)$.

\begin{theorem} \label{thm P is s-a} Let $d\nu(x) = c(x) d\mu(x)$ be 
the $\sigma$-finite  measure on $(V, \B)$ where $\mu$ and $c$ are  
defined  as above. Suppose $P$ is defined by (\ref{eq formula for P}). Then:

(1) The bounded operator $P: L^2(\nu) \to L^2(\nu)$ is self-adjoint.

(2) The operator $P$ considered in the spaces $L^2(\nu)$ and $L^1(\nu)$ 
is  contractive, i.e., 
$$
|| P(f) ||_{L^2(\nu)} \leq || f ||_{L^2(\nu)}, \qquad || P(f) ||_{L^1(\nu)} \leq 
|| f ||_{L^1(\nu)}.
$$ 

(3) Spectrum of $P$ is a subset of $[-1, 1]$.

(4) Suppose that $(V, \B, \mu)$ is a probability measure space and the 
operator $P$ is defined by (\ref{eq P via p(x,y)}). Then $P$ is contractive in 
 $L^2(\mu)$. 

\end{theorem}

We remark that, in fact, $P$ is also a contraction in the space $L^p (\nu),
1 \leq p \leq \infty$, but we will not use this in the paper.

\begin{proof}
 To see that (1) holds, we use Theorem \ref{thm R properties}  (6) and 
 formula (\ref{eq formula for P}):  for any $f, g \in L^2(\nu)$,
$$
 \ba
\langle P(f), g \rangle_{L^2 (\nu)} & = \langle c^{-1} R(f), 
g  \rangle_{L^2 (\nu)}\\ 
& = \langle R(f),  g\rangle_{L^2 (\mu)}\\
& = \langle f,  R(g)\rangle_{L^2 (\mu)}\\
& = \langle f,  c P(g)\rangle_{L^2 (\mu)}\\
& = \langle f,  P(g)\rangle_{L^2 (\nu)}\\
\ea
$$

The proof of (2) follows from the inequalities $P(f)^2 \leq P(f^2)$ and 
$|P(f)| \leq P(|f|)$ and the following calculation:
$$
\ba 
\int_V P(f)^2(x) c(x)\; d\mu(x) & \leq  \int_V P(f^2)(x)c(x)\; d\mu(x) \\
& = \int_V R(f^2)(x) \; d\mu(x)\\
& = \int_V \int_V f^2(y) \; d\rho_x(y) d\mu(x)\\
& = \iint_{V\times V} f^2(y) \; d\rho(x, y)  \ \ \qquad
\mbox{(by\ symmetry\ of\ $\rho$)} \\
&=  \iint_{V\times V} f^2(x) \; d\rho(x, y)\\
& = \int_V f^2(x) \left(\int_V \; d\rho_x(y)\right)\; d\mu(x)\\
& = \int_V f^2(x) c(x)\; d\mu(x).\\
\ea
$$

Similarly,  
$$ 
\int_V |P(f)(x)| \; d\nu(x) \leq \int_V P(|f|)(x) \; d\nu (x) = \int_V
|f(x)|\; d\nu(x)
$$
since $\nu$ is $P$-invariant.

Assertion (3) is now a direct consequence of the proved statements (1) 
and (2). 

To see that (4) holds, we use the Schwarz inequality and  the fact 
that $0\leq p(x,y) \leq 1$:
$$
\ba
\int_V P(f)^2(x) \; d\mu(x) & = \int_V \left( \int_V p(x,y)f(y)  \; d\mu(y)
\right)^2\; d\mu(x)\\
& \leq \int_V \left( \int_V p(x,y)^2  \; d\mu(y)\right) \left( \int_V f(y)^2  
\; d\mu(y)\right) \; d\mu(x)\\
&  \leq  || f ||^2_{L^2(\mu)} \int_V \left( \int_V p(x,y)  \; d\mu(y)\right)
d\mu(x)\\
& = || f ||^2_{L^2(\mu)}.
\ea
$$
Hence,
$$
|| P(f) ||_{L^2(\mu)} \leq || f ||_{L^2(\mu)}.
$$ 

\end{proof}

It is useful to represent  a symmetric measure $\rho$ via a Markov 
operator $P$.

\begin{lemma}\label{lem rho_n via P}  Let $(V, \B, \mu)$ be a measure
 space  and let $\nu = c\mu$.  
A Borel $\sigma$-finite measure $\rho$ on $(V\times V, \B \times \B)$ is
symmetric if and only if there exists a self-adjoint Markov operator
$P$ on $L^2(\nu)$ such that, for any $A, B \in \Bfin$, 
$$
\rho(A\times B) = \int_V \chi_A P(\chi_B)\; d\nu = \langle \chi_A,
P(\chi_B)\rangle_{L^2(\nu)}.
$$
More generally, such an operator $P$ defines a sequence of symmetric 
measures $(\rho_n)_{n \in \N}$ by the formula
\be\label{eq def rho_n}
\rho_n(A\times B) = \langle \chi_A, P^n(\chi_B)\rangle_{L^2(\nu)}.
\ee
\end{lemma}

\begin{proof} This results follows from Proposition 
\ref{prop equiv of def of rho} and the definition of $P$ and $\nu$. 
\end{proof}

It follows from Theorem \ref{thm P is s-a} that  $P^n$ is self-adjoint 
for every $n$, and therefore $\rho_n$ is a well defined symmetric measure.
One can see that, for any  $n \in \N$,
$$
d\rho_n(x,y) = c(x) P_n(x, dy)d\mu(x)
$$ 
and
$$
\iint_{V\times V} f(x, y) \; d\rho_n(x, y) = \iint_{V\times V} f(x, y) c(x)
P_n(x, dy)\; d\mu(x)  = \int_V P^n(f)(x) \; d\nu(x).
$$

\subsection{Harmonic functions for $P$} 
In this part we will deal with Markov operators $P$ in $L^2(\nu)$ preserving
the measure $\nu$. 

Let $P$ be the Markov operator defined by (\ref{eq formula for P}). In other 
words, this operator is define by an irreducible symmetric measure $\rho$. 
We recall that, as $P$ is a positive operator such that $P(\mathbbm 1) =
1$, then for any function $f$
the inequality $P(f)^2 (x)\leq P(f^2)(x)$ holds a.e. We need a stronger form
of this inequality. 

\begin{lemma}\label{lem strict ineq for P} For the Markov operator
$$ 
P(f)(x) = \int_V f(y)\; P(x, dy)
$$ 
and any non-constant function $f \in L^2(\nu)$, 
there exists a subset $A \subset V$ of positive measure $\nu$ such that 
$P(f)^2(x)  < P(f^2)(x), x \in A$.
\end{lemma}

\begin{proof} The proof follows from Jensen's inequality applied to the 
convex function $\va(x) = x^2$ . As shown in the 
proof of \cite[Theorem 3.3]{Rudin1987}, the equality occurs only for affine
convex functions. 
\end{proof}

\begin{theorem}\label{thm harmonic} Let $(V, \B, \nu)$ be a measure space
with finite or $\sigma$-finite measure. Suppose  $P$ is a Markov operator
on  $L^2(\nu)$ defined by an irreducible symmetric measure according to 
(\ref{eq formula for P}). Then  
$$
 L^2(\nu) \cap \h arm(P) = \begin{cases} 0,  & \nu(V) = \infty\\
\mathbb R\mathbbm 1 , & \nu(V) < \infty
 \end{cases}
$$ 
where $\mathbb R\mathbbm 1$ is the set of constant functions. Moreover,
$1$ does not belong to the point spectrum of the operator $P$ on the 
space $L^2(\nu)$.  
\end{theorem}

\begin{proof} If we show that there is no nontrivial harmonic functions
in $L^2(\nu)$, then we prove that $1$ is not an eigenvalue for $P$. Clearly, 
the converse also holds. 

 Assume for contrary that there exists a non-constant 
function $f \in L^2(\nu)$ such that $P(f) = f$. Then, by Lemmas 
\ref{lem nu P = nu} and \ref{lem strict ineq for P}, we have
$$
\ba
|| f ||^2_{L^2(\nu)} = & \int_V f(x)^2 \; d\nu(x) \\
= & \int_V (P(f))^2(x)  \; d\nu(x) \\
<  & \int_V P(f^2)(x)  \; d\nu(x) \\
 = & \int_V f^2(x)  \; d(\nu P)(x) \\
= & \int_V f^2(x)  \; d\nu(x) \\
=& || f ||^2_{L^2(\nu)}.\\
\ea
$$
This contradiction proves the theorem.
\end{proof}

We recall that if $T$ is a contraction in a Hilbert space $\mc K$, then
$\mc K$ is decomposed into the orthogonal direct sum
\be\label{eq MET}
\mc K = \mathrm{Fix}(T) \oplus \ol{\mathrm{Range}(I - T)}.
\ee
This is a form of the mean ergodic theorem, see e.g., 
\cite[Theorem 8.6]{Eisner2015}. 

 We apply this result to the case of an 
abstract Markov operator $P$ acting in $L^2(\nu)$ such that $\nu P = 
\nu$. We assume here that $P$ is self-adjoint and contractive. 
 Denote by $Cb(P)$ the subset  of 
 $L^2(\nu)$ formed by  \textit{$P$-coboundaries}, i.e., $Cb(P) =
 \{g - P(g) \ | \ g \in   L^2(\nu)\}$. 
  Clearly, for any function $g\in L^2(\nu)$, one has 
$\int_V (g - P(g))\; d\nu = 0$ since $\nu$ is $P$-invariant. Hence
$$
Cb(P) \subset  L^1_0(\nu) \cap L^2(\nu) = 
 \{f \in L^2(\nu) : \int_V f \; d\nu =0\}.
$$

\begin{proposition}\label{prop coboundaries}
 (1) Let $P$ be a self-adjoint contractive operator on 
$L^2(\nu)$  satisfying  $\nu P = \nu$. Then 
$$
L^2(\nu) = \h arm_2(P) \oplus \overline{Cb(P)},
$$
where $\h arm_2(P) = \{f \in L^2(\nu) : P(f) = f\}$ and $\overline{Cb(P)}$
 is the closure of $Cb(P)$ in $L^2(\nu)$.

(2) Suppose that $P$ is a Markov operator defined by an irredusible 
symmetric measure $\rho$ as in Theorem 
\ref{thm harmonic}. Then  the set $\{g - P(g) \ | \ g \in L^2(\nu)\}$ of 
$P$-coboundaries is  dense in $L^2(\nu)$.

(3) The operator $(I - P)^{-1}$ is unbounded in $L^2(\nu)$.
\end{proposition}

\begin{proof} 
(1) Clearly, this statement is a form of (\ref{eq MET}).
 Suppose that $f$ is a function from $L^2(\nu) \ominus L^2_0(\nu)$.
 Then, for arbitrary $g \in L^2(\nu)$, 
$$
\ba
0 =& \langle f, g - P(g)\rangle_{L^2(\nu)}\\
=&   \langle f, g\rangle_{L^2(\nu)} - \langle f, P(g)\rangle_{L^2(\nu)}\\
=&   \langle f, g\rangle_{L^2(\nu)} - \langle P(f), g\rangle_{L^2(\nu)}\\
=& \langle f - P(f), g\rangle_{L^2(\nu)}. \\
\ea
$$
Hence, $f = P(f)$. The same proof shows that if $f \in \h arm_2(P)$, then 
$\langle f, g\rangle_{L^2(\nu)} = \langle P(f), g\rangle_{L^2(\nu)} 
= \langle f, P(g)\rangle_{L^2(\nu)}$ and $f \perp Cb(P)$. 

(2) The result follows from Theorem 
\ref{thm harmonic} and statement (1) of this theorem.

(3) This observation follows from (2). 
\end{proof}

In the next statement we summarize  facts about harmonic
functions in $L^2(\nu)$. 

\begin{theorem}\label{thmHarm}
Let $P$ be  a self-adjoint contractive operator on $L^2(\nu)$ such that 
$\nu P= \nu$. The following are equivalent:

(i) $\lambda =1$ is not an eigenvalue for the operator $P$ in $L^2(\nu)$;

(ii) $\{ P(f) = f\} \cap L^2(\nu) = 0$;

(iii) 
$$
\lim_{N \to \infty} \frac{1}{N}  \sum_{n = 1}^{N} P^n(f) = 0, \quad
f \in L^2(\nu);  
$$

(iv) for any $A, B\in \Bfin$,
$$
\lim_{N \to \infty} \frac{1}{N}  \sum_{n = 1}^{N} \rho_n(A \times B) =0.
$$
\end{theorem}

\begin{proof}
 (i) $\Longleftrightarrow$ (ii). This equivalence is a reformulation of the 
 proved results from Theorems \ref{thm P is s-a} and \ref{thm harmonic}.
 
 (ii) $\Longleftrightarrow$ (iii). By the mean ergodic theorem for contractive
 operators (see, for example, \cite{Yosida1995, Eisner2015}), 
 we obtain that, for any vector $f \in L^2(\nu)$, the sequence of vectors 
 $$
 S_N(f) =\frac{1}{N}  \sum_{n = 1}^{N} P^n(f)
 $$ 
 converges strongly to a vector $\va$ that belongs to the closed subspace 
 of $P$-invariant vectors, i.e., $\va$ must be a harmonic function. the
 converse statement is obviously true. 
 
  (iii) $\Longleftrightarrow$ (iv). We observe that the strong converges of 
  $S_N(f)$ is equivalent to the weak converges. Then, for 
  characteristic functions $\chi_A, \chi_B$, $A, B\in \Bfin$, we find that 
  $$
  \ba
  \langle S_N(\chi_A), \chi_B \rangle_{L^2(\nu)} = & \frac{1}{N}  
  \sum_{n = 1}^{N} \langle \chi_A, P^n(\chi_B) \rangle_{L^2(\nu)}\\
=&   \frac{1}{N}  \sum_{n = 1}^{N} \rho_n(A \times B)\\
 & \longrightarrow 0, \ \qquad N\to \infty. 
  \ea
  $$
We used here Lemma  \ref{lem rho_n via P}. 
  
\end{proof}

In the next proposition, we consider several  properties of harmonic functions 
for a Markov operators $P$ acting on the space of measurable function
$\mc F(V, \B, \nu)$ where $\nu P = \nu$. 
 
\begin{proposition} Suppose that $P : \mc F(V, \B, \nu) \to \mc F(V, \B, 
\nu)$ is a positive operator ($f \geq 0 \Longrightarrow P(f) \geq 0$) such
that $P(\mathbbm 1) = 1$ and $\nu P = \nu$.
 
 (1) If a function $h \in L^2(\nu)$ and $P(h)(x) = h(x)$ $\nu$-a.e., then 
$P(h^2) = h^2$ a.e.

(2) If $h,k \in L^2(\nu)$ and $P(h) = h, P(k) = k$, then
$P(hk) = hk$ $\nu$-a.e.

(3) If $h \in L^2(\nu)$ and $P(h) = h$, then $P(hg) = hP(g)$ for any 
function $g \in \mc F$.

(4) If $h \in L^2(\nu)$ and $P(h) = h$, then $P(h^n) = h^n, n\in \N$.
\end{proposition} 

\begin{proof} (1) Since $P$ is a positive and normalized operator, 
the inequality 
$P(f)^2(x) \leq P(f^2)(x)$ holds for every function $f \in \mc F$ and 
every $x$. By assumption,   $h\in L^2(\nu)$, hence $h^2 \in L^1(\nu)$. 
Next, since $\nu$ is $P$-invariant, we have
$$
\int_V P(h^2)\; d\nu = \int_V h^2\; d\nu.
$$
This means that 
$$
\ba
0 \leq & \int_V (P(h^2) - P(h)^2)\; d\nu \\
= & \int_V P(h^2)\; d\nu - \int_V h^2\; d\nu\\
= &0. 
\ea
$$
Thus, $P(h^2) = h^2$ a.e.

(2) Fix $x \in V$. Because $P$ is positive, we see that 
$$
\langle f, g\rangle_x := P(fg)(x) - P(f)(x) P(g)(x)
$$
is a positive definite bi-linear form. Since  $h$ and $k$ are $P$-invariant 
functions, we obtain, by the Schwarz inequality,  that
$$
(P(hk) - hk)^2 \leq (P(h^2) - h^2)(P(k^2) - k^2),
$$
and this inequality holds for every $x$. Now we can apply (1) and conclude
that $P(hk) = hk$.

(3) This statement can be proved similarly to (2).

(4) We use (1) and (3) to deduce (4).
\end{proof}
 
\subsection{Markov processes} 

It is well known that every Markov operator defines a Markov process on a 
measure space.  We will describe this process explicitly for the operator
$P$ determined by (\ref{eq formula for P}).  

Recall our setting: $(V, \B, \mu)$ is a $\sigma$-finite measure space,  and
$\rho$ is a symmetric measure defined by a positive operator $R$ (see
 Proposition \ref{prop equiv of def of rho}) such that $c(x) = \rho_x(V)$. 
 The measure $\rho$ admits a disintegration (see Section \ref{sect basics})
 such that 
$$
d\rho(x, y) = d\rho_x(y) d\mu(x) = P(x, dy) d \nu(x)
$$ 
 where $x \mapsto d\rho_x$ is a measurable family of positive measures,
 and $P(x, dy)$ is the probability measure obtained by normalization of 
 $\rho_x$.
 
 Fix a point $x\in V$, and define inductively a sequence of probability
  measures $(P_n(x, \cdot) : n \in \N_0)$.  For any set $A \in \B$, we
  define
$$
 \ba
& P_0(x, A) = \chi_A(x), \\
&P_1(x, A) =  \int_V P_0(y, A) \; P(x, dy) \\
   & \cdots\cdots\cdots\cdots\cdots\cdots\cdots\cdots\cdots\cdots\\
 &    P_{n+1}(x, A)  = \int_V P_n(x_n, A)  P(x, dx_n),\\
     & \cdots\cdots\cdots\cdots\cdots\cdots\cdots\cdots\cdots\cdots\\
 \ea
 $$  
 To simplify notation, we  write $P(x, A)$ for $P_1(x, A)$. 
 
For the reader's convenience we formulate two statements in the next 
lemma. 
 
\begin{lemma}\label{lem P_n via P^n}
(1) Let $x \in V$ be a fixed point. For every $n \in \N_0$, the map $\B \ni A
 \mapsto P_n(x, A)$ defines a probability measure on $(V, \B)$. For any
 fixed $A\in \B$, the function $x \mapsto P_n(x, A)$ is $\B$-measurable
 $n \in \N_0$. 
 
 (2) For $A\in \B$, 
 $$
 P_n(x, A) = P^n(\chi_A)(x),\qquad n \in \N_0. 
 $$
\end{lemma} 
 
 \begin{proof}
 Statement (1) follows from the definition of $P_n(x, A)$, and (2) can be
  proved by  induction. 
 \end{proof}
 
 \begin{remark}
 (i) It is useful to interpret $P_n(x, A)$ as the probability to get to a set 
$A\in \B$  for $n$ steps  assuming that the process begins at $x$. In 
particular, $P_0(x, A) =\delta_A(x)$. We call $(P_n(x, A)), n \geq 0,$ a 
sequence of  \textit{transition probabilities}.

(ii) In case when the symmetric measure $\rho$ is defined by 
(\ref{eq def rho}), we have the following formulas for $P_n(x, A)$:
$$
\ba
P_1(x, A) & = \int_V \frac{\chi_A(y)}{c(x)} \; d\rho_x(y),  \\
   & \cdots\cdots\cdots\cdots\cdots\cdots\cdots\cdots\cdots\cdots\\
 P_{n}(x, A) & =     \int_V \cdots \int_V \frac{\chi_A(y)}{c(x_{n-1})
 \cdots c(x)}\;   d\rho_{x_{n-1}}(y) \cdots d\rho_x(x_1)\\
 & \cdots\cdots\cdots\cdots\cdots\cdots\cdots\cdots\cdots\cdots\\
\ea
$$
\end{remark}

We finish this subsection by pointing out a curious relation between discrete
Markov chains and continuous Poisson type distributions. The reader
can find relevant materials in \cite{AlbeverioMaRockner2015, 
Applebaum2009, Kallenberg1983}. The reader can find the theory of 
operator  semigroups  in the remarkable  book \cite{Yosida1995}.

\begin{theorem}\label{thm Poisson}
(1) Let $P$ be a Markov operator on $L^2(V, \B, \nu)$ and let 
$(P_n)$ be a discrete Markov process generated by $P$. 
For every $t \in \R_+$ and $A \in \B$, define 
\be\label{eqdef Q}
Q_t(x, A) := \sum_{n= 0}^\infty e^{-\la t} \frac{(\la t)^n}{n !} P_n(x, A).
\ee
Then, the distribution $Q_t$ satisfies the property: 
$$
\int_V Q_s(y, A) Q_t(x, dy) = Q_{t+s}(x, A),
$$ 
Moreover,  $\nu Q_t= \nu$ if and only if $\nu P = \nu$.

(2) $\{Q_t : t\geq 0\} $ is a strongly continuous semigroup such that 
$Q_0 = I$, and the generator $\mc L$ of $\{Q_t : t\geq 0\} $  is 
$\la(P - I)$.

(3) The following are equivalent:

(i) $f$ is a harmonic function with respect to $P$;

(ii) $\mc L(f) = 0$;

(iii) $Q_t(f) = f$ for $t \geq 0$.

(4) The operators
\be\label{eqDef S_t}
S_t(f)(x) :=  e^{-c(x) t} \sum_{n= 0}^\infty  \frac{(c(x)t)^n}{n !} 
\int_V f(y) \; P_n(x, dy)
\ee
form a self-adjoint contractive semigroup of operators in $L^2(\mu)$. 
The  generator $\mc L$ of $\{S_t\}_{t\geq 0}$ is  $c(P - I)$, i.e., 
$$
S_t = e^{-t \Delta}
$$
where $\Delta$ is considered as an unbounded operator in $L^2(\mu)$.
\end{theorem}

\begin{proof}
In the proof of (1), we use the binomial formula and the relation
$$
P_{n +m}(x, A) = \int_V P_m(y, A)\; P_n(x, dy)
$$
which follows from Lemma \ref{lem P_n via P^n}. Thus, we have
$$
\ba
&\int_V Q_s(y, A) Q_t(x, dy)  \\
& = \int_V \sum_{m= 0}^\infty e^{-\la s} \frac{(\la s)^m}{m !} P_m(x, A)
\sum_{n= 0}^\infty e^{-\la t} \frac{(\la t)^n}{n !} P_n(x, dy)\\
& = e^{-\la (t +s)}  \sum_{m,n =0}^\infty \frac{(\la s)^m}{m !} 
\frac{(\la t)^n}{n !} P_{m+n}(x, A)\\
& = e^{-\la (t +s)}  \sum_{k =0}^\infty \sum_{n =0}^k
\frac{(\la s)^{k-n}k!}{(k-n)!} \frac{(\la t)^n}{n !}\frac{1}{k!} P_{k}(x, A)\\
& = e^{-\la (t +s)}  \sum_{k =0}^\infty \frac{(\la(t +s))^k}{k!}
 P_{k}(x, A)\\
 & = Q_{t+s}(x, A).
\ea
$$
It follows from the definition of $Q_t$ that $\nu$ must be invariant
with respect to $P$ and $Q_t$ simultaneously. 

(2) We first note that 
$$
Q_0(f)(x)  =  P_0(f)(x) = \delta_x(f) = f(x)
$$ 
as follows from (\ref{eqdef Q}). Moreover, $Q_t$ is strongly continuous
because $P$ is a bounded operator. Then the generator $\mc L$
of the semigroup 
$\{Q_t : t\geq 0\} $ can be found by direct computation:
$$
\ba
\frac{dQ_t(f)}{dt} & = \sum_{n=0}^\infty \la  e^{-\la t}\left(- 
\frac{(\la t)^n}{n !} +  \frac{(\la t)^{n-1}}{(n -1)!} \right) P_n(f)\\
& = \la \left(P (f) - f \right).
\ea
$$
Hence, the generator of $\{Q_t : t\geq 0\} $ is
$$
\mc L(f) := \lim_{t \to 0} \frac{Q_t(f) - f}{t} = \la(P - I)(f).
$$

(3) This statement is an immediate consequence of (1) and (2).

(4) To check that $\{S_t\}$ is a semigroup, we use the same calculation
as in (1) applied to (\ref{eqDef S_t}). Statement (2), employed  to the
 semigroup $\{S_t\}$, gives  the exact formula for the generator of this
 semigroup. 

\end{proof}

\begin{remark} 
The same approach can be applied to the study of semigroups of operators, 
defined as in (\ref{eqdef Q}) and (\ref{eqDef S_t}), acting in the 
finite energy space $\h_E$, see Sections \ref{sect energy} and 
\ref{sect Laplace in H}. 
\end{remark}

The following fact follows directly from Theorem \ref{thm Poisson}.

\begin{corollary} Suppose that the measure space $(V, \B, \mu)$ and the
measurable field of probability measures $x \mapsto P_n(x, \cdot)$ be as 
above.  Let $(X, \mc A, m)$ be another measure space. For any sets
$A \in \B, N \in \mc A$ and $x \in V$, define
$$
Q_x(N, A) = \sum_{n=0}^\infty e^{-m(N)} \frac{m(N)^n}{n !} 
P_n(x, A).
$$
Then $x \mapsto Q_x(\cdot, \cdot)$ is a measurable field of probability 
measures on $(X \times V, \mc A \times \B)$.
\end{corollary}

\textbf{Path space and measure.} 
We denote by $\Omega$ the infinite Cartesian product  $ V^{\N_0} = V
 \times V \times \cdots$.  Let $(X_n (\omega): n = 0,1,...)$ be the 
 sequence of random variables  $X_n : \Omega \to V$ such that 
  $X_n(\omega) =  \omega_n$. It is convenient to interpret $\Omega$ as 
the path space of the Markov process $(P_n)$. 
Every $\omega \in \Omega$ represents 
  an infinite path, and if $X_0(\omega) = x$, then we say that the path
  $\omega$ begins at $x$. A subset $\{\omega \in
 \Omega : X_0(\omega) \in A_0, ... X_k(\omega) \in A_k\}$ is called 
 a \textit{cylinder set} defined by $A_0, A_1, ..., A_k$, $k \in \N_0$.  
 The collection of cylinder sets generates the $\sigma$-algebra $\mathcal C$
 of Borel subsets of $\Omega$. 
 
 It follows from this definition of Borel structure on $\Omega$ that the 
 function $X_n : \Omega \to V$ is Borel. This construction allows us to 
 define an increasing  sequence of $\sigma$-subalgebras 
 $\mathcal{F}_{\leq n}$ such that  $\mathcal{F}_{\leq n}$ is the smallest
 subalgebra for which the functions $X_0, X_1, ... , X_n$ are Borel. By
 $\mathcal F_n$, we denote the $\sigma$-subalgebra $X_n^{-1}(\B)$. 
 Since $X_n^{-1}(\B)$ is a $\sigma$-subalgebra of $\mc C$, there exists a 
 projection $E_n : L^2(V, \mc C, \la) \to L^2(\Omega, X_n^{-1}(\B), \la)$, 
 where the measure $\la$ is defined in Section \ref{sect diss space}. 
The projection  $E_n$ is called the conditional expectation with respect to
$X_n^{-1}(\B)$. It satisfies the property:
\be\label{eq cond exp E_n}
E_n(f\circ X_n) = f\circ X_n. 
\ee
 
Next, we define a measure on $\Omega$. 
Let $x \in V$ be a fixed point. Then, let  $\Omega_x$ be the set 
of infinite paths beginning at $x$:
$$
\Omega_x := \{\omega\in \Omega : X_0(\omega) = x\}.
$$
Clearly, $\Omega = \coprod_{x\in V}  \Omega_x$. 

\begin{lemma}\label{lem def of P_x}
 For the objects introduced above, there exists a 
probability measure $\mathbb P_x$ on $\Omega_x$ such that its values on 
 cylinder subsets of $\Omega_x$ are determined by the formula:
\be\label{eq meas P_x} 
 \ba
 \mathbb P_x(X_1 \in A_1, ... , X_n \in A_n) &=
 \int_{A_1}\cdots \int_{A_n} P(y_{n-1}, dy_n) \cdots P(x, dy_1) \\
&= \int_{A_{1}}\cdots \int_{A_{n-1}} P(y_{n-1}, A_n) P(y_{n-2},  dy_{n-1})
\cdots  P(x, dy_1).
\ea
 \ee
\end{lemma}

\begin{proof} (Sketch) Since $\mathbb P_x$ is defined explicitly 
on cylinder sets, the only fact 
one needs to check is that the definition of $\mathbb P_x$ is consistent, i.e.,
$\mathbb P_x$ on a cylinder set of length $m$ is the sum of values
of $\mathbb P_x$ on cylinder subsets of length $m+1$. 
Then the result follows from  the Kolmogorov extension theorem
\cite{Kolmogorov1950} which 
states that there exists a unique probability measure on $\Omega_x$ 
extending $\mathbb P_x$ to the sigma-algebra of Borel sets. 
\end{proof}

As a corollary of Lemmas \ref{lem P_n via P^n} and  \ref{lem def of P_x},
 we have also the following formula:
\be\label{eq meas P_x 2}
 \mathbb P_x(X_1 \in A_1, ... , X_n \in A_n) = 
 P(\chi_{A_1} P(\chi_{A_2}P(\ \cdots\ P(\chi_{A_{n-1}} P(\chi_{A_n})) 
 \cdots )))(x).
\ee

It is useful also to mention a formula for the joint distribution of the random 
variables $X_i$:
\be\label{eqjoint distr}
d\mathbb P_x(X_1, ... , X_n)^{-1} = P(x, dy_1) P(y_1, dy_2) \cdots 
P(y_{n-1}, dy_n).
\ee

\begin{lemma}\label{lem st meas space}
The measure space $(\Omega_x, \mathbb P_x)$ is a standard 
probability measure space for $\mu$-a.e. $x\in V$. 
\end{lemma}

\begin{proof} \textit{(Sketch)}
To see that this property holds, we use the definition of 
the measure $P(x, A)$ and Assumption A. It follows that $P(x, A) >0$ if
and only if $\rho_x(E_x\cap A) >0$. This means that the random variable 
$X_1$ takes values in an uncountable measure space. Clearly, the same 
holds for the other random variables $X_n$.

\end{proof}

\begin{remark} We remark that the symmetric measure $\rho(A \times B) = 
\nu(A\cap B)$ (see Example \ref{ex 1})
does not satisfy the condition of Lemma \ref{lem st meas space} because 
the corresponding Markov process $(P_n)$ is deterministic: 
$$
P(x, A) = \delta_x(A), \ A \in \B. 
$$ 
Hence, the corresponding operator $P$ is the identity operator. 
\end{remark}

\begin{example}[\textbf{Countable Borel equivalence relations}] 
Suppose now that $E$ is a countable Borel equivalence 
relation on $(V, \B, \mu)$, see \ref{ex 1}, part (2) . Then the set $E_x$
of points $y$ equivalent to $x$ is countable, and the transition probabilities
$P(x, y)$ are defined by the relation
$$
P(x, y) = \frac{c_{xy}}{c(x)}, \qquad (x, y) \in E, 
$$
where $c_{xy}$ is a symmetric function on $E$ and $c(x) =
\sum_{y \in E_x} c_{xy}$.

We claim that a path 
$$
\omega = (x, x_1, x_2, ... )\in \Omega_x \ \Longleftrightarrow \ 
(x_i, x_{i+1}) \in E_x, \ i \in \N_0,
$$
where $x_0 = x$. It follows that $\Omega_x = E_x^\infty$ (we assume that 
$c_{xy} > 0, \forall y \in E_x$) , and the 
probability measure $\mathbb P_x$ is defined on cylinder functions as 
follows:
$$
\mathbb P_x(X_1 = y_1, ... , X_n = y_n) =
\prod_{i=1}^n \frac{c_{y_{i-1}y_i}}{c(y_{i-1})}, \ \ \ y_0 = x.
$$

We can see that, in this case,  $(\Omega_x, \mathbb P_x)$ can be 
interpreted as the path space of a stationary Bratteli type diagram. All
paths begin at a fixed point $x$, and the transition probability matrix
$P = (P(x, y)_{y\in E_x})$ is the same for all levels of this diagram.
In \cite{BezuglyiJorgensen_2015}, we considered the Laplace and 
Markov operator for arbitrary Bratteli diagrams.

\end{example}

\subsection{Reversible Markov processes}
At the end of this section we discuss the property of reversibility  
for the Markov process $(P_n)$.

\begin{definition}\label{def reversible MP} Let the objects $(V, \B, \mu)$,
$\rho$, $c(x)$, $P$, and $R$ be as above. 
Suppose that $x \mapsto P(x, \cdot )$ is a measurable family of transition
 probabilities on the space $(V, \B)$ which is defined by a Markov operator 
 $P$. It is said that the corresponding Markov process  is 
 \textit{reversible} if, for any sets $A, B \in \B$, the following relation holds:
\be\label{eq def reversible P}
 \int_B c(x) P(x, A)\; d\mu(x) = \int_A c(x) P(x, B)\; d\mu(x).
\ee
\end{definition}

In the next result,  we formulate several  statements which are equivalent
to reversibility of $P$. 

\begin{proposition}\label{prop TFAE reverse MP}
Let $(V, \B, \mu)$ be a standard measure space, $\rho= (x \mapsto 
\rho_x)$ a measure on $V\times V, \B\times \B)$ such that $c(x) = 
\rho_x(V)$ is finite. Let the measure $\nu = c \mu$ and operators $R, P$ 
be defined as above (see \ref{eq def of R}, \ref{eq formula for P}). 
Suppose that  $P(x, \cdot)$ is the
Markov process defined by the operator  $P$.  
Then the following are equivalent:

(i) $P(x, \cdot)$ is reversible;

(ii) the Markov operator $P$ is  self-adjoint in $L^2(\nu)$ and $\nu P= \nu$;

(iii) 
$$
c(x) P(x, dy) d\mu(x) = c(y) P(y, dx)d\mu(y);
$$

(iv) the measure $\rho$ on $(V\times V, \B\times B)$ defined by 
$$
\rho(A \times B) = \int_V \chi_A R(\chi_B)\; d\mu
$$
is symmetric where $R(f) = cP(f)$;

(v) the operator $R$ is symmetric.

\end{proposition}

\begin{remark} If $P(x, \cdot)$ is reversible, then $P_n(x, \cdot)$ 
 satisfies relation (\ref{eq def reversible P}) for every $n > 1$, i.e., 
 $P_n(x, \cdot)$ is also reversible. This observation immediately follows from 
 the fact that the Markov operator $P$ is self-adjoint, and therefore $P^n$
 is also self-adjoint, see Proposition \ref{prop TFAE reverse MP} (ii).
\end{remark}

\begin{proof} (i) $\Longleftrightarrow $ (ii). We first recall that $P(x, A) =
P(\chi_A)(x)$. Then one can see that, for any sets 
$A, B\in \Bfin$, relation (\ref{eq def reversible P}) is written as the
equality of the inner products:
$$
\ba
\langle \chi_B, P(\chi_A) \rangle_{L^2(\nu)} = & \int_V\chi_B(x)
 P(\chi_A)(x) c(x)\; d\mu(x)\\
  = & \int_V\chi_A(x) P(\chi_B)(x) c(x)\; d\mu(x)\\
= & \langle P(\chi_B),  \chi_A \rangle_{L^2(\nu)}\\
\ea
$$
The proof is completed by extension of the above equality by linearity to
 the functions from the set $\Dfin$ which is dense in $L^2(\nu)$ (we note
 that $c$ is locally integrable with resoect to $\mu$).

(iii) $\Longleftrightarrow $ (iv). This equivalence is obvious because the
equality in (iii) means that the measure $d\rho(x, y) = d\rho_x(y) d\mu(x)$
is symmetric. The fact that this symmetric measure can be represented as
in (iv) is proved in Proposition \ref{prop equiv of def of rho}. 

(ii) $\Longleftrightarrow $ (iv). This results has been proved earlier. It 
follows immediately from Theorems \ref{thm R properties} (9) and 
\ref{thm P is s-a}.

(iv) $\Longleftrightarrow $ (v). This equivalence has been proved in 
Proposition \ref{prop equiv of def of rho}. 

\end{proof}

For more details regarding probability and Markov chains, the reader may
 consult the following items \cite{Kolmogorov1950, Bagett_et_al2010,
  Geiger2017,  Terhesiu2017}  and the papers cited there.

\section{ Dissipation space and stochastic analysis} \label{sect diss space}
We define here a useful Hilbert space which
plays an important role in the study of our finite energy space $\h_E$ in 
Section \ref{sect energy}.

\begin{definition} On the measurable space $(\Omega, \mathcal C)$, define 
a $\sigma$-finite measure $\lambda$ by 
\be 
\lambda := \int_V \mathbb P_x \; d\nu(x)
\ee
($\la$ is infinite iff the measure $\nu$ is infinite).
The Hilbert space 
\be\label{eq L^2(lambda)}
Diss :=  \{ \frac{1}{\sqrt 2} f : f \in L^2(\Omega, \lambda)\},  
\ee
is called the \textit{dissipation space}. 
\end{definition} 

The dissipation  Hilbert space $Diss $ is formed by functions from
 $L^2(\Omega, \la)$ which are rescaled by the factor $1/\sqrt 2$. Then 
$$  
  \| f\|_{\mc D} = \frac{1}{\sqrt 2} \| f\|_{L^2(\la)}.
  $$

Because the partition of  $\Omega$ into $(\Omega_x : x \in V)$ is
 measurable,  we
see that the dissipation space can be  naturally decomposed into the direct
 integral of Hilbert spaces: 
\be\label{eq L^2(lambda) dir int}
L^2(\Omega, \la)  = \int^{\oplus}_V L^2(\Omega_x, \mathbb P_x)\;  d\nu(x)
\ee

Since $x \mapsto \mathbb P_x$ is a measurable field of probability 
measures,  we can use the following formula for integration
of measurable functions $f$ over $(\Omega, \mathcal C)$:
 $$
\lambda(f) = \int_{\Omega} f(\omega) \; d\lambda(\omega) = 
\int_V \mathbb E_x(f)\; d\nu(x)
$$ 
where $\mathbb E_x$ denotes the conditional expectation with respect to 
the measures $\mathbb P_x$,
$$
\mathbb E_x(f) = \int_{\Omega_x} f(\omega)\; d\mathbb P_x(\omega). 
$$

The inner product in the Hilbert space $Diss$ is determined by the formula:
\be\label{eq inner prod in D}
\ba
\langle f, g\rangle_{\mc D} = & \frac{1}{2}\int_V \mathbb E_x(fg) \; 
d\nu(x)\\ 
= &\frac{1}{2} \int_V \int_{\Omega_x}  f(\omega) g(\omega) \; d\mathbb 
P_x(\omega)d\nu(x).
\ea
\ee

From the given definitions of $\mathbb P_x$ and the Markov process 
$(X_n)$, one can  deduce the following formulas. We recall that $E_n$
 denotes the conditional expectation with respect to the
subalgebra $X_n^{-1}(\B)$.

\begin{lemma}\label{lem formulas for P} Let $(\Omega_x, \mathbb P_x)$
be as in Section \ref{sect markov}, and let $P$ be the Markov operator
defined in (\ref{eq formula for P}). Then

(i) $\mathbb P_x\circ X_n^{-1}(A) = P_n(x, A) = P^n(\chi_A)(x),\ A\in \B$;

(ii) $P : L^2(V, P_{n+1}(x, \cdot)) \to L^2(V, P_{n}(x, \cdot))$ is 
contractive for all $n$;

(iii)  $P(f)\circ X_n = \mathbb E[f\circ X_{n+1} \ | \ \mc F_n ]= E_n(f\circ 
X_{n+1})$.

\end{lemma}

\begin{proof} (i) This formula follows from the definition of the measure
$\mathbb P_x$, see (\ref{eq meas P_x}) and (\ref{eq meas P_x 2}).


(ii) Let $f \in L^2(V, P_{n+1}(x, \cdot))$. Then 
$$
\ba
\| P(f) \|^2_{P_n} = &   \int_V P(f)^2(y)\; P_n(x, dy) \\
 \leq & \int_V  P(f^2)(y)\; P_n(x, dy) \\
= & \int_V  f^2(y)\; P_{n+1}(x, dy) \\
= & \| f \|^2_{P_{n+1}}. 
\ea
$$

(iii) In fact, we will prove a slightly more general result: for any Borel 
functions $f, h$, one has
$$
\mathbb E_x[(h\circ X_n)\ (f\circ X_{n+1})] = \mathbb E_x[(h\circ X_n)\ 
(P(f)\circ X_{n})].
$$
Indeed, we use (i) to show that 
$$
\ba
\mathbb E_x[(h\circ X_n)\  (P(f)\circ X_{n})] & = \int_V h(y) P(f)(y)\; 
P_n(x, dy)\\ 
& =  \iint_{V \times V} h(y) f(z)\; P(y, dz)P_n(x, dy)\\ 
& =  \mathbb E_x[(h\circ X_n)\ (f\circ X_{n+1})]. 
\ea
$$
\end{proof}

We now show that $L^2(\nu)$ is isometrically embedded into $Diss$. 
Let $W_n$ be defined by the relation 
\be\label{eq def W_n}
W_n(f) = \sqrt{2} (f\circ X_n), \qquad n \in \N
\ee

\begin{lemma}\label{lem W_n isom}
The operator $W_n$ is an isometry from $L^2(\nu)$ to $Diss$ for any 
$n \in \N$. 
\end{lemma}

\begin{proof}
We compute
$$
\ba
\| W_n(f)\|^2_{Diss} & = \int_V \mathbb E_x( f\circ X_n(\omega)^2)\;
d\nu(x)\\
& = \int_V \int_{\Omega_x}  f( X_n(\omega))^2)\; 
d\mathbb P_x(\omega)d\nu(x)\\
& = \int_V \int_{V}  f(y)^2\; dP_n(x, dy)d\nu(x)\\
& = \int_V \int_{V}  f(y)^2\; d\rho(x, y)\\
& = \int_V \int_{V}  f(x)^2\; d\rho(x, y)\qquad\qquad (\rho\ 
\mathrm{is \ symmetric})\\
& = \int_V   f(x)^2c(x) \; d\mu(x)\\
& = \| f \|^2_{L^2(\nu)}.
\ea
$$
\end{proof}

We will use the isometry $W_n$ defined in (\ref{eq def W_n}) in order  to
 extend the definition of the operator  $P$ to the dissipation space. This 
approach is described in the following steps.

Let $F_n$ be a function of $n+1$ variables, 
$$
F_n : \underbrace{V\times \cdots \times V:}_{n +1\ \text{times}} \to \R,
 n\in \N.
$$ 
Set 
$$
\Phi_n(\omega) := F(X_0(\omega), \dots  X_n(\omega)). 
$$ 
Clearly, the Hilbert space $Diss$ contains a dense subset which is constituted
by functions of the form $\Phi_n, n\in \N$. 

Define abounded  linear operator $S$ acting in $Diss$. It suffices to define
it on functions $\Phi_n$:
\be\label{eqdef S}
(S \Phi_n)(\omega) = \int_V F(X_0(\omega), \dots  X_{n-1}(\omega), y)
\; P(X_{n-1}(\omega), dy)  
\ee

\begin{lemma}\label{lem props of S} The operator $S$ defined by 
(\ref{eqdef S}) is contractive and self-adjoint. 
\end{lemma}

\begin{proof} To see that $S$ is contractive as an operator on $Diss$, 
we use (\ref{eqjoint distr}) in  the following computation:
$$
\ba
& \| (SF)(X_0, ... , X_n) \|^2_{\mc D}\\
 & = \frac{1}{2}\int_V \mathbb E_x
\left( \int_V F(X_0, ... , X_{n-1}, y)\; P(X_{n-1}, dy)\right)^2 \; d\nu(x)\\
& \leq \frac{1}{2} \int_V \mathbb E_x
\left( \int_V F^2(X_0, ...  X_{n-1}, y)\; P(X_{n-1}, dy)\right) \; d\nu(x)\\
& = \frac{1}{2} \iiint  F^2(X_0(\omega), ...  X_{n-1}(\omega), y)\; 
P(X_{n-1}(\omega), dy) d\mathbb P_x(\omega)d\nu(x)\\
 & = \frac{1}{2} \int \cdots \int F^2(x, x_1, ... , y)\; P(x, dx_1) \cdots 
 P(x_{n-2},  dx_{n-1})P(x_{n-1}, dy)d\nu(x)\\
 & = \frac{1}{2} \int_V\int_{\Omega_x}  F^2(X_0(\omega), \dots  
 X_{n}(\omega)) \; d\mathbb P_x(\omega)d\nu(x)\\
 & = \|F \|^2_{\mc D}.
\ea
$$

Using similar argument and the invariance of $\nu$ with respect to $P$, 
we can show that, for any cylinder functions $F, G$, 
$$
\langle SF, G\rangle_{\mc D} = \langle  F, SG\rangle_{\mc D}
$$
i.e., $S$ is self-adjoint in $Diss$.

\end{proof}

\begin{remark} Suppose that $\Phi (\omega) = F\circ X_n(\omega)$ for
some $n$, where $F$ is a function from $L^2(\nu)$. Then we can deduce
 that 
$$
S(F\circ X_n) = P(F) \circ X_n.
$$
This means that the following diagram commutes:
$$
\begin{array}[c]{ccc}
L^2(\nu) &\stackrel{P}{\longrightarrow}& L^2(\nu)\\
\downarrow\scriptstyle{W_n}&&\downarrow\scriptstyle{W_n}\\
Diss &\stackrel{S}{\longrightarrow}& Diss
\end{array}
$$
\end{remark}

In the next two results, we discuss the orthogonality properties in 
the dissipation space $L^2(\Omega, \lambda)$. 

\begin{lemma}[\textbf{Key lemma}]\label{lem key lemma} Let $g_1, g_2$ 
be functions from $L^2(\nu)$. Then 
\be\label{eq key orth}
\langle  g_1 \circ X_n, P(g_2)\circ X_n - g_2\circ X_{n+1} \rangle_{\mc D}
= 0.
\ee
\end{lemma}

\begin{proof} It follows from (\ref{eq inner prod in D}) that the result 
would follow  if we proved that  the functions $ g_1 \circ X_n$ and 
$ P(g_2)\circ X_n - g_2\circ X_{n+1} $ are orthogonal in $L^2(\Omega_x,
\mathbb P_x)$ for a.e. $x$. We use here Lemma \ref{lem formulas for P}
and  (\ref{eq cond exp E_n}) to  compute the inner product:
$$
\ba 
& \langle g_1 \circ X_n,  P(g_2)\circ X_n - g_2\circ X_{n+1}
\rangle_{\mathbb P_x}   \\
= &\mathbb E_x(E_n(g_1 \circ X_n)\ 
(P(g_2)\circ X_n - g_2\circ X_{n+1}))\\
=& \mathbb E_x((g_1 \circ X_n)\  
E_n(P(g_2)\circ X_n - g_2\circ X_{n+1}))\\
=&\mathbb E_x((g_1 \circ X_n)\  
(P(g_2)\circ X_n - E_n(g_2\circ X_{n+1})))\\
=&\mathbb E_x((g_1 \circ X_n)\  
(P(g_2)\circ X_n - P(g_2)\circ X_{n}))\\
= & 0
\ea
$$

\end{proof}

\begin{proposition}\label{prop orth decomp in D} For any function 
$f\in L^2(\nu)$ and any $n \in \N$, 
\be\label{eq orth}
(I - P)(f)\circ X_n \ \perp \ (P(f)\circ X_n - f\circ X_{n+1})
\ee
in the dissipation space $Diss$.
\end{proposition}

\begin{proof} To prove (\ref{eq orth}), it suffices to show that 
\be\label{eq orth 1}
(f\circ X_n) \ \perp \ (P(f)\circ X_n - f\circ X_{n+1})
\ee
and 
\be\label{eq orth 2}
P(f)\circ X_n \ \perp \ (P(f)\circ X_n - f\circ X_{n+1})
\ee
Relation (\ref{eq orth 1}) has been proved in the Key Lemma (Lemma 
\ref{lem key lemma}). It follows from Lemma \ref{lem formulas for P} (iii)
and the proof of Key Lemma that, for $\nu$-a.e. $x\in V$, 
$$
\ba
&\ \mathbb E_x ((P(f)\circ X_n)\ (P(f)\circ X_n - f\circ X_{n+1}))\\ 
= &\ \mathbb E_x ((P(f)\circ X_n)\ E_n(P(f)\circ X_n - f\circ X_{n+1}))\\
= &\ \mathbb E_x ((P(f)\circ X_n)\ (P(f)\circ X_n - E_n(f\circ X_{n+1})))\\
=&\  0. 
\ea
$$
This proves  (\ref{eq orth 2}) and we are done.
\end{proof}

We return to the properties of the dissipation space $Diss$. Denote by 
$\sigma$ the shift in the space $(\Omega, \la)$, i.e., 
$$
\sigma : \omega = (\omega_0, \omega_1, ... ) \mapsto 
(\omega_1, \omega_2, ... ).  
$$
Equivalently, $\sigma$ is determined by the relation $X_n \circ \sigma =
X_{n+1}$. Clearly, $\sigma$ is a measurable endomorphism of $(\Omega,
\la)$, and $\sigma$ acts on the measure $\lambda$ by the formula
$$
\la\circ \sigma^{-1}(\psi) = \int_V \mathbb P_x(\psi\circ \sigma)\;
d\nu(x),\qquad \ \ \ \ \psi \in \mc F(\Omega, \mc C).
$$

\begin{lemma}\label{lem sigma on la} (1) Denote by  $L = L_\sigma$ 
 the operator on  $L^2(\la)$ acting as follows:
 $$
 L(f) = f\circ\sigma, \qquad f \in L^2(\la).
 $$
Then $L$ is an isometry.  

(2) Condition $\nu P = \nu$ implies that $\lambda \circ\sigma^{-1} =
\lambda$. If $d(\nu P)(x) = m(x) d\nu$, then 
$$
\frac{d\lambda \circ\sigma^{-1}}{d\la} = \frac{1}{m \circ X_0}.
$$ 
\end{lemma}

\begin{proof}
We leave the proof to the reader. 
\end{proof}

\begin{remark}
(1) To emphasize the fact that the measure $\lambda$ is defined by 
the operator $P$, we will use also the notation $\lambda_P$.
 We note that $\lambda$ is not, in general, a probability measure, so that 
we cannot use the language of probability theory considering the Markov 
process on $(\Omega, \mc C, \lambda)$. If one took  a probability
 measure on $(V, \B)$ equivalent to $\nu$ (and $\mu$), then $\lambda$
  would be a probability measure.  But it is important to mention that,
  for transient Markov processes, the measure $\lambda_P$ must be 
  infinite.

(2) We use the measure $\nu$ (not $\mu$)  in the definition of $\lambda$
 and in (\ref{eq L^2(lambda)}). The reason for this is based on the fact that  
$\nu$ is invariant with respect to $P$. 
\end{remark}

\begin{proposition}\label{prop from A to B}
Let $A$ and $B$ be any two sets from $\Bfin$. Then 
 $$
 \rho_n(A \times B)  = \lambda (X_0 \in A,  X_n  \in B),\ \  n \in \N.
 $$
 In other words, this equality can be interpreted in the following way: 
for the Markov process $(P_n)$, the ``probability''   to get in $B$ for 
$n$ steps starting somewhere in $A$ is exactly  $\rho_n(A \times B) > 0$. 
\end{proposition}

We recall that the measures $\lambda$  and $\rho$ are, in general,
 not probability. 

\begin{proof} 
It follows from the definition of $\lambda$ and $\rho$ that
 $$
 \ba
\lambda(X_0 \in A,  X_n  \in B) =  & \int_A \mathbb P_x(X_n \in B)\;
d\nu(x)\\
= & \int_A P_n(x,  B)\; d\nu(x)\\
= & \int_A P^n(\chi_B)(x)\; d\nu(x)\\
= & \int_V \chi_A(x) P^n(\chi_B)(x)\; d\nu(x)\\
 = & \rho_n(A \times B)\\
\ea
 $$
\end{proof}

The following result is proved analogously to Proposition 
\ref{prop from A to B}. We leave the details to the reader.

\begin{corollary} The Markov process $(P_n)$ is irreducible, i.e., for any
sets $A, B \in \Bfin$ there exists some $n$ such that
$\langle \chi_A, P^n(\chi_B)\rangle_{L^2(\nu)} > 0$. 

\end{corollary}

\section{Finite energy space: Decompositions and covariance computation}
\label{sect energy}

This section is focused on a measurable analogue of the finite energy space
that has been extensively studied in the special case of weighted networks.
 Before formulating our main definitions and results, we discuss a 
 construction of a ``connected graph'' on the set $\Bfin$. 

\subsection{$\Bfin$ as a connected graph} 
\label{subsect Bfin network}  
Let $(V, \B, \mu)$ be a standard measure space, and let $\rho$ be a 
symmetric measure on $(V\times V, \B\times \B)$. 

\begin{lemma}\label{lem A x A^c finite} Suppose that  $c(x) \in
 L^1_{\mathrm{loc}}(\mu)$. Then, for any set $A\in \Bfin$, 
\be\label{eq finite rho on Bfin}
 \rho(A \times A^c) < \infty
 \ee
 where $A^c = V \setminus A$. The converse is not true, in general. 
\end{lemma}

\begin{proof} The following computation uses the definition of $\rho$ and
local integrability of the function $c$:
$$
\ba 
\rho(A \times A^c) = & \int_V \chi_A(x) R(\chi_{A^c})(x) \; d\mu(x)\\
=& \int_A \left( \int_V \chi_{A^c} \; d\rho_x\right)d\mu(x)\\
= & \int_A \rho_x(A^c)\; d\mu(x)\\
\leq & \int_A c(x)\; d\mu(x)\\
< & \infty.
\ea
$$

The converse statement is false because if $c \in L^1(\mu)$, then 
$\rho(A \times A^c) < \infty$ does not imply that $ A \in \Bfin$. 
\end{proof}

The following definition introduces ``edges'' on the set $\Bfin$. 

\begin{definition} For a symmetric measure $\rho$ on 
$(V \times V, \B\times B)$, we say that two sets $A$ and $B$ from
$\Bfin$ are connected by an edge $e$  if $\rho(A \times B) > 0$. 
Then $\alpha : (A, B) \to \rho(A \times B)$ is a symmetric function defined 
on the set of edges in $\Bfin$.
\end{definition} 

\begin{proposition}\label{prop connectedness} Let $(V, \B, \mu)$ and 
$\rho$ be as above. Then
 any two sets $A$ and $B$ from $\Bfin$ are connected in $\Bfin$ 
 by a finite path. 
\end{proposition}

\begin{proof}
We will show that there exists a finite sequence $(A_i : 0 \leq i \leq n)$ of
 disjoint subsets from $\Bfin$ such
 that $A_0 = A$,  $\rho(A_i \times A_{i+1}) > 0$, and $\rho(A_n \times
B) >0$, $i = 0, ..., n-1$.

 If $\rho(A \times B) > 0$, then nothing to prove, so that we can
assume that $\rho(A \times B) = 0$.

 Let $\xi =(C_i : i \in \N)$ be a partition of $V$ into disjoint subsets of
positive finite measure such that
$C_i \in \Bfin$ for all $i$. Without loss of generality, we can assume that 
the sets $A$ and $B$ are included in $\xi$. Let for definiteness, 
$A = C_0$. 

Since $\rho(A \times A^c) > 0$ (by Lemma \ref{lem A x A^c finite}),
there exists a set $C_{i_1}\in \xi$ such that
$\rho(A \times C_{i_1}) > 0$ and $\rho(A \times C_{j}) = 0$ for
all $0 <j < i_1$. Set
$$
A_1 := \bigcup_{0 < j  \leq i_1} C_j.
$$
It is clear that $A_1 \in \Bfin$ and $\rho(A_0 \times A_1) > 0$.
If $\rho(A_1 \times B) > 0$, then we are done. If not, we proceed as
follows.
Because of the property $\rho(A_1 \times A_1^c)> 0$,
there exists some $i_2 > i_1$ such that $\rho(A_1 \times C_{i_2})
 > 0$ and $\rho(A_1 \times C_{j}) = 0$ for
all $i_1 <j < i_2$. Set
$$
A_2 := \bigcup_{ i_1 \leq j \leq i_2} C_j.
$$
Then $\rho(A_1 \times A_2)> 0$, and we check whether $\rho(A_2 \times
 B) > 0$. If not, we continue in the same manner by constructing
  consequently disjoint sets $A_i$ satisfying the property 
  $\rho(A_i \times A_{i+1}) > 0$.
Since $B$ is an element of $\xi$, this process will terminate. This means
that  there exists some
$n$ such that $A_n \supset B$. This argument proves the proposition.

\end{proof}

It follows from Proposition \ref{prop from A to B} and Proposition 
\ref{prop connectedness} that, for the corresponding Markov process,
 the assumed positive  probability to get  from $A$ to $B$ can be 
 interpreted as the 
 connectedness of $\Bfin$. This property is formulated in the following 
 assertion. 
 
\begin{corollary}\label{cor prob from A to B} Let $P(x, \cdot)$ be the 
Markov process defined in Section \ref{sect markov}. For any
two sets $A$ and $B$ from $\Bfin$ and $x \in A$, 
$$
\mathbb P_x(X_1 \in A_1, ... , X_n \in B) > 0 
$$
if and only if there exists a chain of sets $A= A_0, A_1, ..., A_n =B$ such
that $\rho(A_{i-1} \times A_i) > 0$, $ i=1,..., n.$
\end{corollary}

 \begin{remark} Suppose that $(\xi_n), \xi_n = (A_i^{(n)} : n \in \N)$ 
 is a sequence of countable partitions
 of $(V, \B, \mu)$ that satisfies the properties:
 
 (i) $\xi_{n+1}$ refines $\xi_n$:  every $A_i^{(n)}$ is a 
 $\xi_{n+1}$-set, i.e., it is a union of some elements of $\xi_{n+1}$, 
 
 (ii) the $\sigma$-algebra generated by $\xi_n$-sets is $\B$.
 
 Property (ii) is equivalent to the fact that, for every point $x$, there exists
 a nested sequence $(A_{i_n(x))}^{(n)})$ such that 
 $$
 \bigcap_{n\in \N} A_{i_n(x))}^{(n)} = \{x\}.
 $$
In other words, this means that we assign, for every point $x 
\in V$,  an infinite word over a sequence of countable alphabets determined
by atoms of partitions $\xi_n$. 
 
 Denote by $c^{(n)}_{i, j} = \rho(A_i^{(n)} \times A^{(n)}_j)$ and set
 $$
 c^{(n)}_{i_n(x)} = \sum_{j \sim i_n(x)}   c^{(n)}_{i_n(x), j} 
 $$
 where $j \sim i$ means that $\rho(A_i \times A_j) > 0$.
 
 It can be proved that 
 
\textbf{\textit{Claim}}. (1) $c^{(n)}_{i_n(x)}  < \infty$  for every $i, j$.

(2) $c^{(n)}_{i_n(x)}  \geq c^{(n+1)}_{i_{n+1}(x)}$. 
 \end{remark}

Hence, we can define the function $c(x)$ by setting 
$$
c(x) = \lim_{n \to \infty} c^{(n)}_{i_n(x)} .
$$

\subsection{Definition and properties of the finite energy space $\h_E$}
We consider a class of Borel functions over $(V, \B, \mu)$ which is formed 
by functions of \textit{finite energy}. In other words, this section is
focused on a measurable analogue of the energy Hilbert space
which was extensively studied in the context of discrete networks, see e.g.  
\cite{Cho2014, JorgensenPearse2016, Jorgensen2012, LyonsPeres2016}. 

\begin{definition}\label{def f.e. space} Let $(V, \B, \mu) $ be a standard
 measure space with $\sigma$-finite measure $\mu$. Suppose that
 $\rho$ is a symmetric measure on the Cartesian product $(V\times V,
 \B\times \B)$.  We say that a Borel function
$f : V \to \mathbb R$ belongs to the \textit{finite energy space} 
$\mathcal H_E = \h$ if
\be\label{eq def f from H}
\iint_{V\times V}  (f(x) - f(y))^2 \; d\rho(x, y) < \infty.
\ee
If the measure  $\rho$ is defined in terms of a conductance
 function  $c_{xy}$, then a function $f$ is in $\mathcal H$ when
\be 
 \int_V\left( \int_V c_{xy} (f(x) - f(y))^2 \; d\mu(y) \right)\; d\mu(x) 
 < \infty
\ee
\end{definition}

\begin{remark}
(1) It follows from the Cauchy-Schwarz inequality in the space $L^2(\rho)$ 
that the set $\mathcal H$ is a vector space. It contains 
all constant function $k$. Since for the functions $f$ and $f + k$, the
 quantity in (\ref{eq def f from H}) is the same, we can identify such 
 functions in the space $\h$. That is $\h$ can be treated as the space 
 of classes of  equivalent functions where $f \sim g$ iff $f - g$ is a constant.
 With some abuse of notation we will denote this 
quotient space again by $\mathcal H$. 
We show below that $\mathcal H$ is a Hilbert space.

(2) Definition \ref{def f.e. space} assumes that a symmetric irreducible 
measure $\rho$ is fixed on $(V\times V, \B\times \B)$. This means that
the space of functions $f$ on $(V, \B)$ satisfying (\ref{eq def f from H})
depends on $\rho$ and must be written as $\h_E(\rho)$. Since we do not
 study the dependence of $\h_E(\rho)$ of $\rho$, we will write $\h_E$
 or even $\h$ below.
\end{remark}
 
In other words, we can define a bilinear form $\xi(f, g)$ in the space 
$\mathcal H$ by the formula
\be\label{eq inner product}
\xi(f, g) := \frac{1}{2} \iint_{V \times V}
(f(x) - f(y))(g(x) - g(y)) \; d\rho(x, y).
\ee
We set $\xi(f) = \xi(f, f)$. The domain of $\xi$ is the set of function $f$ 
such that $\xi(f) < \infty$, and, assuming connectedness, the kernel of $\xi$ is $\mathbb R 
\mathbf 1$. Then we see that the space $\h$ defined above coincides with
 $\mathrm{dom}(\xi)/\mathrm{ker}(\xi)$. 

Setting $\langle f, g \rangle_{\mathcal H} = \xi(f,g)$, we define an inner 
product on $\h$. Then 
\be\label{eq norm in H_E}
|| f ||^2_{\mathcal H} := \frac{1}{2} 
\iint_{V\times V}  (f(x) - f(y))^2 \; d\rho(x, y), \qquad f \in \mathcal H,  
\ee
turns $\h$ in a  normed vector space.

\begin{lemma}\label{lem energy sp embedding}
The map 
\be\label{eq d in L2(rho)}
\partial :  f(x) \mapsto F_f(x, y) := \frac{1}{\sqrt 2}(f(x) - f(y)) 
\ee
is an isometric embedding of the space $\h$ into $L^2(\rho)$. 
\end{lemma}

\begin{proof}
This fact follows immediately from Definition \ref{def f.e. space} and
(\ref{eq norm in H_E}).
\end{proof}

\begin{theorem}\label{thm H is H space} 
$(\mathcal H, || \cdot ||_{\mathcal H})$ is a Hilbert space. 
\end{theorem} 

\begin{proof}
We need to check that the normed space $(\mathcal H,
 || \cdot ||_{\mathcal H})$ is complete. 

Suppose $(f_n)$ is a Cauchy sequence in $\h$. Then, by Lemma  
\ref{lem energy sp embedding}, the sequence $(F_n)$, where $F_n(x, y) :=
f_n(x) - f_n(y)$, is Cauchy in $L^2(\rho)$. Hence, there exists a function
$G(x, y)$ such that 
$$
|| F_n - G ||_{L^2(\rho)} \to 0, \quad n \to \infty.
$$ 
There exists a subsequence $(F_{n_k})$ that converges pointwise to $G$ 
for all $(x, y)\notin N$ where $\rho (N) = 0$. 
Let $N^y := \{x : (x, y) \in N\}$. 

Then, for $\mu$-a.e. $y \in V$, we have $\rho_x(N^y) = 0$. In particular, 
this means that there exists some $y_0$ such that 
$$
F_{n_k}(x, y_0) \to G(x, y_0), \qquad \rho_x\mbox{-a.e.} 
$$
Therefore the function $g(x) := G(x, y_0)$ is correctly defined. 

It remains to prove that $G(x, y) = g(x) - g(y)$. Indeed, for $(x, y)\notin N$,
$$
\ba 
G(x, y) & = \lim_{k \to \infty} (f_{n_k}(x) - f_{n_k}(y) )\\
& = \lim_{k \to \infty} (f_{n_k}(x) - f_{n_k}(y_0) ) - \lim_{k \to \infty} 
(f_{n_k}(y) - f_{n_k}(y_0) )\\
& = g(x) - g(y).
\ea
$$
In other words, we have proved that the Cauchy sequence $( f_n)$
 converges to $g(x)$ in $\h$. 
\end{proof}

\begin{theorem}\label{thm norm f is 0} Suppose that $\rho$ is a symmetric
irreducible measure on $(V\times V, \B\times \B)$, and $\h_E = 
\h_E(\rho)$ is the finite energy space. If 
$f\in \h_E$ is a function  such that $|| f||_{\h_E} = 0$, then $f(x)$
is a constant $\mu$-a.e.
\end{theorem}

\begin{proof} We can interpret the set $\Omega_x$ as the set of all paths
 which begin at $x$ and are determined by the Markov process $P(x, \cdot)$.
 Then we immediately deduce from Corollary \ref{cor ||f|| via E_x} that 
the function $f$ with the property $|| f||_\h =0$ is constant along any such
path. 

Suppose, for contrary, that $f(x)$ is not constant on $(V, \B, \mu)$. Then
there exists some $a \in \R$ such that the sets $A := \{f < a\}$ and 
$B:= \{ f > a\}$ both have positive measure $\mu$. Take subsets $A_0
 \subset A$ and $B_0 \subset B$ of finite positive  measure.  By
  connectedness of $\Bfin$, there exists a set  of positive measure $A' 
  \subset A_0$  such that
any path beginning in $A'$ gets in $B_0$ in finitely many iterations. 
We obtain a contradiction. 
\end{proof}

\subsection{Energy space is embedded into dissipation space}
 Let $P$ be a Markov operator and  $x$ is a fixed point in $V$. 
Denote by  $P(x, A)$  the probability measure 
defined by $P$ as in Section \ref{sect markov}. This means 
that 
$$
P(x, f) = \int_V f(y) \; d\ol\rho_x(y) = P(f)(x)
$$
where $f$ is a Borel function. If $X_n(\omega)$ is a corresponding 
sequence of random 
variables on $\Omega_x$, then we have the following formulas for the
conditional expectation $\mathbb E_x$ with respect to the probability
measure $\mathbb P_x$:
\be\label{eq E(X_0)}
\mathbb E_x(f \circ X_0) = \int_{\Omega_x} f(X_0(\omega))\; 
d\mathbb P_x(\omega) = \int_{\Omega_x} f(x) \; 
d\mathbb P_x(\omega) = f(x),
\ee
\be\label{eq E(X_1)}
\mathbb E_x(f \circ X_1) = \int_{\Omega_x} f(X_1(\omega))\; 
d\mathbb P_x(\omega) = \int_V f(y) \; P(x, dy) = P(f)(x)
\ee
where $y = X_1(\omega)$.

\begin{definition} 
Define a linear operator $\partial : \h_E \to Diss$ by the formula:
\be\label{eq operator d}
\partial  : f \mapsto f\circ X_1 - f\circ X_0.
\ee
Similarly, we set
\be\label{eq operator d_n}
\partial_n  : f \mapsto f\circ X_{n+1} - f\circ X_{n}.
\ee
\end{definition}

Remark that we use the same notation $\partial$ as in 
(\ref{eq d in L2(rho)}) of Lemma \ref{lem energy sp embedding} because
these operators are essentially similar.

\begin{lemma}\label{lem d isometry} 
The operator $\partial :\h_E \to Diss $ defined in (\ref{eq operator d}) is an
 isometry.
\end{lemma}

\begin{proof} Let $f \in \h_E$. We use the definition of the norm in 
$Diss$ and in the energy space $\h_E$: 
$$
\ba
\| \partial f\|_{\mc D}^2 = & \frac{1}{2} \int_V \mathbb E_x 
[( f\circ X_0 - f\circ X_1)^2]\; d\nu(x)\\
= & \frac{1}{2} \int_V ( f(x) - f(y))^2 \; P(x, dy)d\nu(x)\\
= & \frac{1}{2} \int_V ( f(x) - f(y))^2 \; d\rho_x(y)d\mu(x)\\
= & \frac{1}{2} \int_V ( f(x) - f(y))^2 \; d\rho(x, y)\\
=& \| f\|^2_{\h_E}.
\ea
$$
\end{proof}

As a corollary, we have the following formula that is used below. 

\begin{corollary}\label{cor ||f|| via E_x}
For $f \in \h$ and $\nu = c\mu$, we have 
$$
|| f ||^2_{\h_E} = \frac{1}{2} \int_V \mathbb E_x[(f\circ X_1 - f\circ X_0)^2]
\; d\nu(x).
$$
\end{corollary}

\begin{proof} See the proof of Lemma \ref{lem d isometry}.
\end{proof}


In the next statements we strengthen the result of Corollary
\ref{cor ||f|| via E_x} using the orthogonal decomposition given in 
Proposition \ref{prop orth decomp in D}. 

\begin{theorem}\label{thm main norm f}
Let $f \in \h_E$. Then 
\be\label{eq norm of f}
\ba
\| f \|^2_{\h_E} =&  \frac{1}{2}\left( \int_V (P(f^2) - P(f)^2)\; d\nu +
\int_V (P(f) - f)^2\; d\nu\right)\\ 
= & \frac{1}{2}\left( \int_V (P(f^2) - P(f)^2)\; d\nu + 
\| f - P(f)\|^2_{L^2(\nu)}\right).
\ea
\ee
In particular, both integrals in the right hand side of (\ref{eq norm of f}) are
finite and non-negative.  Moreover, $Var_x(f\circ X_1) = P(f^2) -  P(f)^2 
\geq 0$ and $Var_x(f\circ X_1) \in L^1(\nu)$, 
for any $f \in \h_E$.  

\end{theorem} 

\begin{proof} By Lemma \ref{lem d isometry}, it suffices to prove that 
the right hand side of (\ref{eq norm of f}) equals $\| \partial f\|^2_{\mc D}$.
Indeed, we can use the orthogonal decomposition given in Proposition 
\ref{prop orth decomp in D} and write
$$
\|\partial f\|^2_{\mc D} = \|f \circ X_0 - P(f) \circ X_0 \|^2_{\mc D} 
+ \| P(f)\circ X_0 - f\circ X_1 \|^2_{\mc D}. 
$$
In the proof below we use the following equality:
$$
\ba 
& Var_x(f\circ X_1)\\
= & \int_V  (P(f)(x) - f(y))^2 \; P(x, dy)\\
  = &  P(f)^2(x) - 2P(f)(x)
\int_V f(y)\; P(x, dy) 
 + \int_V f(y)^2\; P(x, dy)\\
=&   P(f)^2(x) - 2P(f)^2(x) +P(f^2)(x) \\
=& P(f^2)(x) - P(f)^2(x).
\ea
$$ 
Then the computation of $\|\partial f\|^2_{\mc D}$ goes as follows:
$$
\ba
\|\partial f\|^2_{\mc D}  = & \frac{1}{2}\int_V \mathbb{E}_x 
[(I - P)(f)^2\circ X_0]\; d\nu(x)\\
&\ \ + \frac{1}{2}  \int_V  \mathbb{E}_x [ (P(f)
\circ X_0 - f\circ X_1)^2]\; d\nu(x) \\
=& \frac{1}{2}\int_V (f - P(f))^2(x) \; P(x, dy) d\nu(x) \\
& \ \ +  \frac{1}{2} 
\int_V  (P(f)(x) - f(y))^2 \;  P(x, dy)d\nu(x) \\
=& \frac{1}{2}\int_V (f - P(f))^2(x) \; d\nu(x) \\ 
& \ \ +  \frac{1}{2} 
\int_V  (P(f^2)(x) - P(f)^2(x)) \; d\nu(x). \\
\ea
$$
The proof is complete. 
\end{proof}

Theorem \ref{thm main norm f} allows us to deduce a number of important
corollaries. 

\begin{corollary}\label{cor harmonic} (1) If  $f \in \h_E$, then $f - P(f) \in
 L^2(\nu)$ and 
$P(f^2) -  P(f)^2 \in L^1(\nu)$. The operator 
$$
I - P : f \mapsto f - P(f) : \h_E \to L^2(\nu)
$$
is contractive, i.e., $\| I - P\|_{\h_E \to L^2(\nu)} \leq 1$.

(2)
$$
\ba
\| f \|_{\h_E} = 0 \ & \Longleftrightarrow \ \begin{cases} & P(f^2) =  
P(f)^2 \\
& P(f) = f
 \end{cases} \ \ \ \ \qquad \nu-\mbox{a.e.}\\
 & \Longleftrightarrow \ \mathrm{both}\ f \mathrm{and} \ f^2 \ 
 \mathrm{are\ harmonic\ functions}.
 \ea
$$

(3) Let $f \in \h_E$, then 
$$
\ba
f \in \h arm_E \Longleftrightarrow & \| f \|^2_{\h_E} = \frac{1}{2} 
\int_V  (P(f^2)(x) - (Pf)^2(x)) \; d\nu(x)\\
\Longleftrightarrow & \| f \|^2_{\h_E} = \frac{1}{2} \int_V Var_x (f\circ 
X_1)\; d\nu(x).
\ea
$$
\end{corollary} 

\begin{proof} Statement (1) immediately follows from (\ref{eq norm of f}).

To see that (2) holds we use again (\ref{eq norm of f}). The right hand
side is zero if and only if $P(f) = f$  and $P(f^2) = P(f)^2$  a.e. (recall that,
for any function $f$, $P(f^2) \geq P(f)^2$). Since $f$ is harmonic, the latter
means that $f^2$ is harmonic. 

(3) This observation is a consequence of (\ref{eq norm of f}),  Theorem 
\ref{thm main norm f}. 
\end{proof}

We remark that formula (\ref{eq norm of f}) for the norm $\| f \|^2_{\h_E}$
consists of two terms: the deterministic term is 
$ \| f - P(f)\|^2_{L^2(\nu)}$ and the stochastic term is 
$\int_V (P(f^2) - P(f)^2)\; d\nu$. Thus, the norm of a harmonic function
is completely determined by the stochastic term.

It can be shown, using  Theorem \ref{thm main norm f}, that the following 
result holds. We leave its proof for the reader.

\begin{corollary} 
$$
\int_V Var_x(f\circ X_1)\; d\nu(x) =  \int_V Var_x(f\circ X_n)\; d\nu(x), 
\  \ n\in N. 
$$
\end{corollary}

In what follows we will deal with the so called Riesz decomposition of 
functions from the energy space $\h_E$. It is important to note, 
that in this case, we make an additional assumption about the Markov 
process $(P_n)$: it must be \textit{transient}.

\begin{corollary}\label{cor Riesz} Assume that $(P_n)$ is a transient Markov process, i.e.,  
 Green's function 
$$
G(x, A) := \sum_{n\in \N_0} P_n(x, A)
$$
is a.e. finite for every $A \in \B$. Then every function $f  \in \h_E$ has 
a unique 
decomposition (Riesz decomposition) $ f = G(\va) + h$ where $\va \in
 L^2(\nu)$ and $h \in \h arm_E$.

Moreover, for every $f\in \h_E$,
\be\label{eq norm of f via va and h}
\| f\|^2_{\h_E} = \frac{1}{2} \left( \| \va \|^2_{L^2(\nu)} + \int_V 
(P(h^2) - h^2)\; d\nu\right).
\ee
\end{corollary}

\begin{proof} \textit{(Sketch)} Let $f  \in \h_E$, then $ \va = (I - P)(f) $ is 
in $L^2(\nu)$. Define $h = f - G(\va)$. Then 
$$
\ba 
(I - P)(h)(x) = & (I - P)( f - G(\va))(x)\\
=&  \va(x) - G(\va)(x) + PG(\va)(x)\\
=& \va(x)  - \sum_{n\in \N_0} P^n(\va)(x) +\sum_{n\in \N_0} 
P^{n+1} (\va)(x) \\
= & 0.
\ea
$$
Hence $h$ is harmonic and 
$$
 f = G(\va) + h. 
$$

To see that this decomposition is unique, we suppose for contrary that
for some function $f\in \h_E$
$$
f = G(\va) + h =G(\va') + h'.  
$$
Apply $(I - P)$ to both parts and obtain
$$
\va = (I - P)(G(\va) + h) = (I- P)(G(\va') + h') = \va'.  
$$
Therefore, $ h = h'$. 
\end{proof}


\subsection{Structure of the energy space}
We address now the question about the structure of the energy space
$\h$. We first show that  $\h$ contains the linear subspace  
spanned by characteristic functions of sets of finite measure. 
In the next statement we also give two formulas for computation of the norm
of $\chi_A$ and the inner product of characteristic functions in terms of
the measure $\rho$. 

\begin{lemma}\label{lem D is in H}
Suppose $c(x)$ is locally integrable with respect to $\mu$. Then 
$$
\mathcal D_{\mathrm{fin}} \subset \h.
$$
Moreover, if $A \in \Bfin$, then 
$$
|| \chi_A ||^2_{\h_E} \leq \int_A c(x) \; d\mu(x),
$$ 
and 
\be\label{eq ||chi A||}
|| \chi_A ||^2_{\h_E} = \rho(A \times A^c)
\ee
where $A^c := V\setminus A$.

More generally, 
\be\label{eq inner prod via rho}
\langle \chi_A, \chi_B\rangle_{\h_E} = \rho((A \cap B )\times V) -
\rho (A \times B)
\ee
and 
$$
\chi_A \ \perp \ \chi_B \ \Longleftrightarrow\ 
\rho((A\setminus B)\times B) = \rho((A \cap B) \times B^c).
$$
\end{lemma}

\begin{proof} We use (\ref{eq norm in H_E}) and compute for $A\in \Bfin$
$$
\ba
 || \chi_A ||^2_{\h_E} & = \frac{1}{2}\int_{V\times V} (\chi_A(x) -
  \chi_A(y))^2\;  d\rho(x, y) \qquad \qquad (\mbox{by} \ 
  (\ref{eq formula fo symm meas}))\\
 &= \frac{1}{2}\int_{V\times V} (2\chi_A(x) -  2\chi_A(x)\chi_A(y))
 \;  d\rho(x, y)\\
 & = \int_V \int_V \chi_A(x) \; d\rho_x(y)d\mu(x) \\ 
 & \ \ \  - \int_V\int_V
  \chi_A(x) \chi_A(y)\; d\rho_x(y)d\mu(x)\\
  & = \int_A c(x) \; d\mu(x) - \int_A \rho_x(A)\; d\mu(x)\\
& \leq \int_A c(x)\; d\mu(x)\\
& < \infty
\ea
$$
because, by the assumption,  $c \in L^1_{\mathrm{loc}}(\mu)$. 

To see that (\ref{eq ||chi A||}) holds, we use the above relation and represent 
it in the convenient form
$$
\ba
 || \chi_A ||^2_{\h_E}   &= \int_V \int_V \chi_A(x) \; d\rho_x(y)d\mu(x) -
  \int_V\int_V  \chi_A(x) \chi_A(y)\; d\rho_x(y)d\mu(x)\\
& = \int_{V\times V} \chi_{A \times V}(x, y) \; d\rho(x, y) -
\int_{V\times V} \chi_{A \times A}(x, y) \; d\rho(x, y)\\
& = \int_{V\times V} \chi_{A \times A^c}(x, y) \; d\rho(x, y)\\
& = \rho(A \times A^c).\\
  \ea
$$

To finish the proof, we show that (\ref{eq inner prod via rho}) holds. 
The computation is based on the definition, given in (\ref{eq norm in H_E}),
and  the property of symmetry for $\rho$:
$$
\ba
\langle \chi_A, \chi_B\rangle_{\h_E} & = \frac{1}{2}\int_{V\times V} 
(\chi_A(x) - \chi_A(y)) (\chi_B(x) - \chi_B(y))\; d\rho(x, y)\\
& =  \int_{V\times V} (\chi_A(x)\chi_B(x)  -  \chi_A(x)\chi_B(y)) \;  
d\rho(x, y)\\
 & = \int_V \int_V \chi_{(A\cap B)\times V}(x, y) \; d\rho(x, y)
  - \int_V\int_V  \chi_{A \times B}(x, y) \; d\rho(x, y)\\
  &= \rho((A \cap B )\times V) - \rho (A \times B).
\ea
 $$
\end{proof}

\begin{lemma}\label{lem orth to char f-ns} 
Let $g$ be an element from $\h$ such that $\langle \chi_A, g\rangle_{\h_E} =
0$ for every $A \in \Bfin$. Then $g$ is a harmonic function. 
\end{lemma}

\begin{proof}  By condition, the function $g(x)$ is orthogonal to every 
characteristic function $\chi_A$, $A \in \Bfin$. Then 
$$
\ba
0 = & \iint_{V \times V} (\chi_A(x) - \chi_A(y))(g(x) - g(y))\; d\rho(x, y)\\
= & \iint_{V \times V} \chi_A(x) (g(x) - g(y))\; d\rho(x, y) -
\iint_{V \times V} \chi_A(y)(g(x) - g(y))\; d\rho(x, y)\\
= & \iint_{V \times V} \chi_A(x) (g(x) - g(y))\; d\rho(x, y) -
\iint_{V \times V} \chi_A(x)(g(y) - g(x))\; d\rho(x, y)\\
= & 2\int_V\!\int_V \chi_A(x) (g(x) - g(y))\; d\rho_x(y)d \mu(x) \\
= & 2\int_V \chi_A(x) (c(x)g(x) - R(g)(x))\; d \mu(x) \\
=& 2\int_A  c(x)(g(x) - P(g)(x))\; d \mu(x). \\
\ea
$$
Hence, $P(g)(x) = g(x)$ for $\mu$-a.e. $x\in V$, and Lemma is proved.
\end{proof}

It follows from Lemma \ref{lem orth to char f-ns} that any harmonic function 
is orthogonal to the closure $\ol{\mathcal D}_{\mathrm{fin}}$ of the space 
spanned by characteristic functions. Thus, we have the following 
decomposition. 

We denote by $\h arm_E$ the subspace in $\h$ such that $P(f) = f$, i.e.,
$$
\h arm_E := \{ f \in \mathcal F(V, \B) : P(f) = f\} \cap \h.
$$

\begin{corollary} \label{cor Royden} The finite energy space $\h$ admits the decomposition 
into the orthogonal sum 
\be\label{eq Royden}
\h = \ol{\mathcal D}_{\mathrm{fin}} \oplus \h arm_E.
\ee
A function $h\in \h_E$ is harmonic if and only if 
$$
\langle h, f - P(f)\rangle_{\h_E} = 0 \qquad \forall f \in \Dfin.
$$

\end{corollary}
  
Relation (\ref{eq Royden}) is an extension of the \textit{Royden
decomposition.} 

\begin{proposition}
Let $A\in \Bfin$. Then $|| \chi_A ||_{\h_E} = 0$ if and only if $\mu(A) = 0$, 
i.e., $\chi_A = 0$ a.e. 
\end{proposition}



In the following statement, we collect several facts about the properties
of characteristic functions considered as elements of the energy space.

\begin{corollary} Let $\rho$ be a symmetric measure on $(V \times V,
\B \times \B)$, and $A, B \in \Bfin$. Under the assumption that 
$c \in L^1_{\mathrm{loc}}(\mu)$, the following statement hold:

(1) $\rho(A \times B) > 0 \  \Longleftrightarrow \ \mu(\{x \in A : \rho_x(B) 
> 0\}) > 0$;

(2) $\chi_A= 0$ in $\h \ \Longleftrightarrow \ ||\chi_A||_{\h_E} = 0 \ 
 \Longleftrightarrow \ \rho_x(A) = c(x),\ 
\mu\mbox{-a.e.}\ x\in A \  \Longleftrightarrow \ \rho(A\times A) =\int_A 
c(x) \; d\mu(x) =\rho(A \times V)$; 

(2a) in general, not assuming connectedness, 
$\| \chi_A\|_{\h_E} = 0 \ \Longrightarrow \ P(\chi_A g) = \chi_A
P(g), \ \forall g \in \h_E$;

(3) 
$$
\chi_A \ \bot \  \chi_B \ \Longleftrightarrow \ \int_{A\cap B} c(x)\;
 d\mu(x)  = \rho(A \times B) = \int_A \rho_x(B)\; d\mu(x);
$$
 
(4) if $A \subset B$ and $\mu(A) > 0$, then   
$$
\chi_A \bot \chi_B \  \Longleftrightarrow \ \rho(A \times B^c) =0 
\  \Longleftrightarrow \ \rho_x(B^c) = 0 \ \mbox{for\ a.e.} \ x \in A.
$$

(5) if $\chi_A \bot \chi_B $ and $A \cap B = \emptyset$, then 
$\rho(A \times B) = 0$; in general, if $A \cap B = \emptyset$, then 
$\langle \chi_A, \chi_B\rangle_{\h_E} \leq 0$.

\end{corollary}


\begin{proof}  We begin with the obvious formula for the measure of
a rectangle in $V \times V$:
$$
\rho(A \times B) = \int_A \rho_x(B)\; d\mu(x), \qquad \quad A, B \in \Bfin,
$$
where $\rho_x(B) = R(\chi_B).$
This proves (1). 

It was proved in Lemma \ref{lem D is in H} that 
$$
 || \chi_A ||^2_{\h_E} = \int_A (c(x) -\rho_x(A))\; d\mu(x) = \int_A 
 \rho_x(V\setminus A)\; d\mu(x). 
$$
Because of (1), we see that 
$$
 || \chi_A ||_{\h_E} = 0 \ \Longleftrightarrow \ \rho(A\times A) = 
 \rho(A \times V). 
$$
This means that (2) holds.

Since  $\chi_A$ and $\chi_B$ are in $\h$, we can compute their inner 
product as in Lemma \ref{lem D is in H}: 
$$
\langle \chi_A, \chi_B \rangle_{\h_E}  = \iint_{V \times V} (\chi_A(x)
\chi_B(x) \; d\rho(x, y) - \chi_A(x) \chi_B(y) )\; d\rho(x, y)
$$
Therefore, $\chi_A \bot \chi_B $ if and only if 
$$
\iint_{V \times V} (\chi_A(x)\chi_B(x) \; d\rho(x, y) = \iint_{V \times V}
\chi_A(x) \chi_B(y) \; d\rho(x, y). 
$$
The latter is equivalent to
$$
\int_V(\chi_{A \cap B}(x) \left(\int_V \; d\rho_x(y)\right)  d\mu(x)
= \iint_{V \times V} \chi_A(x) \otimes \chi_B(y)  \; d\rho(x, y)
$$
or 
$$
\int_{A\cap B} c(x)\; d\mu(x) = \int_A \rho_x(B)\; d\mu(x),
$$
that proves (3).

To show that (4) holds, we assume that $A \subset B$. Then it follows from 
(3) that  
$$
\ba
\chi_A \bot \chi_B & \  \Longleftrightarrow \ \int_{A} c(x)\; d\mu(x) 
= \int_A \rho_x(B)\; d\mu(x);\\
& \  \Longleftrightarrow \ \rho(A \times V) = \rho(A \times B)\\
& \  \Longleftrightarrow \ \rho(A \times B^c) =0.
\ea
$$

Statement (5) follows from (3) and Lemma \ref{lem D is in H}.
\end{proof}

\begin{remark} As follows from the definition of the energy space, the zero
 element of $\h$ corresponds to any  constant function. The proved 
 properties of $|| \chi_A ||$ means that either the set $A$ or $A^c$ must
 have  zero measure $\mu$. 
\end{remark}

Theorem \ref{thm H is H space} and Lemma \ref{lem D is in H} are used to 
describe the orthogonal complement of $\partial d(\h)$ in $L^2(\rho)$.

Given a function $F(x, y) \in L^2(\rho)$, let $F^{\#}(x, y)$ 
denote the function $F(y, x)$. 

\begin{proposition}\label{prop orth compl}
 The orthogonal compliment $L^2(\rho) \ominus 
\partial(\h)$ consists  of all functions $F(x, y) \in L^2(\rho)$ such that
$\wt R(F)(x) = \wt R(F^{\#})(x)$ for $\mu$-a.e. $x\in V$ where the 
operator $\wt R$ is defined in (\ref{eq R on L}).
\end{proposition}

\begin{proof} Suppose that a function $F(x, y)$ belongs to $L^2(\rho)
 \ominus \partial(\h)$. Then, for any $g(x) \in \h$, we have 
 $$
 \ba
\langle F(x, y), g(x) - g(y)\rangle_{L^2(\rho)} & = 
\int_{V\times V} F(x, y)(g(x) - g(y))\; d\rho(x, y)\\
& = \int_{V\times V} F(x, y)g(x)\; d\rho(x, y) \\
& \ \ \  -
\int_{V\times V} F(x, y)g(y)\; d\rho(x, y)\\
& = \int_{V\times V} F(x, y)g(x)\; d\rho(x, y)\\
& \ \ \  -
\int_{V\times V} F^{\#}(x, y)g(x)\; d\rho(x, y)\\
& =\int_{V\times V} g(x) [F(x, y) - F^{\#}(x, y)]\; d\rho(x, y)\\
& =\int_V g(x) [\wt R(F)(x) - \wt R(F^{\#})(x)]\; d\mu(x)\\
& = 0
\ea
$$
The above relation, in particular, holds for any characteristic function
$g = \chi_A$, $A \in \Bfin$. Hence, $\wt R(F)(x) - \wt R(F^{\#})(x) =0$ 
a.e.

Clearly, the converse implication is also true. 
\end{proof}

\begin{remark} 
In another applications of the kind of energy Hilbert space $\h_E$
 we study is the last term in the Beurling-Dini formula, see 
 \cite[Theorem 3.6.5]{Applebaum2009} and 
 Example  \ref{ex 1} (5); focus on the jump
  term. In our general setting, our $\h_E$ is paired with an $L^2$-space, and
 in the case of the standard Beurling-Dini formula, this $L^2$-space is
  $L^2(\R^d)$. For the literature, see e.g., \cite{Katznelson1968, 
 AlbeverioMaRockner2015, MaRockner1995}.
\end{remark}

\section{Spectral theory for graph Laplacians in $L^2(\mu)$}
\label{sect Delta in L2}


We will use here the notation introduced in the previous sections.
In the next statement, we consider the graph Laplace operator $\Delta$
acting in the Hilbert space $L^2(\mu)$. To emphasize this fact, we will
use also the notation $\Delta_2$. As usual, our basic objects are
a measure space $(V, \B, \mu)$ and a symmetric measure $\rho$ such
that $\rho_x(V) = c(x) \in (0, \infty)$ for $\mu$-a.e. $x\in V$.
\medskip

\textbf{\textit{Assumption E}}: We assume in this section that, for every
set $A \in \Bfin$, the function 
$$
x \mapsto \rho_x(A) = \int_V \chi_A(y)\; d\rho_x(y)
$$  
belongs to $L^1(\mu) \cap L^2(\mu)$.
\medskip

Recall that the subspace $\mathcal D_{\mathrm{fin}}$ is spanned by
 characteristic
functions $\chi_A $ with $\mu(A) < \infty$. Clearly,  
$\mathcal D_{\mathrm{fin}}$ is dense in $L^2(\mu)$. 

We use Assumption to justify the definition of the graph Laplace operator 
$\Delta$ as an unbounded linear operator acting in $L^2(\mu)$. 

\begin{lemma} Let 
$$
\Delta(f)(x) = \int_V (f(x) - f(y))\; d\rho_x(y).
$$ 
Then 
\be\label{eq D in domain}
\mathcal D_{\mathrm{fin}} \subset Dom (\Delta)\cap L^2(\mu)
\ee
and $\Delta$ is a densely defined operator. 
\end{lemma}

\begin{proof}
It is obvious that $\mathcal D_{\mathrm{fin}}$ is a dense subset in 
$L^2(\mu)$. We need to check only that $\Delta(\chi_A) $ is in 
$L^2(\mu)$ if $\mu(A) < \infty$. Since
$$
\Delta (\chi_A)(x) = c(x)\chi_A(x) - \rho_x(A),
$$
we conclude that $\Delta (\chi_A)$ is in $L^2(\mu)$ because of 
Assumption E. 
\end{proof}

Having the densely defined $\Delta$, we can associate the Hilbert adjoint 
operator $\Delta^*$, The domain of $\Delta^*$ is defined by the set of all 
elements $g$  of $\h$ for which the linear functional $f \mapsto 
\langle \Delta f, g\rangle_{L^2(\mu)}$ is continuous. Then there exists 
$g^*\in \h$ such that $\langle \Delta f, g\rangle_{L^2(\mu)} = \langle f, 
g^* \rangle_{L^2(\mu)}$. Set $\Delta^*(g) = g^*$. The operator 
$\Delta^*$ is uniquely defined. 

In fact, we can determine $\Delta^*$ explicitly using the Identity 
$$
\langle \Delta (f), g\rangle_{L^2(\mu)} = \langle f, 
\Delta^*(g) \rangle_{L^2(\mu)}
$$
and formula for $\Delta$:
$$
\ba 
\langle \Delta (f), g\rangle_{L^2(\mu)}  =  &\int_V g(x) \left(\int_V (f(x) - f(y))\; d\rho_x(y) \right) \; d\mu(x)  &\\
= & \iint_{V \times V} (g(x)f(x) - f(y)g(x))\; d\rho(x, y)\\
  = & \iint_{V\times V} g(x)f(x)\; d\rho(x, y)  
 - \iint_{V\times V} f(x)g(y)\; d\rho(x, y)\\
= & \int_V f(x) \left(\int_V (g(x) - g(y))\; d\rho_x(y) \right) \; d\mu(x) \\
 = & \langle f,  \Delta^*(g) \rangle_{L^2(\mu)}.
\ea
$$
Hence $\Delta^*(g) = \int_V (g(x) - g(y))\; d\rho_x(y)$.

Therefore we have proved the following result:

\begin{proposition} The graph Laplace operator $\Delta$ considered in 
$L^2(\mu)$ is symmetric with dense domain $\Dfin$, i.e., 
$$
\langle g, \Delta(f) \rangle_{L^2(\mu)} =  \langle  \Delta(g), f
 \rangle_{L^2(\mu)} \ \ \mbox{on}\ \Dfin.
$$
\end{proposition}

We show below that, in fact, the graph Laplace operator $\Delta$ is
self-adjoint.

\begin{theorem}\label{thm Delta is p.d.}
Let $\rho$ be a symmetric measure on $(V \times V, \B\times \B)$ where
$(V, \B, \mu)$ is
a  measure space. The graph Laplace operator  $\Delta :
L^2(\mu) \to L^2(\mu)$ is positive definite, i.e., it satisfies the following
inequality:
\be\label{eq Delta positive}
2 \int_V f^2 c \; d\mu \geq 
\langle f, \Delta f \rangle_{L^2(\mu)} \geq 0, \quad \forall f \in 
\mathcal D_{\mathrm{fin}}
\ee
\end{theorem}

\begin{proof} We first reformulate (\ref{eq Delta positive}) in more 
convenient terms:
$$
\begin{aligned}
\langle f, \Delta f \rangle_{L^2(\mu)} & = \int_V f \Delta(f)\; d\mu\\
&= \int_V f(x) \left( \int_V (f(x) - f(y))\; d\rho_x(y)\right)\; d\mu(x)\\
& = \int_V f^2(x) c(x) \; d\mu(x) 
- \int_V f(x) \left( \int_V  f(y)\; d\rho_x(y)\right)\; d\mu(x)\\
\end{aligned}
$$
Hence $\langle f, \Delta f \rangle_{L^2(\mu)} \geq 0$ if and only if
\be\label{eq inequality for pos}
\int_V f^2(x) c(x) \; d\mu(x) 
 \geq \int_V f(x) \left( \int_V  f(y)\; d\rho_x(y)\right)\; d\mu(x)
\ee
In order to prove (\ref{eq inequality for pos}), we apply 
the Schwarz inequality:
$$
\begin{aligned}
& \left| \int_V f(x) \left( \int_V  f(y)\; d\rho_x(y)\right)\; d\mu(x) \right|
\quad\qquad \qquad\qquad\ \ \ \ \ 
\mbox{(Schwarz inequality)}\\
& \leq \int_V |f(x)| \left( \int_V  f^2(y)\; d\rho_x(y)\right)^{1/2}
\left( \int_V  \; d\rho_x(y)\right)^{1/2}\; d\mu(x) \\
&  = \int_V |f(x)| \sqrt{c(x)}\left( \int_V  f^2(y)\; d\rho_x(y)\right)^{1/2}
\; d\mu(x) \qquad\ \ \ \mbox{(Schwarz inequality)}\\
& \leq \left(\int_V f^2 c\ d\mu\right)^{1/2} \left(
\int_V \rho_x(f^2) \; d\mu\right)^{1/2} \\
& = \left(\int_V f^2 c\ d\mu\right)^{1/2} \left(\int_V f^2 c\ d\mu
\right)^{1/2} \\
& = \int_V f^2 c\ d\mu.
\end{aligned}
$$
We used here the fact that 
$$
\int_V gc \; d\mu = \int_V R(g)\; d\mu  = \int_V\rho_x(g) \; d\mu.
$$
This proves (\ref{eq inequality for pos}).  Therefore, $\Delta$ is positive
 definite. 

To see that the other inequality in (\ref{eq Delta positive}) holds, we consider
 (\ref{eq inequality for pos}) and write it as 
 $$
 \langle f, \Delta f \rangle_{L^2(\mu)}  \leq 
\int_V f^2(x) c(x) \; d\mu(x)  + \left|
\int_V f(x) \left( \int_V  f(y)\; d\rho_x(y)\right)\; d\mu(x)\right|
 $$
The result then follows from (\ref{eq inequality for pos}).
\end{proof}

\begin{corollary}\label{cor Delta is bdd}
 The operator $\Delta$ acting in $L^2(\mu)$ is bounded
 if and only if   $c \in L^\infty(\mu)$.
\end{corollary}

\begin{proof}
The result follows immediately from the inequality 
$$
2 \int_V f^2 c \; d\mu \geq 
\langle f, \Delta f \rangle_{L^2(\mu)}. 
$$
\end{proof}

\begin{theorem}\label{thm inner product is energy norm} Let $f$ be an 
element of the energy space such that $f $ and $\Delta(f)$ are elements 
of $L^2(\mu)$. Then
\be\label{eq inner product is energy norm}
|| f ||^2_{\mathcal H} = \int_V f \Delta(f) \; d\mu.  
\ee
\end{theorem}

\begin{proof} We first observe that the condition of the theorem holds for 
any function $f \in \mathcal D_{\mathrm{fin}}$. We compute the norm of
 $f$ in $\mathcal H$ by using the 
symmetric property of the measure $\rho$. In other words, the equality 
$$
\int_V f(y)\; d\rho(x, y) =   \int_V f(x)\; d\rho(x, y)
$$
holds for any function $f$. Therefore, we have
$$
\ba
|| f ||^2_{\mathcal H}  & = \frac{1}{2} \iint_{V \times V}
(f(x) - f(y))^2 \; d\rho(x, y) \\
& = \iint_{V\times V} [f^2(x) -  f(x) f(y)] \; d\rho(x,y)\\
& = \int_{V}  \left( \int_V [ f(x)^2   - f(x) f(y)] \; d\rho_x(y)\right) \;
d\mu(x)\\ 
& = \int_V [ f(x)^2 c(x) -  f(x)R(f)(x)] \;d\mu(x)\\
& = \int_V f(x) [c(x) f(x) -  R(f)(x)]  \;d\mu(x)\\
&= \int_V f(x) \Delta(f)(x)  \;d\mu(x).
\ea
$$
Hence, If $f$ and $\Delta(f)$ are $ \in L^2(\mu)$, then we obtain that
$$
|| f ||^2_{\mathcal H}  = \langle f, \Delta f\rangle_{L^2(\mu)}.
$$
\end{proof}

\begin{remark} In (\ref{eq inner product is energy norm}), we can use the
equality 
$$
\int_V f \Delta(f) \; d\mu = \langle f, \Delta(f) \rangle_{L^2(\mu)}
$$
only for those functions $f$ from $\mathcal H$ which are also in 
$L^2(\mu)$. If $f$ is not in $L^2(\mu)$, the integral in 
(\ref{eq inner product is energy norm}) is still well defined. It is worth noting
that, in general, Theorem \ref{thm inner product is energy norm} does not 
hold for arbitrary functions $f$ from $\h$. In the case of discrete networks,
it was shown in \cite{Jorgensen_Pearse2011, JorgensenPearse2013} that 
a certain  discrete Gauss-Green formula contains an
additional term, the so called \textit{boundary term}. 
\end{remark}

We are ready to prove our main result of this section. 

\begin{theorem}\label{thm Delta is s. a.}
The graph Laplace operator $\Delta$ is self-adjoint in the Hilbert space 
$L^2(\mu)$.
\end{theorem}

\begin{proof}
We showed that $\Delta$ is a symmetric operator. In order to proof that it
is self-adjoint, it suffices to show that the deficiency index of $\Delta$
is zero. 

\begin{lemma} $\Delta^* u = -u \ \Longrightarrow \ u =0$.
\end{lemma}

\textit{Proof of the lemma.} Since $\Delta$ is symmetric and 
$\Delta^*u = \Delta u = c(u - P u)$, we show that the equality
\be\label{eq about u}
c(u - P u) = -u
\ee
holds only when $ = 0$. Relation (\ref{eq about u}) is equivalent to
$$
Pu = \left(1 +\frac{1}{c}\right) u.
$$
We use the fact proved in Theorem \ref{thm Delta is p.d.} that $\Delta$
is positive definite:
$$
\ba
\langle u,  \Delta u \rangle_{L^2(\mu)} & = \langle u, c(u - Pu) 
\rangle_{L^2(\mu)} \geq 0\\
\Updownarrow\\
\int_V c u^2\; d\mu & \geq \int_V cu P(u)\; d\mu\\
\Updownarrow\\
\int_V c u^2\; d\mu & \geq \int_V cu (1 + c^{-1})u\; d\mu\\
\Updownarrow\\
\int_V c u^2\; d\mu & \geq \int_V c u^2\; d\mu  + \int_V u^2\; d\mu\\
\ea
$$
Hence $u= 0$ in $L^2(\mu)$, and this completes the proof of the theorem.
\end{proof}

\begin{corollary}\label{cor inner prod}
Let $f, g\in \h_E$ be two functions such that $f, g$ and $\Delta f, \Delta g$
belong to $L^2(\mu)$. Then
\be\label{eq inner prod H and L2}
\langle f, g\rangle_{\h_E} = \langle f, \Delta g \rangle_{L^2(\mu)}.
\ee
\end{corollary}

\begin{proof} Relation (\ref{eq inner prod H and L2}) immediately follows
from two facts: formula (\ref{eq inner product}) of Theorem 
\ref{thm inner product is energy norm}, applied to $|| f + g ||_{\h_E}$, and 
the self-adjointness of the operator $\Delta$ in $L^2(\mu)$,   
Theorem \ref{thm Delta is s. a.}.
\end{proof}

 For more details regarding the potential theory and finite energy space 
 (Dirichlet space), the reader may consult the following items
  \cite{Kaneko2006, Kaneko2014, KawabiRockner2007}
  and the papers cited there.

\section{Spectral theory of the graph Laplacian in the energy space}
\label{sect Laplace in H}

In this section, we consider the graph Laplace operator $\Delta$ acting in 
the energy  space $\h = \h_E$. We will also discuss the properties of
this  operator $\Delta$. 

Our approach is based on the notion of symmetric pairs of operators. We
briefly describe this approach. For more details regarding the theory of
 unbounded operators, readers may
 consult the following items \cite{DunfordSchwartz1988,
 JorgensenTian2017} and the papers cited there.

Let $\h_1$ and $\h_2$ be Hilbert spaces, and let $\mc D_1 \subset \h_1$ 
and $\mc D_2 \subset \h_2$ be dense subspaces. Suppose that two linear
operators
\be\label{eq J and K}
J : \mc D_1 \to \h_2, \ \ \ K  : \mc D_2 \to \h_1
\ee
are defined on these dense subspaces. The pair $(J, K)$ is called a 
\textit{symmetric pair} if
\be\label{eq symm pair}
\langle J \varphi, \psi \rangle_{\h_2}  = 
\langle  \varphi, K\psi \rangle_{\h_1}, \ \ \varphi \in \mc D_1, \psi \in 
\mc D_2.
\ee

The following statement is a well known result in the theory of unbounded
operators. 

\begin{lemma}
(1) Suppose $(J, K)$ be a symmetric pair satisfying (\ref{eq J and K}) and 
(\ref{eq symm pair}). Then the operators $J$ and $K$ are closable and 
$J \subset K^*$, $K \subset J^*$.  Without loss of generality, one can 
assume that $J = \ol J, K = \ol K$. 

(2) $J^*J$ is a self-adjoint densely defined operator 
in $\h_1$, and $K^*K$ is a self-adjoint densely defined operator in $\h_2$. 
\end{lemma}

Now we apply the above statement to the case of Hilbert spaces $L^2(\mu)$
and $\h_E$. To distinguish  the graph Laplace operators acting in 
$L^2(\mu)$ and $\h_E$, we will use the notation $\Delta_2$ and 
$\Delta_{\h}$, respectively. 

As was proved in Theorems \ref{thm Delta is p.d.} and 
\ref{thm Delta is s. a.}, the operator $\Delta_2$ is positive definite and
 essentially self-adjoint; therefore, by the spectral theorem, there exists 
a projection-valued measure $Q(dt)$ such that
$$
\Delta_2 = \int_0^\infty t \; dQ(t)
$$
or, for any $\va \in L^2(\mu)$,
\be\label{eq spectral thm for Delta}
\langle \varphi, \Delta_2\va \rangle_{L^2(\mu)} = \int_0^\infty t \;
|| Q(dt) \va||^2_{L^2(\mu)} 
\ee
(we used here the fact that $Q(dt)$ is a projection). 

\begin{lemma}\label{lem dense D_Q}
In the above notation, let
$$
H_{n,m} = Q([n^{-1}, m]) L^2(\mu), \ \ \ n, m \in \N.
$$ 
Then 
$$
\mc D_Q := \bigcup_{n,m} H_{n, m}
$$
is a dense subspace in $L^2(\mu)$ which is also invariant with respect to 
$\Delta_2$ and  $\Delta_2^{-1}$.  

Moreover,  $\mc D_Q $ can be viewed also as a subspace of $\h_E$.
\end{lemma}

\begin{proof}
The density of $\mc D_Q $  follows directly from the spectral theorem since
the double-indexed sequence of projections 
$\{Q([n^{-1}, m])\}$ strongly converges to the identity operator $I$ in 
$L^2(\mu)$ as $n,m \to \infty$. The invariance of $\mc D_Q$ 
with respect to $\Delta_2$ and $\Delta_2^{-1}$ is deduced from 
the boundness of $\Delta_2$ and $\Delta_2^{-1}$ on every set $H_{n,m}$.
We see that
$$
n^{-1} ||\va ||_{L^2(\mu)} \leq || \Delta_2 \va ||_{L^2(\mu)} \leq m 
||\va ||_{L^2(\mu)}, \ \  \va \in H_{n,m}.
$$
Similarly, for $\va \in H_{n,m}$, we have
$$
m^{-1} ||\va ||_{L^2(\mu)} \leq || \Delta_2^{-1} \va ||_{L^2(\mu)} \leq n 
||\va ||_{L^2(\mu)},
$$
because 
$$
\Delta_2^{-1} = \int_0^\infty t^{-1} \; dQ(t).
$$
Hence, if $\va$ is in $\mc D_Q$, then $\Delta_2 \va \in \mc D_Q$ and
$\Delta_2^{-1} \va \in \mc D_Q$. 

The proof of the second assertion of the lemma follows from relation 
(\ref{eq inner product is energy norm}). We have 
\be\label{eq norm in H for DQ}
|| f ||_{\h_E}^2 = \langle f, \Delta_2 f\rangle_{L^2(\mu)},
\ee
and this holds for any function $f \in \mc D_Q \subset L^2(\mu)$ 
(note that then $\Delta_2 f$ is in $\mc D_Q$). It follows from
(\ref{eq norm in H for DQ}) that, for $f \in H_{n,m}$,
$$
|| f ||_{\h_E}^2 \leq m || f ||_{L^2(\mu)}.
$$
\end{proof}

\begin{lemma}\label{lem Dfin is in D_Q} Let $Harm$ be the set of harmonic 
functions in $\h_E$. Then the space
$$
 \mc C:= \mc D_Q +  Harm
$$
is dense in $\h_E$.
\end{lemma}

\begin{proof} This result follows from the inclusion $\Dfin \subset 
\mc D_Q$ and the decomposition (\ref{eq Royden}).

\end{proof}

We now define  two operators, $J$ and $K$, that constitute a symmetric 
pair. Based on Lemmas \ref{eq spectral thm for Delta} and \ref{lem dense 
D_Q}, we can define the densely defined operator $J$:
\be\label{eq def J}
L^2(\mu) \supset \mc D_Q \ni \va \ \stackrel{J}\longrightarrow \va 
\in \h_E.
\ee
To define its counterpart, the operator $K$, we use Lemma \ref{lem Dfin is in 
D_Q} and put
\be\label{eq def K}
K h = 0, \ h \in Harm,  \quad  K \psi = \Delta_2 \psi, \ \psi \in D_Q,
\ee
where $Harm$ is the set of harmonic functions in $\h_E$. 
Then $K$ is a densely defined operator on the subspace $\mc C$ of $\h_E$
such that $K(\mc C) \subset L^2(\mu)$. 

\begin{lemma}\label{lem J K symm pair} The operators $J$ and $K$,
defined by (\ref{eq def J}) and (\ref{eq def K}), form a symmetric pair, i.e.,
\be\label{eq symm pair J and K}
\langle J \varphi, \psi \rangle_{\h_E}  = 
\langle  \varphi, K\psi \rangle_{L^2(\mu)}, \ \ \varphi \in \mc D_Q, \psi \in 
\mc C.
\ee
\end{lemma}

\begin{proof} We first observe that, by Lemma \ref{lem Dfin is in D_Q},
 every function $\psi \in \mc D_Q$ can be 
represented as $\psi = \Delta^{-1}_2 \xi$. Then  we use Corollary 
\ref{cor inner prod} in the proof. By definition of $K$, we obtain that,
for $\va, \psi \in \mc D_Q$, 
$$
\ba
\langle J\va, \psi\rangle_{\h_E} = & \langle J\va, \Delta^{-1}_2 \xi
\rangle_{\h_E}\\
= & \langle \va, \Delta_2(\Delta^{-1}_2 \xi) \rangle_{L^2(\mu)}\\\
= &  \langle \va, K\psi \rangle_{L^2(\mu)}
\ea
$$
If $\psi$ is in $Harm$, then the left- and right-hand sides in  
(\ref{eq symm pair J and K}) are simultaneously equal to zero. 
\end{proof}

It follows from Lemma \ref{lem J K symm pair} that: 

(1) $J^* = K$ and $K^* = J$, 

(2) the operators $J^*J$ and $K^*K$ are self-adjoint in $L^2(\mu)$ and 
 $\h_E$, respectively.

We combine the results of the lemmas proved in this section 
in the following theorem.

Let $\Delta$ be a linear operator acting on Borel functions $f \in \mc F(X,
\B)$ by 
$$
\Delta(f)(x) = \int_V (f(x) - f(y)) \; d\rho_x(y)
$$
where $\rho = \int_V \rho_x \; d\mu(x)$ is a symmetric measure.

\begin{theorem}\label{thm Delta_h} The Laplace operator $\Delta$ admits
 its realizations in the Hilbert spaces $L^2(\mu)$ and $\h_E$ such that:

(i) $\Delta_2 = J^*J$ is a positive definite essentially self-adjoint operator;

(ii)  $\Delta_{\h}$ is a positive definite and symmetric operator which is
not self-adjoint, in general;  a self-adjoint  extension of   $\Delta_{\h}$ 
is given by  the opearor $JJ^* = K^*K$.

\end{theorem}

\begin{remark} (1) The operator $\Delta_{\h}$ has a self-adjoint extension
 $JJ^*$ but its
deficiency indices might be non-zero. Corresponding examples can 
be found in the discrete theory of Laplace operators (see 
\cite{JorgensenPearse2016}).

(2)  We recall that $\Dfin$ is a natural dense subset in the Hilbert space
$L^2(\mu)$. Moreover, functions from $\Dfin$ belong to the energy
space $\h_E$. Then we could define the operator $J$ by putting
\be\label{eqJ}
\chi_A \stackrel J \longrightarrow \chi_A : L_2(\mu) \to \h_E.
\ee
But for the definition of the operator $K$ we do need the dense subset
$\mc C$ in the energy space $\h_E$. Then the conclusion of Theorem 
\ref{thm Delta_h} can then be obtained from the pair $J, K$.

(3) The definition of the operator $\Delta_E$ is based on the 
construction of a symmetric pair of operators.  A similar technique can be
 used to define an analogue of the Markov operator $P$ acting in the space 
 $\h_E$. We briefly outline this approach. 
 
 Let the operators $\partial : \h_E \to Diss$ and $S : Diss \to Diss$ be
 defined by (\ref{eq operator d}) and (\ref{eqdef S}), respectively.
 
\textbf{ Claim}. The operator $P : \h_E \to \h_E$ is defined by the
formula
$$
P = \partial^* S \partial.
$$
In other words,  the following diagram is commutative.
$$
\begin{array}[c]{ccc}
\h_E &\stackrel{P}{\longrightarrow}& \h_E \\
\downarrow\scriptstyle{\partial}&&\uparrow\scriptstyle{\partial^*}\\
Diss &\stackrel{S}{\longrightarrow}& Diss
\end{array}
$$

We take now $J$ as in (\ref{eqJ}) and define $K : \h_E \to L^2(\nu)$ 
by the formula (see Corollary \ref{cor Riesz})
\be\label{eqK}
K : f = G(\va) + h \mapsto \va. 
\ee
Then one can check that $J$ and $K$ form a symmetric pair, i.e., $K = J^*$.
Using this pair we can define an operator $P : \h_E \to \h_E$, an 
analogue of the Markov operator on $L^2(\nu)$.
\end{remark}

\textbf{Acknowledgments.} The authors are pleased to thank colleagues and
 collaborators, especially members of the seminars in Mathematical Physics 
 and Operator Theory at the University of Iowa, where versions of this work 
 have been presented. We acknowledge very helpful conversations with 
 among others Professors Paul Muhly, Wayne Polyzou; and conversations at 
 distance with Professors Daniel Alpay, and his colleagues at both Ben Gurion 
 University, and Chapman University. The second named author presented an 
 early version of our paper at a 2017-Chapman University conference, Signal 
 Processing and Linear Systems: New Problems and Directions.

\bibliographystyle{alpha}
\bibliography{references1}

\end{document}